\documentclass[reqno, 12pt]{amsart}
\title[Limits of Besov Norms on Quantum Euclidean Spaces]{Limits of Besov Norms on Quantum Euclidean Spaces}
\author{Qingquan Deng, Guixiang Hong, Ping Li and Yanzhen Long\textsuperscript{\dag}}

\address{Qingquan Deng, School of Mathematics and Statistics, and  Key Laboratory of Nonlinear Analysis \&  Applications (Ministry of Education), Central China Normal University, Wuhan 430079, P.R. China}
\email{dengq@ccnu.edu.cn}

\address{Guixiang Hong, Institute for Advanced Study in Mathematics, Harbin Institute of Technology,
Harbin 150001, China}
\email{gxhong@hit.edu.cn}

\address{Ping Li, School of Mathematics and Systems Sciences, and Center for Mathematical Sciences, Wuhan University of Science and Technology, Wuhan 430065, China}
\email{liping@whu.edu.cn}

\address{Yanzhen Long, School of Mathematics and Statistics, and Hubei Key Laboratory of Mathematical Sciences, Central China Normal University, Wuhan 430079, P.R. China}
\email{longyanzhen@mails.ccnu.edu.cn}

\thanks{\textsuperscript{\dag} Corresponding author}
\date{\today}
\keywords{Quantum Euclidean space, Besov space, modulus of smoothness, Bourgain--Br\'{e}zis--Mironescu
formula, Maz'ya--Shaposhnikova formula, real interpolation}
\usepackage{color}
\usepackage{amsmath}
\usepackage{amssymb}
\usepackage[mathscr]{euscript}
\usepackage{paralist}
\usepackage{amsthm}
\usepackage{enumitem}
\usepackage{ulem}
\usepackage{diagbox}
\usepackage{multirow}
\usepackage{makecell}
\usepackage{threeparttable}
\usepackage{footnote}

\usepackage[colorlinks,linkcolor=red,anchorcolor=red,citecolor=blue,hypertexnames=false]{hyperref}

\theoremstyle{plain}
\newtheorem{theorem}{Theorem}[section]
\newtheorem{lemma}[theorem]{Lemma}
\newtheorem{proposition}[theorem]{Proposition}
\newtheorem{corollary}[theorem]{Corollary}

\theoremstyle{definition}
\newtheorem{definition}[theorem]{Definition}

\newtheoremstyle{boldremark}%
  {}{}%
  {\normalfont}{}%
  {\bfseries}{.}%
  { }{}%
\theoremstyle{boldremark}
\newtheorem{remark}[theorem]{Remark}

\newcommand\R{\mathbb{R}}
\newcommand\Z{\mathbb{Z}}

\newcommand\N{\mathbb{N}}
\newcommand\x{x}
\newcommand\y{y}
\newcommand\z{z}
\newcommand\mcm{\mathcal{M}}
\newcommand\mch{\mathcal{H}}
\newcommand\mcb{\mathcal{B}}
\newcommand\mcs{\mathcal{S}}

\newcommand\mcc{\mathcal{C}}

\newcommand\mcf{\mathcal{F}}

\numberwithin{equation}{section}

\begin{document}

\begin{abstract}\baselineskip=13pt
We extend the Bourgain--Br\'{e}zis--Mironescu (BBM) and Maz'ya--Shaposhnikova (MS) limit formulae to higher-order Besov norms on quantum Euclidean spaces $\R^d_\theta$, a fundamental class of noncompact noncommutative manifolds obtained by deformation quantization of $\R^d$. 
A central ingredient is a Banach realization of
the homogeneous Sobolev endpoint that forms a compatible couple with
\(L_p(\R_\theta^d)\), together with the equivalence
\[
K\bigl(\x,\varepsilon^k;
L_p(\R_\theta^d),\dot W_{p,0}^k(\R_\theta^d)\bigr)
\approx
\omega_p^k(\x,\varepsilon).
\]
This framework yields a unified proof of both endpoint formulae
through the endpoint continuity of real interpolation scales. The principal new ingredient at the MS
endpoint is the sharp large-translation asymptotic
\[
\lim_{|h|\to\infty}\|T_h\x-\x\|_p
=
2^{1/p}\|\x\|_p,
\qquad
\x\in L_p(\R_\theta^d),\quad 1\leq p<\infty,
\]
whose proof combines oscillatory decay of the non-pure terms in the
tracial expansion with weighted spectral-measure arguments adapted
to the semifinite trace.
\end{abstract}
\maketitle

\tableofcontents

\baselineskip=15pt
\section{Introduction}
\subsection{Background and main results}
Let \(d\geq1\), \(1\leq p<\infty\), and \(0<s<1\).  The \textit{homogeneous fractional Sobolev space}
\(\dot W_p^s(\mathbb R^d)\), introduced by Gagliardo \cite{Gag58} and
Slobodeckij \cite{Slo58}, consists of all measurable functions \(f\)
on \(\R^d\) for which
\begin{equation}\label{eq:gagliardo-seminorm}
\|f\|_{\dot W_p^s(\R^d)}
:=
\left(
\int_{\R^d}\int_{\R^d}
\frac{|f(x)-f(y)|^p}{|x-y|^{d+sp}}
\,\mathrm dy\,\mathrm dx
\right)^{1/p}
<\infty.
\end{equation} 
A well-known ``deficiency'' of this scale of seminorms is that \(\|f\|_{\dot W_p^s(\mathbb R^d)}\) does not converge to the homogeneous Sobolev seminorm
\(
\|f\|_{\dot W_p^1(\mathbb R^d)}
:=
\|\nabla f\|_{L^p(\mathbb R^d)}
\)
when \(s\to1^{-}\), and one does not recover the norm \(\|f\|_{L_p(\mathbb R^d)}\) as $s\downarrow0$. Indeed, for every nonzero $f\in \mcc_c^\infty(\mathbb R^d)$, one has
\(\|f\|_{\dot W_p^s(\mathbb R^d)}\to\infty\) both as \(s\uparrow 1\) (see \cite{BBM2000,BBM01,HB02}) and \(s \downarrow 0\) (see \cite{GY2021,MS02}).

At the upper endpoint, Bourgain, Br\'ezis, and Mironescu proved in
their fundamental work \cite{BBM01} that, on smooth bounded domains
\(\Omega\subset\R^d\), the normalized seminorm
\((1-s)^{1/p}\|f\|_{\dot W_p^s(\Omega)}\) recovers, in the limit as
\(s\uparrow1\), a constant multiple of
\(\|\nabla f\|_{L_p(\Omega)}\).
Br\'ezis \cite{HB02} subsequently extended the result to the entire \(\R^d\). More precisely, for $f\in W_p^1(\R^d)$ and $1\leq p<\infty$, it holds
\begin{equation}\label{eq:BBM-limit}\tag{BBM}
\lim_{s\uparrow 1}(1-s)^{\frac{1}{p}}\,\|f\|_{\dot W_p^s(\mathbb R^d)}
=C_{d,p}\,\|\nabla f\|_{L_p(\mathbb R^d)},
\end{equation}
where \(C_{d,p}>0\) is an explicit constant depending only on \(d\) and \(p\).

Later, Maz'ya--Shaposhnikova \cite{MS02} complemented the study by considering the opposite endpoint $s\downarrow 0$ and introducing a different normalization. They proved that if $f \in L_p(\R^d) \cap \bigcup_{0<s<1}W_{p,0}^{s}(\R^d)$, $1\le p<\infty$, one has
\begin{equation}\label{eq:MS-limit-norm}\tag{MS}
\lim_{s\downarrow 0^+} s^{\frac{1}{p}}\,\|f\|_{\dot W_p^s(\R^d)}
= c_{d,p}\,\|f\|_{L^{p}(\R^d)},
\end{equation}
where \(c_{d,p}>0\) is independent of \(f\). Here \(W_{p,0}^{s}(\R^d)\) denotes the completion of
\(\mcc_c^\infty(\R^d)\) with respect to the Gagliardo seminorm
\eqref{eq:gagliardo-seminorm}.

Taken together, the limit relations \eqref{eq:BBM-limit} and \eqref{eq:MS-limit-norm} quantify the blow-up rates of the fractional Sobolev seminorm 
at the two endpoints of the Sobolev scale $\{W_{p,0}^s\}_{0<s<1}$. 
Both formulae have given rise to an extensive literature. For further developments, applications, 
and extensions of the BBM formula to broader function-space settings, see,
for example,
\cite{ACPS2020,BBM2002,BSY2023,HB16,DGPYYZ2022,FFG2023,
KIJ2025,MK2023,Lu2014,PSV2019,ZYY2024}
and the references therein. Further studies and variants of the MS
formula can be found in
\cite{ACPS2020b,CDKANP2023,DLY2024,HPXZ2024,KIJ2025,
MSV2022,NTYYZ2025,PYYZ2024,PSV2017}
and the references therein.

In the present paper, we investigate the BBM- and MS-type endpoint
behavior of higher-order Besov norms on quantum Euclidean spaces
\(\R_\theta^d\).  As significant objects in noncommutative geometry, 
quantum Euclidean spaces serve as model examples of noncompact noncommutative manifolds \cite{CGRS2014,GGBI04}
and may be viewed as locally compact counterparts of quantum tori.  
These spaces, whose origins lie in the Heisenberg commutation relations and the Moyal
product in quantum mechanics, are also known as Moyal spaces and
arise as noncommutative deformations of classical Euclidean spaces;
see \cite{Groe1946,Moyal1949,Rie1993}. 
A precise definition of \(\R_\theta^d\), together with a detailed
discussion, is given in Subsection~\ref{section2.2}.
Building on foundational developments in noncommutative martingale theory and harmonic analysis \cite{Junge2002,JX2007,Mei2007,PX1997,PX03}, as well as on the
corresponding theory for quantum tori \cite{CXY2013,XXY18},
substantial progress has been made in harmonic analysis on
\(\R_\theta^d\).
In particular, noncommutative analogues of pseudodifferential operators, singular integral theory, Calder\'{o}n--Zygmund theory, Fourier restriction and multiplier
estimates, and function spaces have been developed on \(\R_\theta^d\), with applications to partial
differential equations; see \cite{CH2026,CHWW2025,FHW2025,GaJM2022,GJP21,HLW23,JMPX2021,LSZ20,ME24,MSX20,RST2026}.

Our work is motivated by two complementary developments. Milman \cite{MM05} interpreted the BBM and MS limits in terms of the
continuity of real interpolation scales, and this perspective was
subsequently developed for higher-order Besov spaces in
\cite{KMX05} and for fractional Sobolev spaces in \cite{DM2023}.  In the
noncommutative setting, Xiong, Xu, and Yin \cite{XXY18} developed a
systematic theory of Sobolev and Besov spaces on quantum tori \(\mathbb T_\theta^d\)  and
derived the corresponding BBM- and MS-type limit formulae by means of
real interpolation. 
As illustrated by the contrasting analyses in
\cite{MSX2019,MSX20}, the passage from the periodic quantum-torus
setting to the noncompact quantum Euclidean setting is far from
formal and may require genuinely new arguments. In the present work, the specific difficulties that arise in extending the
BBM- and MS-type formulae from \(\R^d\) and
\(\mathbb T_\theta^d\) to \(\R_\theta^d\), together with the
strategies used to overcome them, are discussed in
Subsection~\ref{subsec:methods-difficulties}.

To formulate our main results, let
\(B_{p,q}^{\alpha}(\R_\theta^d)\) denote the inhomogeneous Besov
space defined by means of a Littlewood--Paley decomposition (see
Definition~\ref{def:Besov}). As in the classical Euclidean setting,
for each fixed \(\alpha>0\) and every \(k\in\N\) satisfying \(k>\alpha\), the space
\(B_{p,q}^{\alpha}(\R_\theta^d)\) admits an equivalent characterization in terms of
differences (see Lemma~\ref{lem:Besov-difference-characterization}):
\[
\|\x\|_{B_{p,q}^{\alpha}(\R_\theta^d)}
\approx
\|\x\|_{L_p(\R_\theta^d)}
+
\left(
\int_0^\infty
\bigl[t^{-\alpha}\omega_p^k(\x,t)\bigr]^q
\,\frac{\mathrm dt}{t}
\right)^{1/q}.
\]
Here
\[
\omega_p^k(\x,t)
:=
\sup_{0<|h|\leq t}
\|\Delta_h^k\x\|_{L_p(\R_\theta^d)},
\qquad t>0,
\]
is the \(k\)-th order modulus of smoothness, and
\(\Delta_h^k\), \(h\in\R^d\), denotes the \(k\)-th order difference
operator; see Subsection~\ref{section2.4}. Our concern, however, is the endpoint behavior, for a fixed
\(k\in\N\), of the higher-order Besov scale
\[
\bigl\{
B_{p,q}^{\alpha}(\R_\theta^d):
0<\alpha<k
\bigr\},
\]
equipped with the corresponding \(k\)-th order difference norms.
For convenience, for \(0<\alpha<k\), we write
\begin{equation}\label{eq:difference-Besov-norm-intro}
\begin{aligned}
\|\x\|_{B_{p,q}^{\alpha,k}(\R_\theta^d)}
&:=
\|\x\|_{L_p(\R_\theta^d)}
+
|\x|_{B_{p,q}^{\alpha,k}(\R_\theta^d)},\\
|\x|_{B_{p,q}^{\alpha,k}(\R_\theta^d)}
&:=
\left(
\int_0^\infty
\bigl[t^{-\alpha}\omega_p^k(\x,t)\bigr]^q
\,\frac{\mathrm dt}{t}
\right)^{1/q}.
\end{aligned}
\end{equation}
For $k\in\mathbb N$, we denote by
$W_p^k(\R_\theta^d)$ and $\dot W_p^k(\R_\theta^d)$ the inhomogeneous
and homogeneous Sobolev spaces, respectively; see
Section~\ref{section3.1}.

Our first result concerns the upper endpoint $\alpha\uparrow k$ and
provides a higher-order BBM-type formula on $\R_\theta^d$.

\begin{theorem}[BBM-type limit]\label{th1.2}
Let $1\leq p,q<\infty$ and $k\in\mathbb N$. Then, for every $\x\in W_p^k(\R^d_\theta)$,
\begin{align}\label{1.4}
\lim_{\alpha\uparrow k}
(k-\alpha)^{1/q}
\|\x\|_{B_{p,q}^{\alpha,k}(\R_\theta^d)}\approx q^{-1/q}\,|\x|_{\dot{W}_p^k(\R^d_\theta)}.
\end{align}
The implicit constants depend only on $d$, $p$, and $k$.
\end{theorem}

The behavior at the lower endpoint $\alpha\downarrow0$ is described by the following MS-type limit formula.
\begin{theorem}[MS-type limit]\label{th1.3}
Let $1\leq p,q<\infty$ and $k\in\mathbb N$. Suppose that
$\x\in B_{p,q}^{\alpha_0}(\R_\theta^d)$ for some
$\alpha_0\in(0,k)$. Then
\begin{align}\label{eq:MS-quantum-main}
\lim_{\alpha\downarrow0}
\alpha^{1/q}
\|\x\|_{B_{p,q}^{\alpha,k}(\R_\theta^d)}
\approx
q^{-1/q}
\|\x\|_{L_p(\R_\theta^d)}.
\end{align}
The implicit constants depend only on $d$, $p$, and $k$.
\end{theorem}

The endpoint spaces identified in \eqref{1.4} and
\eqref{eq:MS-quantum-main} are independent of the Besov index \(q\);
the latter enters only through the normalization and the explicit
factor \(q^{-1/q}\). 
 In contrast, the classical Sobolev formulae
\eqref{eq:BBM-limit} and \eqref{eq:MS-limit-norm} correspond to
\(q=p\). Thus, in this sense, Theorems~\ref{th1.2}
and~\ref{th1.3} provide higher-order Besov-scale extensions of the
classical BBM and MS formulae, with \(q\) independent of \(p\).

\subsection{Methods and difficulties}
\label{subsec:methods-difficulties}

We establish a unified proof of the BBM- and MS-type limits within
the framework of real interpolation by applying Milman's theorem
(Theorem~\ref{thm:Milman}). Two difficulties arise in carrying out
this approach on \(\R_\theta^d\). First, the homogeneous Sobolev
seminorm does not by itself yield a Banach space that forms a
compatible Banach couple with \(L_p(\R_\theta^d)\). Second, the verification of the normality
condition associated with the MS endpoint requires determining the large-scale behavior of the modulus
of smoothness, for which the
classical support-separation argument is unavailable. We explain
below how these difficulties are resolved and contrast the resulting
arguments with their counterparts on \(\R^d\) and on quantum tori
\(\mathbb T_\theta^d\).

\medskip
\noindent
\qquad\(\bullet\)\quad\emph{A compatible realization of the homogeneous endpoint.}
Since Milman's theorem is formulated for compatible Banach couples,
the homogeneous Sobolev endpoint appearing in the BBM limit,
naturally associated with the seminorm
\[
|\x|_{\dot W_p^k(\R_\theta^d)}
=
\sum_{|\beta|=k}
\|\partial^\beta\x\|_{L_p(\R_\theta^d)},
\]
must be realized as a Banach space \(X_1\) such that
\(
\bigl(L_p(\R_\theta^d),X_1\bigr)
\)
is a compatible Banach couple.

On the quantum torus \(\mathbb T_\theta^d\), this issue is resolved by
removing the zero Fourier mode. Indeed, set
\[
L_p^0(\mathbb T_\theta^d)
:=
\bigl\{
\x\in L_p(\mathbb T_\theta^d):
\widehat{\x}(0)=0
\bigr\},\qquad W_{p,0}^k(\mathbb T_\theta^d)
:=
\bigl\{
\x\in W_p^k(\mathbb T_\theta^d):
\widehat{\x}(0)=0
\bigr\}.
\]
The higher-order Poincar\'e inequality
\cite[Theorem~2.12]{XXY18} gives
\[
\|\x\|_{W_p^k(\mathbb T_\theta^d)}
\approx
|\x|_{\dot W_p^k(\mathbb T_\theta^d)},
\qquad
\x\in W_{p,0}^k(\mathbb T_\theta^d).
\]
Consequently,
\[
\vec X_{\mathbb T}
:=
\bigl(
L_p^0(\mathbb T_\theta^d),
W_{p,0}^k(\mathbb T_\theta^d)
\bigr),
\]
with \(W_{p,0}^k(\mathbb T_\theta^d)\) equipped with the homogeneous
Sobolev norm \(|\cdot|_{\dot W_p^k(\mathbb T_\theta^d)}\), is a compatible Banach couple to which Milman's theorem
applies. Such an argument is not available on \(\R_\theta^d\), since no
Poincar\'{e} inequality controls the \(L_p\)-norm by the homogeneous
Sobolev seminorm.

In the classical Euclidean setting of \(\R^d\),
the homogeneous Sobolev space may be constructed by completing 
$C_c^\infty(\R^d)$ with respect to the homogeneous Sobolev
seminorm.  At the distributional level, the kernel
of this seminorm is
\[
\mathcal P_{k-1}
:=
\{P:P\text{ is a polynomial of degree at most }k-1\},
\]
and hence one naturally works modulo \(\mathcal P_{k-1}\). In particular,
\(\mathcal P_{k-1}\cap C_c^\infty(\R^d)=\{0\}\). Rather than introducing a noncommutative analogue of this quotient,
we define \(\dot W_{p,0}^k(\R_\theta^d)\) as the completion of
\(\mcs(\R_\theta^d)\) with respect to
\(|\cdot|_{\dot W_p^k(\R_\theta^d)}\). Set
\[
\mathfrak A_k
:=
\{\beta\in\mathbb N_0^d:|\beta|=k\},
\qquad
E_{p,k}
:=
\ell_1\bigl(\mathfrak A_k;L_p(\R_\theta^d)\bigr),
\]
and consider
\[
\mathbf D^k:
\mcs(\R_\theta^d)\longrightarrow E_{p,k},
\qquad
\mathbf D^k\x
=
\bigl(\partial^\beta\x\bigr)_{\beta\in\mathfrak A_k}.
\]
The map \(\mathbf D^k\) extends uniquely to an isometric isomorphism
from \(\dot W_{p,0}^k(\R_\theta^d)\) onto the closure of
\(\mathbf D^k(\mcs(\R_\theta^d))\) in \(E_{p,k}\).

To show that
\(
\bigl(L_p(\R_\theta^d),\dot W_{p,0}^k(\R_\theta^d)\bigr)
\)
is a compatible Banach couple, let
\[
\mathcal Y_k
:=
\prod_{\beta\in\mathfrak A_k}
\mcs'(\R_\theta^d)
\]
equipped with the product of the weak-\(^*\) topologies. The natural inclusion $E_{p,k}\hookrightarrow\mathcal Y_k$
induces a continuous injective map from
$\dot W_{p,0}^k(\R_\theta^d)$ into $\mathcal Y_k$,
whereas \(L_p(\R_\theta^d)\) is mapped into \(\mathcal Y_k\) by
\[
\jmath_0(\x)
:=
\bigl(\partial^\beta\x\bigr)_{\beta\in\mathfrak A_k}.
\]
The key point is the injectivity of \(\jmath_0\). If
\[
\x\in L_p(\R_\theta^d),
\qquad
\partial^\beta\x=0
\quad\text{for every }|\beta|=k,
\]
then, for each \(j\) and
\(\phi\in\mcs(\R_\theta^d)\), the function
\[
t\longmapsto
\langle T_{t\mathbf e_j}\x,\phi\rangle
\]
is a bounded polynomial of degree at most \(k-1\), and is therefore
constant. Hence \(\x\) is translation invariant. It follows that
\(\x=c\mathbf1\) as a quantum tempered distribution. Since
\(\tau_\theta(\mathbf1)=\infty\), one has
\(\mathbf1\notin L_p(\R_\theta^d)\) for \(1\leq p<\infty\), and
therefore \(c=0\). We refer to the proof of
Lemma~\ref{lem:compatible-homogeneous-couple} for the details. 
\vspace{0.3cm}

For
this resulting compatible couple, Corollary~
\ref{cor:homogeneous-K-functional} gives
\begin{equation}\label{cor:homogeneous-K-functional-intro}
K\bigl(
\x,\varepsilon^k;
L_p(\R_\theta^d),\dot W_{p,0}^k(\R_\theta^d)
\bigr)
\approx
\omega_p^k(\x,\varepsilon),
\qquad
\varepsilon>0,
\end{equation}
and Corollary~\ref{cor:interpolation-Besov} yields, with
\(\alpha=\eta k\), \(\eta \in (0,1)\),
\begin{equation}\label{eq:homogeneous-interpolation-norm-intro}
\left\|
\x
\right\|_{
\bigl(
L_p(\R_\theta^d),\dot W_{p,0}^k(\R_\theta^d)
\bigr)_{\eta,q}
}
\approx
k^{1/q}
|\x|_{B_{p,q}^{\alpha,k}(\R_\theta^d)}.
\end{equation}
In view of \eqref{cor:homogeneous-K-functional-intro}, the
\(X_1\)- and \(X_0\)-normality conditions in Milman's theorem amount,
respectively, to
\begin{equation}\label{eq:small-scale-normality-completion-intro}
\lim_{\varepsilon\downarrow0}
\frac{\omega_p^k(\x,\varepsilon)}{\varepsilon^k}
\approx
\|\x\|_{\dot W_{p,0}^k(\R_\theta^d)},
\qquad
\x\in\dot W_{p,0}^k(\R_\theta^d),
\end{equation}
and
\begin{equation}\label{eq:large-scale-normality-intro}
\lim_{\varepsilon\to\infty}
\omega_p^k(\x,\varepsilon)
\approx
\|\x\|_{L_p(\R_\theta^d)},
\qquad
\x\in L_p(\R_\theta^d).
\end{equation}
The small-scale relation
\eqref{eq:small-scale-normality-completion-intro} follows by density
from the characterization established in
Subsection~\ref{section3.1}.
 By contrast, the large-scale estimate
\eqref{eq:large-scale-normality-intro} is substantially more delicate
and constitutes the second principal difficulty.

\medskip
\noindent
\qquad\(\bullet\)\quad\emph{Large translations at the MS endpoint.}
We now outline the proof of
\eqref{eq:large-scale-normality-intro}. On the quantum torus \(\mathbb T_\theta^d\), the \(X_0\)-normality
condition for the mean-zero couple \(\vec X_{\mathbb T}\) causes no
additional difficulty.
Indeed, the \(K\)-functional characterization
\cite[Theorem~5.16]{XXY18}, together with the Poincar\'e inequality,
gives
\begin{equation}\label{eq:tori-K-function}
K\bigl(
\x,\varepsilon^k;
\vec X_{\mathbb T}
\bigr)
\approx
\omega_p^k(\x,\varepsilon),
\qquad
\x\in L_p^0(\mathbb T_\theta^d),\quad
0<\varepsilon\leq1.
\end{equation}
Moreover, the continuous embedding
\(
W_{p,0}^k(\mathbb T_\theta^d)
\hookrightarrow
L_p^0(\mathbb T_\theta^d)
\)
implies
\[
K\bigl(\x,1;\vec X_{\mathbb T}\bigr)
\approx
\|\x\|_{L_p(\mathbb T_\theta^d)}.
\]
Taking \(\varepsilon=1\) in \eqref{eq:tori-K-function}, we obtain
\[
\omega_p^k(\x,1)
\approx
\|\x\|_{L_p(\mathbb T_\theta^d)}.
\]
Hence, by the monotonicity of
\(\varepsilon\mapsto\omega_p^k(\x,\varepsilon)\) and the estimate
\(\omega_p^k(\x,\varepsilon)\leq2^k\|\x\|_p\),
\[
\lim_{\varepsilon\to\infty}
\omega_p^k(\x,\varepsilon)
\approx
\|\x\|_{L_p(\mathbb T_\theta^d)},
\qquad
\x\in L_p^0(\mathbb T_\theta^d).
\]

In the classical Euclidean setting of \(\R^d\), if
\(f\in C_c^\infty(\R^d)\), then \(T_hf\) and \(f\) have disjoint
supports for sufficiently large \(|h|\), and hence
\[
\|T_hf-f\|_p^p
=
2\|f\|_p^p.
\]
It follows by density that
\[
\lim_{t\to\infty}\omega_p^1(f,t)
\approx
\|f\|_{L_p(\R^d)}
\qquad
(f\in L_p(\R^d)).
\]

No analogous support-separation argument is available for
noncommutative \(L_p\)-elements on \(\R_\theta^d\). We instead analyze
the large-translation behavior of \(T_h\x-\x\).
Marchaud's inequality
(Proposition~\ref{prop:Marchaud}) reduces
\eqref{eq:large-scale-normality-intro} to the first-order modulus.
For \(\x\in\mcs(\R_\theta^d)\) and \(p=2r\) with \(r\in \N\), we expand
\[
\|T_h\x-\x\|_{2r}^{2r}
=
\tau_\theta\left(
\bigl((T_h\x-\x)^*(T_h\x-\x)\bigr)^r
\right).
\]
The iterated \(\theta\)-convolution and the trace formula
reduce the mixed and mixed diagonal terms in the tracial expansion to
oscillatory integrals of the form
\[
\int_{\R^N}
e^{2\pi i\langle h,L(u)\rangle}G(u)\,\mathrm du,
\qquad
G\in L_1(\R^N),
\]
where \(L:\R^N\to\R^d\) is surjective. By
Lemma~\ref{lem:matrix-RL}, these terms tend to zero as
\(|h|\to\infty\), whereas the two pure diagonal terms contribute
\(2\|\x\|_{2r}^{2r}\). Consequently,
\[
\lim_{|h|\to\infty}
\|T_h\x-\x\|_{2r}^{2r}
=
2\|\x\|_{2r}^{2r}.
\]

To extend the result from even exponents to all
\(1\leq p<\infty\), we introduce finite measures obtained by weighting
the spectral measures of \(|T_h\x-\x|^2\) and \(|\x|^2\). For \(p\geq2\), 
the even-exponent result and the Weierstrass
approximation theorem yield the required \(L_p\)-limit through these
weighted spectral measures.
For \(1\leq p<2\), a different choice of weights, together with weak
compactness and the Portmanteau theorem, gives the required lower
bounds for both \(T_h\x-\x\) and \(T_h\x+\x\). The noncommutative Clarkson inequality provides the matching upper
bound when \(1<p<2\), while the case \(p=1\) follows from the triangle
inequality.
Finally, density gives
\[
\lim_{|h|\to\infty}
\|T_h\x-\x\|_p
=
2^{1/p}\|\x\|_p,
\qquad
\x\in L_p(\R_\theta^d),\quad 1\leq p<\infty.
\]
Together with Marchaud's inequality, this proves
\eqref{eq:large-scale-normality-intro}; see Lemma \ref{lem:MS-normal-condition} for details.
\vspace{0.3cm}

Once the two normality conditions are established, Milman's theorem
gives the endpoint limits of the real interpolation norms associated
with 
\(
\bigl(L_p(\R_\theta^d),\dot W_{p,0}^k(\R_\theta^d)\bigr)
\).
To transfer the corresponding limits to the difference norms by using the uniform equivalence \eqref{eq:homogeneous-interpolation-norm-intro}, we also need the exact
identities established in
Appendix~\ref{app:exact-endpoint-asymptotics}:
\begin{equation}\label{eq:exact-BBM-modulus-limit-intro}
\lim_{\alpha\uparrow k}
(k-\alpha)^{1/q}
\|\x\|_{B_{p,q}^{\alpha,k}(\R_\theta^d)}
=
q^{-1/q}
\lim_{\varepsilon\downarrow0}
\frac{\omega_p^k(\x,\varepsilon)}{\varepsilon^k},\quad \x\in W_p^k(\R_\theta^d),
\end{equation}
and, whenever
\(\x\in B_{p,q}^{\alpha_0}(\R_\theta^d)\) for some
\(\alpha_0\in(0,k)\),
\begin{equation}\label{eq:exact-MS-modulus-limit-intro}
\lim_{\alpha\downarrow0}
\alpha^{1/q}
\|\x\|_{B_{p,q}^{\alpha,k}(\R_\theta^d)}
=
q^{-1/q}
\lim_{\varepsilon\to\infty}
\omega_p^k(\x,\varepsilon).
\end{equation}
Since the normalized \(L_p\)-term vanishes at both endpoints, these identities ensure
the existence of both normalized Besov seminorm limits. Together with 
\eqref{eq:homogeneous-interpolation-norm-intro}, Milman's theorem then
yields a unified real-interpolation proof of
Theorems~\ref{th1.2} and~\ref{th1.3}.

The remainder of the paper is organized as follows.
Section~\ref{section2} collects the necessary preliminaries on
noncommutative \(L_p\)-spaces, quantum Euclidean spaces, Fourier
multipliers, difference operators, and moduli of smoothness.
Section~\ref{section3} introduces the Sobolev and Besov spaces used
in the paper, establishes their difference characterizations, and
constructs a Banach realization of the homogeneous Sobolev endpoint
that forms a compatible couple with \(L_p(\R_\theta^d)\). In Section~\ref{section5}, we obtain the
relevant \(K\)-functional and real-interpolation characterizations of
Besov spaces,
establish the large-translation estimate at the MS endpoint, and
apply Milman's theorem to prove
Theorems~\ref{th1.2} and~\ref{th1.3}. Finally, Appendix~\ref{app:exact-endpoint-asymptotics} gives direct
proofs of the exact identities for the difference Besov
norms, which also yield alternative proofs of the main results.

Throughout this paper, we write $A\lesssim B$ to mean that there exists a constant $C > 0$, independent of the main parameters, such that $A \leq CB$. 
We write $A \approx B$ if both $A \lesssim B$ and $B \lesssim A$ hold. We denote by $\N$ the set of positive integers and set $\N_0=\N \cup \{0\}$.
 For $p\in[1, \infty]$, we write $p'$ for the H{\"o}lder conjugate exponent satisfying $\frac{1}{p}+\frac{1}{p'}=1$.

\section{Preliminaries}\label{section2}

\subsection{Noncommutative $L_p$-spaces}

Let $\mathcal{M}$ be a von Neumann algebra equipped with a normal semifinite faithful trace $\tau$. 
The support of $\x \in \mathcal{M}$, denoted by $s(\x)$, is the smallest projection $e$ satisfying $e\x e = \x$. 
Let $\mcs_\tau^+(\mathcal M)$ be the set of all positive elements $\x$ in $\mathcal{M}$ with $\tau(s(\x)) < \infty$, 
and let $\mcs_\tau(\mathcal M)$ denote the linear span of $\mcs_\tau^+(\mathcal M)$. Given $1 \leq p < \infty$, we define
\[
\|\x\|_{L_p(\mathcal{M})} = \bigl(\tau(|\x|^p)\bigr)^{1/p}, \quad \x \in \mcs_\tau(\mathcal M),
\]
where $|\x| = (\x^*\x)^{1/2}$. One can verify that $\|\cdot\|_p$ is well-defined and is a norm on $\mcs_\tau(\mathcal M)$.
 The completion of $(\mcs_\tau(\mathcal M), \|\cdot\|_p)$ is denoted by $L_p(\mcm)$, the noncommutative $L_p$-space associated with $(\mcm, \tau)$. We set $L_\infty(\mcm) = \mcm$ equipped with the operator norm. 
 The elements of $L_p(\mcm, \tau)$ can be realized as closed densely defined operators on the Hilbert space $\mch$ on which $\mcm$ acts.  
 For more information on noncommutative $L_p$-spaces, we refer to \cite{FK96,PX03}.

\subsection{Quantum Euclidean space $\R^d_\theta$}\label{section2.2}

We refer the reader to \cite{GaJM2022,HLW23,LSZ20,ME24,MSX20,RST2026} for a detailed discussion on quantum Euclidean spaces and for additional
information.

\textbf{Definition of $\R^d_\theta$.}\;
Noncommutative Euclidean spaces admit several equivalent definitions; here we follow the approach in \cite{LSZ20}.
\begin{definition}\label{def2.1}
Let $\theta=\{\theta_{j,k}\}$ be a fixed real antisymmetric $d \times d$ matrix. Define $\mathbb{R}^d_\theta$ (also denoted by $L_{\infty}(\mathbb{R}^d_\theta)$) as the von Neumann subalgebra of $\mcb(L_2(\mathbb{R}^d))$ generated by the unitary family $\{U_\theta(t)\}_{t \in \mathbb{R}^d}$, which acts on $L_2(\mathbb{R}^d)$ by
\begin{equation}\label{eq2.2}
\big(U_\theta(t) f\big)(r) = e^{-\frac{i}{2} \langle t, \theta r\rangle} f(r - t), \quad f \in L_2(\mathbb{R}^d), \quad r, t \in \mathbb{R}^d.
\end{equation}
\end{definition}

It can be proved that $\{U_\theta(t)\}_{t \in \mathbb{R}^d}$ is strongly continuous and satisfies the following Weyl relation:
\begin{equation}\label{eq2.3}
U_\theta(s)\, U_\theta(t) = e^{-\frac{i}{2}\langle s,\theta t\rangle}\, U_\theta(s+t), \qquad s,t\in\R^d,
\end{equation}
from which it follows that $U_\theta(s)\, U_\theta(t) = e^{i\langle t,\theta s\rangle}\,U_\theta(t)\, U_\theta(s)$ and $U_\theta(t)^* = U_\theta(-t)$.

\begin{remark}
In the case \(\theta=0\), \(\R_0^d\) is the von Neumann algebra
generated by the unitary group of translations on \(\R^d\), which is
\(*\)-isomorphic to \(L_\infty(\R^d)\).
\end{remark}

\textbf{Integration on $\R^d_\theta$.}\;
Now we introduce a map from $L_1(\mathbb{R}^d)$ to $L_\infty(\mathbb{R}_\theta^d)$ which is also denoted by $U_\theta$.
For $f \in L_1(\mathbb{R}^d)$, define $U_\theta(f) : L_2(\mathbb{R}^d) \to L_2(\mathbb{R}^d)$ by
\begin{equation}\label{eq2.7}
U_\theta(f)g(r) := \int_{\mathbb{R}^d} f(t)\,U_\theta(t)g(r)\,\mathrm{d}t, \qquad g \in L_2(\mathbb{R}^d).
\end{equation}
The \(L_2(\R^d)\)-valued integral is convergent
in the Bochner sense and one can show that \(U_\theta(f) \in L_\infty(\mathbb{R}_\theta^d)\) and 
\(
\|U_\theta(f)\|_{L_2\to L_2} \le\|f\|_{L_1(\R^d)};
\)
see \cite[Lemma~2.3]{MSX20}.

In what follows, \(\mcs(\R^d)\) denotes the Schwartz space on \(\R^d\). For any $f\in\mcs(\mathbb{R}^d)$, define the Fourier transform as
\[
\mcf(f)(\xi)
:=\widehat{f}(\xi) = \int_{\R^d} f(t)\, e^{-2\pi i\langle t,\xi\rangle}\,\mathrm{d}t, \qquad \xi\in \R^d.
\]
and the inverse Fourier transform
\[
\mcf^{-1}(f)(t)
:=\check{f}(t)= \int_{\R^d} f(\xi)\, e^{2\pi i\langle \xi,t\rangle}\,\mathrm{d}\xi,\qquad t\in \R^d.
\]
The Schwartz class on $\R_\theta^d$ is defined as the image of $\mcs(\mathbb{R}^d)$ under $U_\theta$:
\[
\mcs(\R_\theta^d) := \bigl\{\x \in L_\infty(\R_\theta^d) : \x = U_\theta(f) \text{ for some } f \in \mcs(\mathbb{R}^d)\bigr\}.
\]
Then $U_\theta : \mcs(\mathbb{R}^d) \to \mcs(\R_\theta^d)$ is a bijection, and $\mcs(\R_\theta^d)$ is equipped with the Fr\'{e}chet topology induced by $U_\theta$. The space of quantum tempered distributions $\mcs'(\R_\theta^d)$ is defined as the topological dual of $\mcs(\R_\theta^d)$, and $U_\theta$ extends to a bijection $\mcs'(\mathbb{R}^d) \to \mcs'(\R_\theta^d)$ via
\begin{equation}\label{eq:duality-Utheta}
(U_\theta(f),\, U_\theta(g)) := (f,\, \tilde{g}), \quad f\in\mcs'(\R^d),\; g \in \mcs(\mathbb{R}^d),
\end{equation}
where $\widetilde g(t):=g(-t)$.

For $\x = U_\theta(f) \in \mcs(\R_\theta^d)$ with $f \in \mcs(\mathbb{R}^d)$, define
\[
\tau_\theta(\x) := f(0).
\]
Then \(\tau_\theta\) extends to a normal
faithful semifinite trace on
\(\R_\theta^d\) (see \cite[Proposition~1.1]{GJP21}).  
The noncommutative $L_p$-space associated to $(\mathbb{R}_\theta^d,\tau_\theta)$ with norm $\|\cdot\|_{L_p(\mathbb{R}_\theta^d)}$ is denoted by
\(L_p(\mathbb{R}_\theta^d)\). 
The Schwartz class $\mcs(\mathbb{R}_\theta^d)$ is dense in
$L_p(\mathbb{R}_\theta^d)$ for $1\le p<\infty$ in the norm topology, and dense
in $L_\infty(\mathbb{R}_\theta^d)$ in the weak-$*$ topology, for which we refer
to \cite[Proposition~1.13]{GJP21} and \cite[Lemma~3.12]{MSX20}.  
Furthermore, $L_1(\R_\theta^d)+L_\infty(\R_\theta^d)$ embeds into $\mcs'(\R_\theta^d)$ via
\[
(\x,\phi) = \tau_\theta(\x\,\phi), \quad \x\in L_1(\R_\theta^d)+L_\infty(\R_\theta^d),\;\; \phi\in \mcs(\R_\theta^d).
\]

Define the \(\theta\)-convolution of \(f,g\in\mcs(\R^d)\) by
\begin{equation}\label{eq:theta-conv}
\begin{aligned}
(f*_\theta g)(s)
:=
\int_{\R^d}
e^{-\frac{\mathrm i}{2}\langle s,\theta t\rangle}
f(s-t)g(t)\,\mathrm dt 
\quad s\in\R^d.
\end{aligned}
\end{equation}
Set
\[
f^\sharp(s):=\overline{f(-s)},
\qquad s\in\R^d.
\]
The Weyl relation \eqref{eq2.3}, together with a change of variables,
gives
\[
U_\theta(f)^*= \int_{\mathbb R^d} \overline{f(-s)}\,U_\theta(s)\,\mathrm{d}s = U_\theta(f^\sharp),
\]
and
\begin{equation}
\begin{aligned}\label{eq:Utheta-product}
U_\theta(f)\,U_\theta(g)
&= \int_{\mathbb R^d}\Bigl(\int_{\mathbb R^d} e^{-\frac{i}{2}\langle s,\theta t\rangle}\,f(s-t)\,g(t)\,\mathrm{d}t\Bigr)U_\theta(s)\,\mathrm{d}s\\
&= U_\theta(f *_\theta g).
\end{aligned}
\end{equation}
Moreover,
\(
\mcs(\R^d)*_\theta\mcs(\R^d)
\subseteq
\mcs(\R^d),
\)
and the $\theta$-convolution \( *_\theta \) reduces to the usual convolution when
\(\theta=0\).

\textbf{Translations and differential calculus on $\R_\theta^d$.}\;
There are two equivalent ways to define the partial derivatives in quantum
Euclidean spaces, as discussed respectively in \cite{LMSZ2023,ME24} and \cite{LSZ20,MSX20}:
one through the translation group and the other through the commutator.
In this article, we will focus on the former one.  

For \(t\in\R^d\), let \(T_t\) be the unique \(^*\)-automorphism of
\(L_\infty(\R_\theta^d)\) determined by
\begin{equation}\label{eq:translation-automorphism}
T_t\bigl(U_\theta(s)\bigr)
=
e^{2\pi i\langle s,t\rangle}U_\theta(s),
\qquad t,s\in\R^d.
\end{equation}
Equivalently, let $V_t$ be the unitary operator on
$L_2(\R^d)$ defined by
\[
(V_tg)(r)
=
e^{2\pi i\langle t,r\rangle}g(r),
\qquad g\in L_2(\R^d),
\]
then the translation by $t$ on $\R_\theta^d$ may be defined as
\[
T_t\x
:=
V_t\x V_t^*,
\qquad
\x\in L_\infty(\R_\theta^d).
\]

For $f\in\mcs(\R^d)$, the definition of $U_\theta(f)$ gives
\begin{equation}\label{eq:translation-formula}
T_t\,U_\theta(f)
=
U_\theta\bigl(e^{2\pi i\langle t,\cdot\rangle}f\bigr),
\qquad
t\in\R^d.
\end{equation}
It follows that \(T_t\) preserves \(\tau_\theta\) and extends to a
positive isometry on \(L_p(\R_\theta^d)\) for every
\(1\leq p\leq\infty\). The group
\((T_t)_{t\in\R^d}\) is strongly continuous on
\(L_p(\R_\theta^d)\) whenever \(1\leq p<\infty\); see
\cite{MSX20}.
More generally, for $\x\in\mcs'(\R_\theta^d)$, define $T_t\x\in
\mcs'(\R_\theta^d)$ by
\begin{equation}\label{e2.8}
(T_t\,\x,\phi)
=
(\x,T_{-t}\,\phi),
\qquad
\phi\in\mcs(\R_\theta^d).
\end{equation}

For $1\leq p<\infty$ and $j\in\{1,\ldots,d\}$, we define 
$\partial_j$ as the infinitesimal generator of the \(C_0\)-group
\((T_{r\mathbf e_j})_{r\in\R}\) on \(L_p(\R_\theta^d)\), where
\(\mathbf e_j\) denotes the \(j\)-th standard basis vector of
\(\R^d\). Its domain is
\[
\operatorname{Dom}(\partial_j)
:=
\left\{
\x\in L_p(\R_\theta^d):
\lim_{r\to0}
\frac{T_{r\mathbf{e}_j}\x-\x}{r}
\ \text{exists in }L_p(\R_\theta^d)
\right\},
\]
and, for $\x\in\operatorname{Dom}(\partial_j)$,
\begin{align}\label{eq:translation-generator-def}
\partial_j\x
:=
\lim_{r\to0}
\frac{T_{r\mathbf{e}_j}\x-\x}{r}.
\end{align}
For a multi-index
$\alpha=(\alpha_1,\ldots,\alpha_d)\in\N_0^d$, we sFor a multi-index
\(\alpha=(\alpha_1,\ldots,\alpha_d)\in\N_0^d\), set
\[
\partial^\alpha
:=
\partial_1^{\alpha_1}\cdots\partial_d^{\alpha_d}.
\]
For \(f\in\mcs(\R^d)\), one has
\begin{equation}\label{eq:partial-schwartz}
\partial^\alpha U_\theta(f)
=
U_\theta\bigl((2\pi i t)^\alpha f(t)\bigr)
\in
\mcs(\R_\theta^d),
\end{equation}
where
\(
t^\alpha=t_1^{\alpha_1}\cdots t_d^{\alpha_d}.
\)

For \(\x\in\mcs'(\R_\theta^d)\), define its distributional
derivatives by
\[
(\partial^\alpha\x,\phi)
:=
(-1)^{|\alpha|}(\x,\partial^\alpha\phi),
\qquad
\phi\in\mcs(\mathbb R_\theta^d).
\]
Then, in \(\mcs'(\R_\theta^d)\),
\begin{equation}\label{eq:deriv-via-translation}
\partial_j\x
=
\left.
\frac{\mathrm d}{\mathrm dr}
T_{r\mathbf e_j}\x
\right|_{r=0};
\end{equation}
equivalently,
\[
\lim_{r\to0}
\left(
\frac{T_{r\mathbf e_j}\x-\x}{r},
\phi
\right)
=
(\partial_j\x,\phi),
\qquad
\phi\in\mcs(\R_\theta^d).
\]
If \(1\leq p<\infty\) and
\(\x\in\operatorname{Dom}(\partial_j)\), this distributional
derivative coincides with the \(L_p\)-derivative defined in
\eqref{eq:translation-generator-def}.

\textbf{Convolution on $\R_\theta^d$.}\;
\begin{definition}\label{def:convolution}
Let $1 \leq p \leq \infty$ and $\x \in L_p(\mathbb{R}^d_\theta)$. For $\psi \in L_1(\mathbb{R}^d)$,  define
\[
\psi * \x = \int_{\mathbb{R}^d} \psi(s)\, T_{-s}\, \x \,\mathrm{d}s
\]
as an $L_p(\mathbb{R}^d_\theta)$-valued Bochner integral when $p < \infty$, and as a weak$^*$ integral when $p = \infty$.
\end{definition}

Since \(T_s\) is an isometry on \(L_p(\R_\theta^d)\), the triangle
inequality for Bochner integrals gives
\begin{equation}\label{eq:conv-Lp-bound}
\|\psi*\x\|_{L_p(\R_\theta^d)}
\leq
\|\psi\|_{L_1(\R^d)}
\|\x\|_{L_p(\R_\theta^d)}.
\end{equation}
If \(\psi\in\mcs(\R^d)\) and
\(\x=U_\theta(f)\in\mcs(\R_\theta^d)\), then
\begin{equation}\label{eq:conv-schwartz}
\psi*\x
=
\int_{\R^d}\psi(s)T_{-s}\x\,\mathrm ds
=
U_\theta(\widehat\psi\,f)
\in\mcs(\R_\theta^d).
\end{equation}
For \(\psi\in\mcs(\R^d)\), convolution extends to
\(\mcs'(\R_\theta^d)\) by duality:
\begin{equation}\label{eq:conv-distribution}
(\psi*\x,\phi)
:=
(\x,\widetilde\psi*\phi),
\qquad
\x\in\mcs'(\R_\theta^d),\quad
\phi\in\mcs(\R_\theta^d),
\end{equation}
where
\(\widetilde\psi(s):=\psi(-s)\).
This extension commutes with distributional derivatives. Namely, for
\(\x\in\mcs'(\R_\theta^d)\), \(\psi\in\mcs(\R^d)\), and
\(\alpha\in\N_0^d\),
\begin{equation}\label{eq:conv-schw-commut-distri}
\partial^\alpha(\psi*\x)
=
(\partial^\alpha\psi)*\x
=
\psi*(\partial^\alpha\x)
\qquad
\text{in }\mcs'(\R_\theta^d).
\end{equation}

\subsection{Fourier transform and Fourier multipliers on $\R^d_\theta$}\label{Section 2.3}

By \eqref{eq:duality-Utheta}, \(U_\theta\) extends to a bijection from
\(\mcs'(\R^d)\) onto \(\mcs'(\R_\theta^d)\).  We define the Fourier transform
of \(\x\in\mcs'(\R_\theta^d)\) by
\[
\widehat{\x}
:=
U_\theta^{-1}(\x)
\in
\mcs'(\R^d).
\]
If \(\x\in\mcs(\R_\theta^d)\), then \(\widehat{\x}\in\mcs(\R^d)\)
and
\[
\widehat{\x}(s)
=
\tau_\theta\bigl(\x U_\theta(s)^*\bigr),
\qquad
s\in\R^d.
\]
By definition,
\begin{equation}\label{eq:dis-f-nf}
\widehat{U_\theta(f)}=f,
\qquad f\in\mcs'(\R^d).
\end{equation}

The Plancherel identities
\cite[Lemmas~2.4 and~2.7]{HLW23} show that \(U_\theta\) extends to a
unitary operator from \(L_2(\R^d)\) onto
\(L_2(\R_\theta^d)\), whose inverse is the Fourier transform.
Consequently,
\[
U_\theta(\widehat{\x})
=
\x,
\qquad
\x\in L_2(\R_\theta^d).
\]
For \(1\leq p\leq2\), the Hausdorff--Young inequality
\cite[Lemma~2.4]{HLW23} gives
\[
\|\widehat{\x}\|_{L_{p'}(\R^d)}
\leq
\|\x\|_{L_p(\R_\theta^d)},
\qquad
\x\in L_p(\R_\theta^d).
\]

\begin{definition}\label{def:fourier-multiplier}
Let $g \in C^\infty(\mathbb R^d)$ be such that \(g\) and all its derivatives has at most polynomial growth.
The Fourier multiplier with
symbol \(g\) is defined by
\begin{equation}\label{eq:fourier-multiplier}
g(D)\x
:=
U_\theta\bigl(g\,\widehat{\x}\bigr),
\qquad
\x\in\mcs'(\R_\theta^d).
\end{equation}
\end{definition}

For \(\x\in\mcs(\R_\theta^d)\), this definition takes the form
\[
g(D)\x
=
\int_{\R^d}
g(\xi)\widehat{\x}(\xi)U_\theta(\xi)\,\mathrm d\xi.
\]
If \(g\in\mcs(\R^d)\), then
\begin{equation}\label{eq:multiplier-convolution}
g(D)\x
=
\check g*\x,
\qquad
\x\in\mcs'(\R_\theta^d).
\end{equation}
If, more generally, \(\check g\in L_1(\R^d)\), then \(g(D)\) acts
boundedly on \(L_p(\R_\theta^d)\) for every \(1\leq p\leq\infty\),
and \eqref{eq:conv-Lp-bound} gives
\[
\|g(D)\x\|_{L_p(\R_\theta^d)}
\leq
\|\check g\|_{L_1(\R^d)}
\|\x\|_{L_p(\R_\theta^d)}.
\]
Further results on Fourier multipliers on \(\R_\theta^d\) may be
found in \cite{ME24,RST2026}.

\subsection{Difference operators and moduli of smoothness on
\(\R_\theta^d\)}
\label{section2.4}

For \(h\in\R^d\), set
\[
e_h(\xi):=e^{2\pi i\langle h,\xi\rangle},
\qquad
d_h(\xi):=e_h(\xi)-1.
\]
By the distributional extension of
\eqref{eq:translation-formula} and \eqref{eq:dis-f-nf},
\[
\widehat{T_h\x}=e_h\,\widehat{\x},
\qquad
\x\in\mcs'(\R_\theta^d).
\]
Consequently,
\(T_h=e_h(D)\) on \(\mcs'(\R_\theta^d)\).

\begin{definition}\label{def:difference}
For \(h\in\R^d\) and \(m\in\N\), define
\[
\Delta_h^m
:=
(T_h-I)^m
=
(d_h^m)(D)
\]
on \(\mcs'(\R_\theta^d)\). Equivalently,
\begin{equation}\label{eq:binomial-diff}
\Delta_h^m\x
=
\sum_{j=0}^m
(-1)^{m-j}\binom{m}{j}T_{jh}\x.
\end{equation}
For \(1\leq p\leq\infty\), \(\x\in L_p(\R_\theta^d)\), and
\(\varepsilon>0\), the modulus of smoothness of order \(m\) is defined by
\[
\omega_p^m(\x,\varepsilon)
:=
\sup_{0<|h|\leq\varepsilon}
\|\Delta_h^m\x\|_{L_p(\R_\theta^d)}.
\]
\end{definition}

\begin{remark}\label{rem:modulus-properties}
Let \(1\leq p\leq\infty\), \(m\in\N\), and
\(\x\in L_p(\R_\theta^d)\). The following properties will be used
throughout:
\begin{enumerate}[label=(\roman*)]
\item
For every \(\varepsilon>0\),
\begin{align}\label{eq:diff-bound}
\omega_p^m(\x,\varepsilon)
\leq
2^m\|\x\|_{L_p(\R_\theta^d)}.
\end{align}

\item
The function
\(\varepsilon\mapsto\omega_p^m(\x,\varepsilon)\) is nondecreasing
and, when \(p<\infty\), continuous on \((0,\infty)\).

\item
For every \(\varepsilon>0\), the map
\(
\x\longmapsto\omega_p^m(\x,\varepsilon)
\)
is a seminorm on \(L_p(\R_\theta^d)\).

\item
If \(k\in\N\) and \(k<m\), then
\begin{equation}\label{eq:modulus-monotone-order}
\omega_p^m(\x,\varepsilon)
\leq
2^{m-k}\omega_p^k(\x,\varepsilon),
\qquad
\varepsilon>0.
\end{equation}
\end{enumerate}
\end{remark}

For later use, we record the following scaling properties of the
moduli of smoothness.
\begin{lemma}\label{lem:moduli-scaling}
Let \(m\in\N\), \(1\leq p\leq\infty\), and
\(\x\in L_p(\R_\theta^d)\). Then
\begin{equation}\label{eq:moduli-integer-scaling}
\omega_p^m(\x,n\varepsilon)
\leq
n^m\omega_p^m(\x,\varepsilon),
\qquad
n\in\N,\quad \varepsilon>0,
\end{equation}
and
\begin{equation}\label{eq:moduli-real-scaling}
\omega_p^m(\x,\lambda\varepsilon)
\leq
(\lambda+1)^m\omega_p^m(\x,\varepsilon),
\qquad
\lambda>0.
\end{equation}
\end{lemma}

\begin{proof}
For \(n\in\mathbb N\), the identity
\[
d_{nh}
=
e_{nh}-1
=
e_h^n-1
=
\left(\sum_{j=0}^{n-1}e_{jh}\right)d_h
\]
implies that
\(
d_{nh}^m
=
\left(\sum_{j=0}^{n-1}e_{jh}\right)^m d_h^m.
\)
Since \(e_{jh}(D)=T_{jh}\) and \(d_h^m(D)=\Delta_h^m\), it follows that
\[
\Delta_{nh}^m
=
\left(\sum_{j=0}^{n-1}T_{jh}\right)^m\Delta_h^m.
\]
The \(L_p\)-isometry of the translations therefore gives
\[
\|\Delta_{nh}^m\x\|_{L_p(\R_\theta^d)}
\leq
n^m\|\Delta_h^m\x\|_{L_p(\R_\theta^d)}.
\]
Taking the supremum over \(0<|h|\leq\varepsilon\) proves
\eqref{eq:moduli-integer-scaling}.
For \(\lambda>0\), let \(n=\lceil\lambda\rceil\). By monotonicity and
\eqref{eq:moduli-integer-scaling},
\[
\omega_p^m(\x,\lambda\varepsilon)
\leq
\omega_p^m(\x,n\varepsilon)
\leq
n^m\omega_p^m(\x,\varepsilon)
\leq
(\lambda+1)^m\omega_p^m(\x,\varepsilon).
\]
\end{proof}

\section{Characterizations of Sobolev and Besov spaces on $\R_\theta^d$ via differences}\label{section3}

In this section, we develop the Sobolev and Besov space framework
needed for the proofs of the main results. We first introduce the
inhomogeneous Sobolev spaces, construct a Banach realization of the
homogeneous Sobolev space compatible with \(L_p(\R_\theta^d)\), and
establish a characterization of the homogeneous Sobolev seminorm by
moduli of smoothness. We then recall the Littlewood--Paley definition
of quantum Besov spaces and its equivalent characterization in terms
of differences.

\subsection{Sobolev spaces on $\R_\theta^d$ and their characterization via differences}\label{section3.1}

We begin by recalling the integer-order Sobolev spaces on $\R_\theta^d$.

\begin{definition}\label{def:sobolev}
Let $1\leq p \leq \infty$, and $m\in \N$.
The Sobolev space $W_p^m(\R_\theta^d)$ consists of all
$\x\in\mcs'(\R_\theta^d)$ such that $\partial^\beta \x \in L_p(\mathbb{R}_\theta^d)$ for every $\beta \in \mathbb{N}_0^d$ with $|\beta| \leq m$, equipped with the norm
\begin{equation}\label{eq:inhomogeneous-Sobolev-norm}
\|\x\|_{W_p^m(\R_\theta^d)} = \sum_{|\beta| \leq m} \|\partial^\beta \x\|_{L_p(\R_\theta^d)}.
\end{equation}
We also write
\begin{equation}\label{eq:homogeneous-Sobolev-seminorm}
|\x|_{\dot W_p^m(\R_\theta^d)}
:=
\sum_{|\beta|=m}
\|\partial^\beta \x\|_{L_p(\R_\theta^d)}
\end{equation}
for the corresponding homogeneous seminorm.
\end{definition}

By \cite[Propositions~3.3 and~3.14]{MSX20},
\(W_p^m(\R_\theta^d)\), equipped with the norm
\eqref{eq:inhomogeneous-Sobolev-norm}, is a Banach space for
\(1\leq p\leq\infty\). Moreover,
\(\mcs(\R_\theta^d)\) is dense in \(W_p^m(\R_\theta^d)\) whenever
\(1\leq p<\infty\).

For the real-interpolation argument in
Section~\ref{section5}, we need to realize the homogeneous Sobolev seminorm as the norm of a
Banach space compatible with \(L_p(\R_\theta^d)\).  We first observe that
\eqref{eq:homogeneous-Sobolev-seminorm} is a norm on
\(\mcs(\R_\theta^d)\). Indeed, suppose that
\(\x=U_\theta(f)\in\mcs(\R_\theta^d)\) and
\(\partial^\beta \x=0\) for every \(|\beta|=m\). 
Since
\[
\partial^\beta U_\theta(f)
=
U_\theta\bigl((2\pi i\xi)^\beta f\bigr)
\]
and $U_\theta$ is injective, we have
\[
\xi_j^m f(\xi)=0,
\qquad j=1,\ldots,d.
\]
Hence \(f(\xi)=0\) whenever \(\xi\ne0\). Continuity then gives \(f(0)=0\), so
\(f=0\) and therefore \(\x=0\). Thus
\(|\cdot|_{\dot W_p^m(\R_\theta^d)}\) is a norm on
\(\mcs(\R_\theta^d)\).

We now define the corresponding homogeneous Sobolev space by
completion.

\begin{lemma}\label{lem:compatible-homogeneous-couple}
Let $1\leq p<\infty$ and $k\in\mathbb N$. Let
$\dot W_{p,0}^k(\R_\theta^d)$ be the completion of
$\mcs(\R_\theta^d)$ with respect to
$|\cdot|_{\dot W_p^k}$, equipped with the induced norm
$\|\cdot\|_{\dot W_{p,0}^k}$. Then
$\dot W_{p,0}^k(\R_\theta^d)$ is a Banach space, and
\[
\bigl(
L_p(\R_\theta^d),
\dot W_{p,0}^k(\R_\theta^d)
\bigr)
\]
is a compatible Banach couple.
\end{lemma}

\begin{proof}
Set
\[
\mathfrak A_k
:=
\bigl\{
\beta\in\mathbb N_0^d:|\beta|=k
\bigr\},
\qquad
E_{p,k}
:=
\ell_1\bigl(
\mathfrak A_k;L_p(\R_\theta^d)
\bigr),
\]
with norm
\[
\|(\x_\beta)_{\beta\in\mathfrak A_k}\|_{E_{p,k}}
:=
\sum_{\beta\in\mathfrak A_k}
\|\x_\beta\|_{L_p(\R_\theta^d)}.
\]
Then \(E_{p,k}\) is a Banach space. Define 
\[
\mathbf D^k:
\mcs(\mathbb R_\theta^d)
\longrightarrow E_{p,k},
\qquad
\mathbf D^k \x
=
(\partial^\beta \x)_{\beta\in\mathfrak A_k}.
\]
By definition,
\[
\|\mathbf D^k \x\|_{E_{p,k}}
=
|\x|_{\dot W_p^k},
\qquad
\x\in\mcs(\R_\theta^d).
\]
Thus \(\mathbf D^k\) is an isometry when
\(\mcs(\R_\theta^d)\) is equipped with
\(|\cdot|_{\dot W_p^k(\R_\theta^d)}\). Consequently, it extends
uniquely to an isometric isomorphism from
\(\dot W_{p,0}^k(\R_\theta^d)\) onto
\[
\mathcal R_{p,k}
:=
\overline{
\mathbf D^k\bigl(\mcs(\R_\theta^d)\bigr)
}^{\,E_{p,k}}.
\]
Since \(\mathcal R_{p,k}\) is a closed subspace of \(E_{p,k}\),
\(\dot W_{p,0}^k(\R_\theta^d)\) is a Banach space.  In what follows, we identify
$\dot W_{p,0}^k(\R_\theta^d)$ with $\mathcal R_{p,k}$.
Under this identification, every
$z\in\dot W_{p,0}^k(\R_\theta^d)$ has the form
\[
\z=(\z_\beta)_{\beta\in\mathfrak A_k},
\qquad
\z_\beta\in L_p(\R_\theta^d),
\]
and
\[
\|\z\|_{\dot W_{p,0}^k(\R_\theta^d)}
=
\sum_{\beta\in\mathfrak A_k}
\|\z_\beta\|_{L_p(\R_\theta^d)}.
\]

To establish compatibility, we construct continuous injective maps
from both spaces into a common Hausdorff topological vector space.  
Let
\[
\mathcal Y_k
=
\prod_{\beta\in\mathfrak A_k}
\mcs'(\R_\theta^d)
\]
equipped with the product topology induced by the weak-$^*$ topology
\(
\sigma\bigl(\mcs'(\R_\theta^d),\mcs(\R_\theta^d)\bigr)
\) on each factor.
Then \(\mathcal Y_k\) is a Hausdorff locally convex space.

Define
\[
\jmath_0:
L_p(\R_\theta^d)
\longrightarrow
\mathcal Y_k,
\qquad
\jmath_0(\x)
:=
\bigl(\partial^\beta\x\bigr)_{\beta\in\mathfrak A_k}.
\]
For every \(\beta\in\mathfrak A_k\) and
\(\phi\in\mcs(\R_\theta^d)\), H\"{o}lder's inequality gives
\[
\bigl|
\langle\partial^\beta\x,\phi\rangle
\bigr|
=
\bigl|
\langle\x,(-1)^k\partial^\beta\phi\rangle
\bigr|
\leq
\|\x\|_{L_p(\R_\theta^d)}
\|\partial^\beta\phi\|_{L_{p'}(\R_\theta^d)}.
\]
Hence each coordinate map
\(\x\mapsto\partial^\beta \x\) is continuous from \(L_p(\R_\theta^d)\) into
\(\mcs'(\R_\theta^d)\) endowed with its weak-\(^*\) topology.
Therefore, \(\jmath_0\) is continuous.

We now prove that \(\jmath_0\) is injective. Let
\(\x\in L_p(\R_\theta^d)\) satisfy
\(
\partial^\beta \x=0\) for every \(\beta\in\mathfrak A_k\).
In particular,
\[
\partial_j^k \x=0,
\qquad j=1,\ldots,d.
\]
Fix \(j\in\{1,\ldots,d\}\) and
\(\phi\in\mcs(\mathbb R_\theta^d)\), and set
\[
P_{j,\phi}(t)
:=
\langle T_{t\mathbf e_j}\x,\phi\rangle,
\qquad t\in\mathbb R.
\]
The map \(t\mapsto T_{-t\mathbf e_j}\phi\) is smooth with values in
\(\mcs(\R_\theta^d)\), and hence \(P_{j,\phi}\in C^\infty(\R)\).
Moreover,
\[
P_{j,\phi}^{(r)}(t)
=
\bigl\langle
T_{t\mathbf e_j}\partial_j^r\x,\phi
\bigr\rangle,
\qquad
r\in\mathbb N_0.
\]
Thus \(P_{j,\phi}^{(k)}=0\), and so \(P_{j,\phi}\) is a polynomial of
degree at most \(k-1\). On the other hand, H\"{o}lder's inequality and
the \(L_p\)-isometry of the translations give
\[
|P_{j,\phi}(t)|
\leq
\|\x\|_{L_p(\R_\theta^d)}
\|\phi\|_{L_{p'}(\R_\theta^d)},
\qquad
t\in\R.
\]
Hence \(P_{j,\phi}\) is constant. Since \(\phi\) is arbitrary,
\[
T_{t\mathbf e_j}\x=\x,
\qquad
t\in\R,\quad j=1,\ldots,d.
\]
The group property of the translations therefore yields
\[
T_h\x=T_{h_1\mathbf e_1}\cdots T_{h_d\mathbf e_d}\x=\x,
\qquad
h\in\R^d.
\]
Taking Fourier transforms, we obtain
\[
e^{2\pi i\langle h,\cdot\rangle}\widehat \x
=
\widehat \x
\quad\text{in }\mcs'(\mathbb R^d),
\qquad h\in\mathbb R^d.
\]
Differentiating this identity with respect to \(h_j\) at \(h=0\)
gives
\[
\xi_j\widehat{\x}=0,
\qquad
j=1,\ldots,d.
\]
Consequently,
\(\operatorname{supp}\widehat{\x}\subset\{0\}\). Hence
\[
\widehat \x
=
\sum_{|\gamma|\le N}
c_\gamma\partial^\gamma\delta_0
\]
for some \(N\in\mathbb N_0\). Since
\[
\xi_j\partial^\gamma\delta_0
=
-\gamma_j\partial^{\gamma-\mathbf e_j}\delta_0,
\]
the identities \(\xi_j\widehat{\x}=0\) imply that
\(c_\gamma=0\) for every \(\gamma\neq0\). Hence
\[
\widehat{\x}=c_0\delta_0.
\]
 Finally, \(U_\theta(\delta_0)=\mathbf1\) in
\(\mcs'(\R_\theta^d)\). Indeed, for \(g\in\mcs(\R^d)\),
\[
\langle U_\theta(\delta_0),U_\theta(g)\rangle
=
\langle\delta_0,\widetilde g\rangle
=
g(0)
=
\tau_\theta(U_\theta(g))
=
\langle\mathbf1,U_\theta(g)\rangle.
\]
It follows that \(\x=c_0\mathbf1\). Since
\(\tau_\theta(\mathbf1)=\infty\), one has
\(\mathbf1\notin L_p(\R_\theta^d)\) for \(1\leq p<\infty\).
Therefore \(c_0=0\), and hence \(\x=0\). Thus \(\jmath_0\) is
injective.

Using the identification
\(\dot W_{p,0}^k(\R_\theta^d)=\mathcal R_{p,k}\), define
\[
\jmath_1:
\dot W_{p,0}^k(\R_\theta^d)
\longrightarrow
\mathcal Y_k,
\qquad
\jmath_1(z)
:=
\bigl(\iota_p(z_\beta)\bigr)_{\beta\in\mathfrak A_k},
\]
where
$z=(z_\beta)_{\beta\in\mathfrak A_k}\in\mathcal R_{p,k}$
 and
\(\iota_p:L_p(\mathbb R_\theta^d)\hookrightarrow
\mcs'(\mathbb R_\theta^d)\)
denotes the canonical embedding. For every
\(\beta\in\mathfrak A_k\) and
\(\phi\in\mcs(\mathbb R_\theta^d)\),
\[
\bigl|
\langle\iota_p(\z_\beta),\phi\rangle
\bigr|
\leq
\|\z_\beta\|_{L_p(\R_\theta^d)}
\|\phi\|_{L_{p'}(\R_\theta^d)}
\leq
\|\z\|_{\dot W_{p,0}^k(\R_\theta^d)}
\|\phi\|_{L_{p'}(\R_\theta^d)}.
\]
Thus each coordinate map of \(\jmath_1\) is continuous, and hence
\(\jmath_1\) is continuous with respect to the product topology of
\(\mathcal Y_k\). Moreover, if \(\jmath_1(\z)=0\), then the
injectivity of \(\iota_p\) gives \(\z_\beta=0\) for every
\(\beta\in\mathfrak A_k\), so \(\z=0\). Therefore,
\(\jmath_1\) is injective.

For \(\x\in\mcs(\R_\theta^d)\), regarded as an element of
\(\dot W_{p,0}^k(\R_\theta^d)\), one has
\[
\jmath_1(\x)
=
\bigl(\partial^\beta\x\bigr)_{\beta\in\mathfrak A_k}
=
\jmath_0(\x).
\]
Consequently, \(\jmath_0\) and \(\jmath_1\) are continuous
injections of \(L_p(\R_\theta^d)\) and
\(\dot W_{p,0}^k(\R_\theta^d)\), respectively, into the Hausdorff
topological vector space \(\mathcal Y_k\). Hence
\[
\bigl(
L_p(\R_\theta^d),
\dot W_{p,0}^k(\R_\theta^d)
\bigr)
\]
is a compatible Banach couple.
\end{proof}

We next establish a Gagliardo--Nirenberg inequality, from which we
derive an equivalent norm on \(W_p^k(\R_\theta^d)\) involving only
the \(L_p\)-norm and the derivatives of order \(k\).

\begin{lemma}\label{lem:Gagliardo-Nirenberg}
Let \(1\leq p\leq\infty\), \(k\in\N\), and \(1\leq m<k\).
Then, for every \(\x\in W_p^k(\R_\theta^d)\),
\begin{equation}\label{eq:GN-interpolation}
|\x|_{\dot W_p^m(\R_\theta^d)}
\lesssim_{d,m,k}
\|\x\|_{L_p(\R_\theta^d)}^{1-m/k}
|\x|_{\dot W_p^k(\R_\theta^d)}^{m/k}.
\end{equation}
\end{lemma}

\begin{proof}
Let \(y\in W_p^2(\R_\theta^d)\), \(t>0\), and
\(j\in\{1,\ldots,d\}\). Applying Taylor's formula to 
\(s\longmapsto T_{st\mathbf e_j}y\),
 \(0\leq s\leq1,
\)
gives
\[
\Delta_{t \mathbf{e}_j}y
=
t\,\partial_jy
+
t^2\int_0^1
(1-s)T_{st \mathbf{e}_j}(\partial_j^2y)\,\mathrm ds,
\]
where, for \(p=\infty\), the identity and the integral are understood
in the weak-\(^*\) sense. The \(L_p\)-isometry of the translations
therefore yields
\[
t\|\partial_jy\|_p
\leq
2\|y\|_p
+
\frac{t^2}{2}\|\partial_j^2y\|_p.
\]
Dividing by \(t\) and optimizing over \(t>0\), we obtain
\begin{equation}\label{eq:Landau-Kolmogorov}
\|\partial_jy\|_p
\leq
2\|y\|_p^{1/2}\|\partial_j^2y\|_p^{1/2}.
\end{equation}

For \(\x\in W_p^k(\R_\theta^d)\), set
\[
A_0:=\|\x\|_p,
\quad
A_r
:=
|\x|_{\dot W_p^r(\R_\theta^d)}
=
\sum_{|\beta|=r}
\|\partial^\beta \x\|_p,
\quad 1\leq r\leq k.
\]
Let \(1\leq r<k\). For every multi-index \(\beta\) with
\(|\beta|=r-1\), one has
\(\partial^\beta\x\in W_p^2(\R_\theta^d)\). Applying
\eqref{eq:Landau-Kolmogorov} to
\(y=\partial^\beta\x\), summing over \(\beta\) and \(j\), and using
the Cauchy--Schwarz inequality, we obtain
\[
\begin{aligned}
A_r
&\leq
\sum_{|\beta|=r-1}\sum_{j=1}^d
\|\partial^{\beta+e_j}\x\|_p                                      \\
&\leq
2
\left(
\sum_{|\beta|=r-1}\sum_{j=1}^d
\|\partial^\beta \x\|_p
\right)^{1/2}
\left(
\sum_{|\beta|=r-1}\sum_{j=1}^d
\|\partial^{\beta+2e_j}\x\|_p
\right)^{1/2}                                                   
\lesssim_d
A_{r-1}^{1/2}A_{r+1}^{1/2}.
\end{aligned}
\]
An induction based on the preceding estimate yields
\[
A_m
\lesssim_{d,m,k}
A_0^{1-m/k}A_k^{m/k},
\qquad 1\leq m<k,
\]
which is precisely \eqref{eq:GN-interpolation}.
\end{proof}

\begin{corollary}\label{cor:top-order-Sobolev-characterization}
Let \(1\leq p<\infty\) and \(k\in\N\). Then
\[
W_p^k(\R_\theta^d)
=
\left\{
\x\in L_p(\R_\theta^d):
\partial^\beta \x\in L_p(\R_\theta^d)
\text{ for every }|\beta|=k
\right\}.
\]
Moreover, for every \(\x\in W_p^k(\R_\theta^d)\),
\begin{equation}\label{eq:in-hom-norm-equivalent}
\|\x\|_{W_p^k(\R_\theta^d)}
\approx_{d,k}
\|\x\|_{L_p(\R_\theta^d)}
+
|\x|_{\dot W_p^k(\R_\theta^d)}.
\end{equation}
\end{corollary}
\begin{proof}
One inclusion follows directly from the definition of
\(W_p^k(\R_\theta^d)\). For the converse, the case \(k=1\) is
immediate. Assume \(k\geq2\) and
\[
\x\in L_p(\R_\theta^d),
\qquad
\partial^\beta\x\in L_p(\R_\theta^d)
\quad\text{for every }|\beta|=k.
\]

Choose \(\psi\in\mcs(\R^d)\) such that
\(
\int_{\R^d}\psi(h)\,\mathrm dh=1,
\)
and set
\[
\psi_\varepsilon(h)
:=
\varepsilon^{-d}\psi(h/\varepsilon),
\qquad
\x_\varepsilon
:=
\psi_\varepsilon*\x,
\]
 By \eqref{eq:conv-schw-commut-distri},
\[
\partial^\gamma\x_\varepsilon
=
(\partial^\gamma\psi_\varepsilon)*\x
\qquad
\text{in }\mcs'(\R_\theta^d)
\]
for every \(\gamma\in\mathbb N_0^d\).  
Since
\(\partial^\gamma\psi_\varepsilon\in L_1(\R^d)\),
\eqref{eq:conv-Lp-bound} implies that
\(\x_\varepsilon\in W_p^r(\R_\theta^d)\) for every \(r\in\mathbb N\).
For \(|\beta|=k\), one also has
\[
\partial^\beta\x_\varepsilon
=
\psi_\varepsilon*(\partial^\beta\x).
\]
The \(L_p\)-approximation property of
\((\psi_\varepsilon)_{\varepsilon>0}\)
\cite[Theorem~3.10]{MSX20} therefore gives
\begin{equation}\label{eq:mollification-top-order-convergence}
\x_\varepsilon\longrightarrow\x
\quad\text{in }L_p(\R_\theta^d),
\qquad
\partial^\beta\x_\varepsilon
\longrightarrow\partial^\beta\x
\quad\text{in }L_p(\R_\theta^d)
\end{equation}
for every \(|\beta|=k\) as \(\varepsilon\to 0\).

Let \(1\leq m<k\). Applying
Lemma~\ref{lem:Gagliardo-Nirenberg} to
\(\x_\varepsilon-\x_\delta\), we obtain
\[
|\x_\varepsilon-\x_\delta|_{\dot W_p^m}
\lesssim_{d,m,k}
\|\x_\varepsilon-\x_\delta\|_p^{1-m/k}
|\x_\varepsilon-\x_\delta|_{\dot W_p^k}^{m/k}.
\]
It follows from
\eqref{eq:mollification-top-order-convergence} that
\[
|\x_\varepsilon-\x_\delta|_{\dot W_p^m}
\longrightarrow0,\quad \text{as}\quad \varepsilon, \delta \to 0.
\]
Hence, for every \(\gamma\) with \(|\gamma|=m\),
\((\partial^\gamma\x_\varepsilon)_{\varepsilon>0}\) is Cauchy in
\(L_p(\R_\theta^d)\). Let
\[
\y_\gamma
:=
\lim_{\varepsilon\to 0}
\partial^\gamma\x_\varepsilon
\quad\text{in }L_p(\R_\theta^d).
\]
The continuous embedding
\(L_p(\R_\theta^d)\hookrightarrow\mcs'(\R_\theta^d)\) gives
\[
\partial^\gamma\x_\varepsilon
\longrightarrow
\y_\gamma
\quad\text{in }\mcs'(\R_\theta^d)\qquad
(\varepsilon\to 0).
\]
By \eqref{eq:mollification-top-order-convergence},
\(\x_\varepsilon\to\x\) in \(\mcs'(\R_\theta^d)\), and hence
\[
\partial^\gamma\x_\varepsilon
\longrightarrow
\partial^\gamma\x
\quad\text{in }\mcs'(\R_\theta^d)
\qquad
(\varepsilon\to 0).
\]
Therefore,
\[
\partial^\gamma\x
=
\y_\gamma
\in L_p(\R_\theta^d),
\qquad
1\leq|\gamma|<k.
\]
Thus all distributional derivatives of \(\x\) of order at most \(k\)
belong to \(L_p(\R_\theta^d)\), and hence
\(\x\in W_p^k(\R_\theta^d)\).

It remains to establish \eqref{eq:in-hom-norm-equivalent}. By
Lemma~\ref{lem:Gagliardo-Nirenberg} and the weighted
arithmetic--geometric mean inequality,
\[
\begin{aligned}
|\x|_{\dot W_p^m(\R_\theta^d)}
&\lesssim_{d,m,k}
\|\x\|_{L_p(\R_\theta^d)}^{1-m/k}
|\x|_{\dot W_p^k(\R_\theta^d)}^{m/k}
&\lesssim_{d,m,k}
\|\x\|_{L_p(\R_\theta^d)}
+
|\x|_{\dot W_p^k(\R_\theta^d)},
\qquad
1\leq m<k.
\end{aligned}
\]
Summing over \(m=1,\ldots,k-1\), we obtain
\[
\begin{aligned}
\|\x\|_{W_p^k(\R_\theta^d)}
=
\|\x\|_{L_p(\R_\theta^d)}
+
\sum_{m=1}^k
|\x|_{\dot W_p^m(\R_\theta^d)}
\lesssim_{d,k}
\|\x\|_{L_p(\R_\theta^d)}
+
|\x|_{\dot W_p^k(\R_\theta^d)}.
\end{aligned}
\]
The reverse estimate follows from the definition of
\(\|\cdot\|_{W_p^k(\R_\theta^d)}\).
\end{proof}

\begin{remark}\label{rem:Sobolev-intersection}
Let \(1\leq p<\infty\) and \(k\in\N\). Since
\(\mcs(\R_\theta^d)\) is dense in \(W_p^k(\R_\theta^d)\) and
\[
|\x|_{\dot W_p^k(\R_\theta^d)}
\leq
\|\x\|_{W_p^k(\R_\theta^d)},
\qquad \x\in\mcs(\R_\theta^d),
\]
the identity map on \(\mcs(\R_\theta^d)\) extends uniquely to a
continuous linear map
\[
\iota_k:
W_p^k(\R_\theta^d)
\longrightarrow
\dot W_{p,0}^k(\R_\theta^d).
\]
The map \(\iota_k\) satisfies
\begin{equation}\label{eq:canonical-Sobolev-embedding}
\|\iota_k(\x)\|_{\dot W_{p,0}^k(\R_\theta^d)}
=
|\x|_{\dot W_p^k(\R_\theta^d)},
\qquad
\x\in W_p^k(\R_\theta^d),
\end{equation}
and, with \(\jmath_0\) and \(\jmath_1\) as in the proof of
Lemma~\ref{lem:compatible-homogeneous-couple},
\[
\jmath_1\circ\iota_k
=
\jmath_0\big|_{W_p^k(\R_\theta^d)}.
\]
Since \(\jmath_0\) is injective, so is \(\iota_k\). We therefore
identify \(W_p^k(\R_\theta^d)\) with its image under
\(\iota_k\) and write
\[
W_p^k(\R_\theta^d)
\hookrightarrow
\dot W_{p,0}^k(\R_\theta^d).
\]

With this identification,
\begin{equation}\label{eq:Sobolev-intersection}
W_p^k(\R_\theta^d)
=
L_p(\R_\theta^d)
\cap
\dot W_{p,0}^k(\R_\theta^d),
\end{equation}
where the intersection is understood in \(\mathcal Y_k\). Indeed, if
\(
\x\in
L_p(\R_\theta^d)
\cap
\dot W_{p,0}^k(\R_\theta^d),
\)
then its representative in
\(\dot W_{p,0}^k(\R_\theta^d)\) coincides with
\(
\bigl(\partial^\beta \x\bigr)_{|\beta|=k}
\)
in the sense of distributions.
Hence
\(\partial^\beta\x\in L_p(\R_\theta^d)\) for every \(|\beta|=k\), and
Corollary~\ref{cor:top-order-Sobolev-characterization} gives
\(\x\in W_p^k(\R_\theta^d)\).
\end{remark}

We conclude this subsection by characterizing the homogeneous
Sobolev seminorm in terms of moduli of smoothness. The analogous
results on \(\R^d\) and \(\mathbb T_\theta^d\) are given in
\cite[Proposition~2.4]{KL05} and
\cite[Theorem~2.20]{XXY18}, respectively. The following theorem
establishes the result on \(\R_\theta^d\).

\begin{theorem}\label{th3.7}
Let \(k\in\N\) and \(1\leq p<\infty\). Then, for every
\(\x\in W_p^k(\R_\theta^d)\),
\begin{equation}\label{eq:Sobolev-Lip}
\lim_{\varepsilon\downarrow0}
\frac{\omega_p^k(\x,\varepsilon)}{\varepsilon^k}
=
\sup_{\varepsilon>0}
\frac{\omega_p^k(\x,\varepsilon)}{\varepsilon^k}
\approx
|\x|_{\dot W_p^k(\R_\theta^d)},
\end{equation}
with equivalence constants depending only on \(d\) and \(k\).
\end{theorem}

We require the following two auxiliary lemmas for the proof.

\begin{lemma}\label{lemma5}
Let $1\leq p<\infty$, $k \in \mathbb{N}$, and $\x \in W^k_p(\mathbb{R}_\theta^d)$. Then the function
\(
\varepsilon
\longmapsto
\varepsilon^{-k}\omega_p^k(\x,\varepsilon)
\)
is continuous on \((0,\infty)\), and
\[
\sup_{\varepsilon > 0} \frac{\omega_p^k(\x,\varepsilon)}{\varepsilon^k} = \sup_{0 <\varepsilon\leq 1} \frac{\omega_p^k(\x,\varepsilon)}{\varepsilon^k} = \lim_{\varepsilon \to 0} \frac{\omega_p^k(\x,\varepsilon)}{\varepsilon^k}.
\]
\end{lemma}

\begin{proof}
  Set
\[
A(\varepsilon)
:=
\varepsilon^{-k}\omega_p^k(\x,\varepsilon),
\qquad
\varepsilon>0.
\]
By Remark~\ref{rem:modulus-properties}, \(A\) is continuous on
\((0,\infty)\).
Moreover, \eqref{eq:moduli-integer-scaling} gives 
\begin{equation}\label{eq:A-integer-scaling}
\begin{aligned}
A(n\varepsilon)
&=
(n\varepsilon)^{-k}
\omega_p^k(\x,n\varepsilon)
&\leq
\varepsilon^{-k}
\omega_p^k(\x,\varepsilon)
=
A(\varepsilon),
\qquad
n\in\N,\quad \varepsilon>0.
\end{aligned}
\end{equation}
For \(\tau>0\), choose \(n\in\N\) such that \(n\geq\tau\), and set
\(s:=\tau/n\in(0,1]\). Then
\[
A(\tau)
=
A(ns)
\leq
A(s)
\leq
\sup_{0<\varepsilon\leq1}A(\varepsilon).
\]
Taking the supremum over \(\tau>0\), we obtain
\[
\sup_{\varepsilon>0}A(\varepsilon)
=
\sup_{0<\varepsilon\leq1}A(\varepsilon).
\]

It remains to prove that
\[
\lim_{\varepsilon\downarrow0}A(\varepsilon)
=
\sup_{0<\varepsilon\leq1}A(\varepsilon).
\]
Fix \(\varepsilon_0\in(0,1]\), and choose a sequence
\((\varepsilon_j)_{j\geq1}\) such that
\[
\varepsilon_j\downarrow0,
\quad \text{and}\quad
\lim_{j\to \infty}A(\varepsilon_j)
=
\liminf_{\varepsilon\downarrow0}A(\varepsilon).
\]
For all sufficiently large \(j\), set
\(
n_j
:=
\left\lfloor
\frac{\varepsilon_0}{\varepsilon_j}
\right\rfloor
\geq1.
\)
Then
\[
n_j\varepsilon_j
\leq
\varepsilon_0
<
(n_j+1)\varepsilon_j,
\]
and hence \(n_j\varepsilon_j\to\varepsilon_0\). By
\eqref{eq:A-integer-scaling} and the continuity of \(A\),
\[
A(\varepsilon_0)
=
\lim_{j\to\infty}A(n_j\varepsilon_j)
\leq
\lim_{j\to\infty}A(\varepsilon_j)
=
\liminf_{\varepsilon\downarrow0}A(\varepsilon).
\]
Since \(\varepsilon_0\in(0,1]\) was arbitrary,
\[
\sup_{0<\varepsilon\le 1} A(\varepsilon) \le \liminf_{\tau\downarrow 0} A(\tau) \le \limsup_{\tau\downarrow 0} A(\tau) \le \sup_{0<\varepsilon\le 1} A(\varepsilon), 
\]
so the limit exists and equals the supremum.
\end{proof}

The next lemma relates the mixed difference $\Delta_{h_1}\cdots\Delta_{h_k}$ to the iterated difference $\Delta_h^k$.
The corresponding classical identity is given in
\cite[Lemma~5.4.11]{BS88}.

\begin{lemma}\label{lem3.6}
Let $h_1,\ldots,h_k\in\mathbb{R}^d$. Then
\begin{equation}\label{eq:mixed-to-iterated}
\Delta_{h_1}\cdots\Delta_{h_k}
= \sum_{D\subseteq\{1,\ldots,k\}} (-1)^{|D|}\,T_{h^*_D}\,\Delta_{h_D}^k,
\end{equation}
where the sum runs over all subsets $D$ of $\{1,\ldots,k\}$, $|D|$ is the cardinality of $D$, and
\[
h^*_D = \sum_{j\in D} h_j, \qquad h_D = -\sum_{j\in D} j^{-1} h_j.
\]
Furthermore, for $\x\in L_p(\mathbb{R}^d_\theta)$ with $1\leq p \leq \infty$,
\begin{equation}\label{eq:mixed-equiv}
\sup_{\substack{|h_1|\le\varepsilon,\,\ldots,\,|h_k|\le\varepsilon}}
\bigl\|\Delta_{h_1}\cdots\Delta_{h_k}\, \x\bigr\|_p
\approx \omega_p^k(\x,\varepsilon),\quad \varepsilon>0.
\end{equation}
\end{lemma}

\begin{proof}
The proof of the classical identity
\cite[Lemma~5.4.11]{BS88} is algebraic and uses only
\[
\Delta_h=T_h-I
\qquad\text{and}\qquad
T_hT_s=T_{h+s}=T_sT_h,
\qquad
h,s\in\R^d.
\]
Since these identities also hold for the quantum translations, the
same argument gives \eqref{eq:mixed-to-iterated}.

Suppose that
\(
|h_j|\leq\varepsilon,\)
\(
j=1,\ldots,k.
\)
For every \(D\subseteq\{1,\ldots,k\}\),
\[
|h_D|
\leq
\sum_{j\in D}\frac{|h_j|}{j}
\leq
k\varepsilon.
\]
Taking \(L_p\)-norms in \eqref{eq:mixed-to-iterated} and using the
\(L_p\)-isometry of the translations, we obtain
\[
\|\Delta_{h_1}\cdots\Delta_{h_k}\x\|_p
\leq
\sum_{D\subseteq\{1,\ldots,k\}}
\|\Delta_{h_D}^k\x\|_p
\leq
2^k\omega_p^k(\x,k\varepsilon)
\leq
(2k)^k\omega_p^k(\x,\varepsilon).
\]
where the last inequality follows from
\eqref{eq:moduli-integer-scaling}. Taking the supremum over
\(|h_1|,\ldots,|h_k|\leq\varepsilon\) gives one direction of
\eqref{eq:mixed-equiv}.
The reverse estimate follows by taking
\(h_1=\cdots=h_k=h\).
\end{proof}

\begin{proof}[Proof of Theorem \ref{th3.7}]
 For \(y\in W_p^1(\R_\theta^d)\), applying the fundamental theorem of
calculus to \(t\mapsto T_{th}y\) gives
\begin{align}\label{eq:one-diff-identity}
\Delta_hy
=
\int_0^1 T_{th}(h\cdot\nabla_\theta y)\,\mathrm dt,
\end{align}
where
\(
h\cdot\nabla_\theta
:=
\sum_{j=1}^d h_j\partial_j
\). Iterating \eqref{eq:one-diff-identity}, we obtain, for
\(\x\in W_p^k(\R_\theta^d)\),
\[
\Delta_h^k \x
=
\int_{[0,1]^k}
T_{(t_1+\cdots+t_k)h}
(h\cdot\nabla_\theta)^k \x
\,\mathrm dt_1\cdots\mathrm dt_k.
\]
Minkowski's integral inequality and the \(L_p\)-isometry of the
translations yield
\[
\|\Delta_h^k \x\|_p
\leq
\|(h\cdot\nabla_\theta)^k \x\|_p.
\]
By the multinomial formula,
\[
(h\cdot\nabla_\theta)^k \x
=
\sum_{|\alpha|=k}
\binom{k}{\alpha}
h^\alpha\partial^\alpha \x.
\]
Since \(|h^\alpha|\leq |h|^k\), it follows that
\begin{equation}\label{eq:kth-difference-Sobolev-bound}
\|\Delta_h^k \x\|_p
\leq
\|(h\cdot\nabla_\theta)^k \x\|_p
\leq
k!\,|h|^k
|\x|_{\dot W_p^k(\R_\theta^d)}.
\end{equation}
Consequently,
\[
\omega_p^k(\x,\varepsilon)
\leq
k!\,\varepsilon^k
|\x|_{\dot W_p^k(\R_\theta^d)},
\qquad
\varepsilon>0,
\]
and hence
\[
\sup_{\varepsilon>0}
\frac{\omega_p^k(\x,\varepsilon)}{\varepsilon^k}
\lesssim
|\x|_{\dot W_p^k(\R_\theta^d)}.
\] 

For the converse estimate, fix
\(\alpha\in\mathbb N_0^d\) with \(|\alpha|=k\), and write
\[
\partial^\alpha
=
\partial_{j_1}\cdots\partial_{j_k}
\]
for suitable \(j_1,\ldots,j_k\in\{1,\ldots,d\}\).
Repeated application of \eqref{eq:one-diff-identity} gives
\[
\varepsilon^{-k}
\Delta_{\varepsilon\mathbf e_{j_1}}\cdots
\Delta_{\varepsilon\mathbf e_{j_k}}\x
=
\int_{[0,1]^k}
T_{\varepsilon\sum_{\ell=1}^k
t_\ell\mathbf e_{j_\ell}}
\partial^\alpha\x
\,\mathrm dt_1\cdots\mathrm dt_k.
\]
Hence, by Minkowski's integral inequality,
\[
\begin{aligned}
&\left\|
\varepsilon^{-k}
\Delta_{\varepsilon\mathbf e_{j_1}}\cdots
\Delta_{\varepsilon\mathbf e_{j_k}}\x
-
\partial^\alpha \x
\right\|_p\\
&\qquad\leq
\int_{[0,1]^k}
\left\|
\left(
T_{\varepsilon\sum_{\ell=1}^k
t_\ell\mathbf e_{j_\ell}}-I
\right)
\partial^\alpha \x
\right\|_p
\,\mathrm dt_1\cdots\mathrm dt_k
\longrightarrow0
\end{aligned}
\]
as \(\varepsilon\downarrow0\). Here the convergence follows from the
strong continuity of the translations and the dominated convergence
theorem, since the integrand is bounded by
\(2\|\partial^\alpha\x\|_p\).
 By Lemma~\ref{lem3.6},
\[
\left\|
\Delta_{\varepsilon\mathbf e_{j_1}}\cdots
\Delta_{\varepsilon\mathbf e_{j_k}}\x
\right\|_p
\lesssim_k
\omega_p^k(\x,\varepsilon).
\]
It follows that
\[
\begin{aligned}
\|\partial^\alpha\x\|_p
&=
\lim_{\varepsilon\downarrow0}
\varepsilon^{-k}
\left\|
\Delta_{\varepsilon\mathbf e_{j_1}}\cdots
\Delta_{\varepsilon\mathbf e_{j_k}}\x
\right\|_p
&\lesssim_k
\liminf_{\varepsilon\downarrow0}
\frac{\omega_p^k(\x,\varepsilon)}{\varepsilon^k}.
\end{aligned}
\]
Summing over all \(\alpha\) with \(|\alpha|=k\), we obtain
\[
|\x|_{\dot W_p^k(\R_\theta^d)}
\lesssim_{d,k}
\liminf_{\varepsilon\downarrow0}
\frac{\omega_p^k(\x,\varepsilon)}{\varepsilon^k}
\leq
\sup_{\varepsilon>0}
\frac{\omega_p^k(\x,\varepsilon)}{\varepsilon^k}.
\]
Together with the preceding upper estimate and
Lemma~\ref{lemma5}, this proves \eqref{eq:Sobolev-Lip}.
 \end{proof}

\subsection{Besov spaces on \(\R_\theta^d\) and their characterization
via differences}
\label{section3.2}

We next recall the Littlewood-Paley definition of the inhomogeneous
Besov spaces on \(\R_\theta^d\); see, for example,
\cite{CH2026,CHWW2025,ME24}. Let \(\varphi\in C_c^\infty(\R^d)\) satisfy
\[
\operatorname{supp}\varphi
\subset
\left\{
\xi\in\R^d:
\frac12\leq|\xi|\leq2
\right\},
\qquad
\sum_{j\in\Z}\varphi(2^{-j}\xi)=1,
\quad \xi\neq0.
\]
For \(j\geq1\), set
\[
\varphi^{(j)}(\xi):=\varphi(2^{-j}\xi),
\qquad
\varphi^{(0)}(\xi)
:=
1-\sum_{j\geq1}\varphi(2^{-j}\xi).
\]
Then
\[
\sum_{j=0}^\infty\varphi^{(j)}(\xi)=1,
\qquad \xi\in\R^d,
\]
and
\[
\operatorname{supp}\varphi^{(j)}
\subset
\left\{
\xi\in\R^d:
2^{j-1}\leq|\xi|\leq2^{j+1}
\right\},
\qquad j\geq1.
\]
For \(j\in\N_0\), let \(\varphi^{(j)}(D)\) denote the Fourier
multiplier on \(\mcs'(\R_\theta^d)\) with symbol
\(\varphi^{(j)}\).

\begin{definition}\label{def:Besov}
Let \(\alpha\in\R\) and \(1\leq p,q\leq\infty\). The inhomogeneous
Besov space \(B_{p,q}^\alpha(\R_\theta^d)\) consists of all
\(\x\in\mcs'(\R_\theta^d)\) such that \(\|\x\|_{B^\alpha_{p,q}} <\infty\), where
\[
\|\x\|_{B^\alpha_{p,q}} := \left(
\sum_{j=0}^\infty
2^{j\alpha q}
\|\varphi^{(j)}(D)\x\|_p^q
\right)^{1/q},\qquad \text{for}\quad q<\infty,
\]
and
\[
\|\x\|_{B^\alpha_{p,\infty}} :=\sup_{j\geq0}
2^{j\alpha}\|\varphi^{(j)}(D)\x\|_p.
\]
\end{definition}

The definition of \(B_{p,q}^\alpha(\R_\theta^d)\) is independent of
the choice of \(\varphi\), in the sense that different admissible
choices yield equivalent norms. Moreover,
\(B_{p,q}^\alpha(\R_\theta^d)\) is a Banach space, and
\(\mcs(\R_\theta^d)\) is dense in it whenever \(p,q<\infty\); see
\cite[Lemma~3.15]{ME24}.

We next introduce the difference norms used throughout this paper and
recall their equivalence with the Littlewood--Paley norms defined
above. Let
\(m\in\N\), \(0<\alpha<m\), \(1\leq p,q\leq\infty\), and
\(\x\in L_p(\R_\theta^d)\). With \(\omega_p^m\) as in
Definition~\ref{def:difference}, define
\begin{equation}\label{eq:difference-Besov-seminorm}
|\x|_{B_{p,q}^{\alpha,m}(\R_\theta^d)}
:=
\begin{cases}
\displaystyle
\left(
\int_0^\infty
\bigl[
\varepsilon^{-\alpha}
\omega_p^m(\x,\varepsilon)
\bigr]^q
\,\frac{\mathrm d\varepsilon}{\varepsilon}
\right)^{1/q},
& 1\leq q<\infty,\\[2.2ex]
\displaystyle
\sup_{\varepsilon>0}
\varepsilon^{-\alpha}
\omega_p^m(\x,\varepsilon),
& q=\infty,
\end{cases}
\end{equation}
and 
\begin{equation}\label{eq:difference-Besov-norm}
\|\x\|_{B_{p,q}^{\alpha,m}(\R_\theta^d)}
:=
\|\x\|_{L_p(\R_\theta^d)}
+
|\x|_{B_{p,q}^{\alpha,m}(\R_\theta^d)}.
\end{equation}

The following characterization is obtained from
\cite[Theorem~3.6]{CH2026} by taking \(N=0\); the finite sum over the
coordinate directions in that theorem is absorbed into the
equivalence constants.

\begin{lemma}
\label{lem:Besov-difference-characterization}
Let \(1\leq p,q\leq\infty\), \(\alpha>0\), and let
\(m\in\mathbb N\) satisfy \(m>\alpha\). Then, for every
\(\x\in B_{p,q}^\alpha(\R_\theta^d)\),
\begin{equation}\label{eq:Besov-difference-characterization}
\|\x\|_{B_{p,q}^\alpha(\R_\theta^d)}
\approx
\|\x\|_{B_{p,q}^{\alpha,m}(\R_\theta^d)}.
\end{equation}
The implicit constants are independent of \(\x\).
\end{lemma}

For later use, we record the following Marchaud
inequality relating moduli of smoothness of different orders.

\begin{proposition}
\label{prop:Marchaud}
Let \(1\leq p\leq\infty\), \(\x\in L_p(\R_\theta^d)\), and let
\(0<n<k\) be integers. Then, for every \(\varepsilon>0\),
\begin{equation}\label{eq:Marchaud}
2^{n-k}\omega_p^k(\x,\varepsilon)
\leq
\omega_p^n(\x,\varepsilon)
\lesssim_{n,k}
\varepsilon^n
\int_\varepsilon^\infty
\frac{\omega_p^k(\x,\delta)}{\delta^n}
\,\frac{\mathrm d\delta}{\delta}.
\end{equation}
\end{proposition}

\begin{proof}
The left-hand inequality follows immediately from
\eqref{eq:modulus-monotone-order}. For the reverse inequality, observe
that
\[
 d_{2h}^n
 = 2^n d_h^n
   + \Bigl[(e_{h}+1)^n - 2^n\Bigr] d_h^n
 = 2^n d_h^n + \sum_{j=1}^n \binom{n}{j} \sum_{\ell=0}^{j-1} e_{\ell h} d_h^{n+1},
 \]
 which gives
\begin{equation}\label{eq:Marchaud-identity}
\Delta_h^n = 2^{-n}\Bigl\{\Delta_{2h}^n - \sum_{j=1}^n\binom{n}{j}\sum_{\ell=0}^{j-1}\Delta_h^{n+1}T_{\ell h}\Bigr\}.
\end{equation}

 Taking $L_p$-norms in \eqref{eq:Marchaud-identity} and using the
\(L_p\)-isometry of the translations, we obtain
\begin{equation}\label{eq:Marchaud-recursion}
\omega_p^n(\x,\varepsilon) \le \tfrac{n}{2}\,\omega_p^{n+1}(\x,\varepsilon) + 2^{-n}\,\omega_p^n(\x,2\varepsilon).
\end{equation}
Iterating \eqref{eq:Marchaud-recursion} $N$ times gives
\[
\omega_p^n(\x,\varepsilon) \le \frac{n}{2}\sum_{j=0}^{N-1}2^{-jn}\,\omega_p^{n+1}(\x,2^j\varepsilon) + 2^{-Nn}\,\omega_p^n(\x,2^N\varepsilon).
\]
Since
\(\omega_p^n(\x,\delta)\leq2^n\|\x\|_p\), the last term tends to zero
as \(N\to\infty\). Moreover, by the monotonicity of the modulus,
\[
\begin{aligned}
\varepsilon^n
\int_{2^j\varepsilon}^{2^{j+1}\varepsilon}
\frac{\omega_p^{n+1}(\x,\delta)}{\delta^n}
\,\frac{\mathrm d\delta}{\delta}
\geq
2^{-n(j+1)}(\log2)
\omega_p^{n+1}(\x,2^j\varepsilon),\quad  j\geq0.
\end{aligned}
\]
Summing over \(j\geq0\), we obtain
\begin{equation}\label{eq:Marchaud-one-step}
\omega_p^n(\x,\varepsilon) \le \frac{n2^{n-1}}{\log 2}\, \varepsilon^n \int_\varepsilon^\infty\frac{\omega_p^{n+1}(\x,\delta)}{\delta^n}\,\frac{\mathrm{d}\delta}{\delta}.
\end{equation}
Iterating \eqref{eq:Marchaud-one-step} from \(n\) to \(k-1\) and
applying Tonelli's theorem yields
\[
\omega_p^n(\x,\varepsilon)
\lesssim_{n,k}
\varepsilon^n
\int_\varepsilon^\infty
\frac{(\delta-\varepsilon)^{k-n-1}}{\delta^k}
\omega_p^k(\x,\delta)
\,\mathrm d\delta.
\]
Since
\(
(\delta-\varepsilon)^{k-n-1}
\leq
\delta^{k-n-1},
\)
we conclude that
\[
\omega_p^n(\x,\varepsilon)
\lesssim_{n,k}
\varepsilon^n
\int_\varepsilon^\infty
\frac{\omega_p^k(\x,\delta)}{\delta^n}
\,\frac{\mathrm d\delta}{\delta},
\]
which proves \eqref{eq:Marchaud}.
\end{proof}

\section{BBM and MS limits for higher-order Besov norms on
\(\R_\theta^d\)} \label{section5}

In this section, we prove Theorems~\ref{th1.2}
and \ref{th1.3} by the real interpolation method and conclude with the corresponding limits
for the inhomogeneous Sobolev couple.

\subsection{\(K\)-functional characterizations and real interpolation}
\label{subsec:difference-interpolation}

We begin with two elementary observations showing that the normalized
BBM limit determined entirely by the asymptotic behavior of the modulus of smoothness as
$\varepsilon\downarrow0$, whereas the normalized MS limit is
determined by its behavior as \(\varepsilon\to\infty\).

At the BBM endpoint, the contribution from
scales $\varepsilon\geq1$ vanishes after normalization. Indeed, since
\(\omega_p^k(\x,\varepsilon)\leq 2^k\|\x\|_p\), we have
\begin{align}\label{eq:BBM-large-scale-bound}
0
&\leq
(k-\alpha)
\int_1^\infty
\varepsilon^{-\alpha q}
\omega_p^k(\x,\varepsilon)^q
\,\frac{\mathrm d\varepsilon}{\varepsilon}
\nonumber\\
&\qquad\qquad\leq
2^{kq}(k-\alpha)\|\x\|_p^q
\int_1^\infty
\varepsilon^{-\alpha q}
\,\frac{\mathrm d\varepsilon}{\varepsilon}
=
\frac{2^{kq}(k-\alpha)}{\alpha q}\|\x\|_p^q\longrightarrow0
\qquad (\alpha\uparrow k).
\end{align}
This is in contrast to the MS endpoint as 
$\alpha\downarrow0$, where instead the large-scale part provides the leading
contribution. More precisely, if
$\x\in B_{p,q}^{\alpha_0}(\R_\theta^d)$ for some $\alpha_0\in(0,k)$, then, for $0<\alpha<\alpha_0$,
\begin{align}\label{eq:ms-small-scale}
0
\leq
\alpha\int_0^1
\varepsilon^{-\alpha q}
\omega_p^k(\x,\varepsilon)^q
\,\frac{\mathrm{d}\varepsilon}{\varepsilon}
\leq
\alpha\int_0^1
\varepsilon^{-\alpha_0 q}
\omega_p^k(\x,\varepsilon)^q
\,\frac{\mathrm{d}\varepsilon}{\varepsilon}
\longrightarrow0
\qquad (\alpha\downarrow0).
\end{align}
A direct analysis of the remaining integrals gives the following
exact identities:
\begin{enumerate}[label=\textnormal{(\roman*)}]
\item
For every \(\x\in W_p^k(\R_\theta^d)\),
\begin{equation}\label{eq:exact-BBM-modulus-limit}
\lim_{\alpha\uparrow k}
(k-\alpha)^{1/q}
\|\x\|_{B_{p,q}^{\alpha,k}(\R_\theta^d)}
=
q^{-1/q}
\lim_{\varepsilon\downarrow0}
\frac{\omega_p^k(\x,\varepsilon)}{\varepsilon^k};
\end{equation}

\item If
\(\x\in B_{p,q}^{\alpha_0}(\R_\theta^d)\) for some
\(\alpha_0\in(0,k)\), then
\begin{equation}\label{eq:exact-MS-modulus-limit}
\lim_{\alpha\downarrow0}
\alpha^{1/q}
\|\x\|_{B_{p,q}^{\alpha,k}(\R_\theta^d)}
=
q^{-1/q}\lim_{\varepsilon\to\infty}
\omega_p^k(\x,\varepsilon),
\end{equation}
\end{enumerate}
see Appendix~\ref{app:exact-endpoint-asymptotics} for the detailed derivation of
\eqref{eq:exact-BBM-modulus-limit} and
\eqref{eq:exact-MS-modulus-limit}. Moreover, the limit on the right-hand side of
\eqref{eq:exact-BBM-modulus-limit} exists by Theorem~\ref{th3.7}, while
\(\omega_p^k(\x,\infty)\) exists because
\(\varepsilon\mapsto\omega_p^k(\x,\varepsilon)\) is nondecreasing
and bounded by \(2^k\|\x\|_p\). Hence, both normalized limits exist and their exact values are
determined by the corresponding endpoint behavior of the modulus of
smoothness.

The identities \eqref{eq:exact-BBM-modulus-limit} and
\eqref{eq:exact-MS-modulus-limit} reduce the proofs of
Theorems~\ref{th1.2} and~\ref{th1.3} to identifying their right-hand
sides, up to equivalence, with the homogeneous Sobolev seminorm and
the \(L_p\)-norm, respectively. These estimates also underlie the
real-interpolation approach. Once the \(K\)-functional of the compatible couple
\(
\bigl(
L_p(\mathbb R_\theta^d),
\dot W_{p,0}^k(\mathbb R_\theta^d)
\bigr)
\) 
is characterized, these two identifications enter the verification of the
\(X_1\)- and \(X_0\)-normality conditions in Milman's theorem.
We can then apply
Milman's theorem to treat the BBM and MS limits within a common
real-interpolation framework.

We first recall the relevant notions from real interpolation and
Milman's theorem.
Let $\vec X=(X_0,X_1)$ be a compatible Banach couple. Its
$K$-functional is defined by
\[
K(\x,\varepsilon;\vec X) = \inf\bigl\{\|\x_0\|_{X_0}+\varepsilon\,\|\x_1\|_{X_1}: \x=\x_0+\x_1,\\ \x_i\in X_i,\,i=0,1\bigr\},\quad \varepsilon>0.
\]
For $0<\eta<1$ and $1\leq q<\infty$, the real interpolation space
$\vec X_{\eta,q}=(X_0,X_1)_{\eta,q}$ is equipped with the norm
\[
\|\x\|_{\vec X_{\eta,q}}
:=
\left(
\int_0^\infty
\bigl[\varepsilon^{-\eta}K(\x,\varepsilon;\vec X)\bigr]^q
\,\frac{\mathrm{d}\varepsilon}{\varepsilon}
\right)^{1/q},
\]
with the usual supremum modification when $q=\infty$; see \cite{BS88,BJ76} for background on interpolation theory.

Since $\varepsilon \mapsto K(\x,\varepsilon;\vec X)$ is increasing and
$\varepsilon\mapsto \varepsilon^{-1}K(\x,\varepsilon;\vec X)$ is decreasing, the following
extended limits exist:
\begin{align}\label{eq:5.9}
\|\x\|_{\widetilde X_0}
:=
\lim_{\varepsilon\to\infty}K(\x,\varepsilon;\vec X),
\qquad
\|\x\|_{\widetilde X_1}
:=
\lim_{\varepsilon\downarrow0}\frac{K(\x,\varepsilon;\vec X)}{\varepsilon}.
\end{align}
The associated Gagliardo closures of \(X_0\) and \(X_1\) are defined by
\[
\widetilde X_i
:=
\bigl\{
\x\in X_0+X_1:
\|\x\|_{\widetilde X_i}<\infty
\bigr\},
\qquad i=0,1.
\]
Following \cite[Definition~1 and Remark~1]{MM05}, we shall say that the couple
$\vec X$ is ``normal'' if
\begin{align}\label{eq:normal}
\|\x\|_{\widetilde X_i}
\approx
\|\x\|_{X_i},
\qquad \x\in X_i,\quad i=0,1.
\end{align}

The following endpoint continuity theorem of Milman
\cite[Theorem~1 and Remark~3]{MM05} provides the abstract framework
for the two limiting arguments.

\begin{theorem}\label{thm:Milman}
Let $\vec X=(X_0,X_1)$ be a compatible Banach couple, and let
$1\leq q<\infty$.
\begin{enumerate}[label=\textnormal{(\roman*)}]
\item
If \eqref{eq:normal} holds for $i=1$, then, for every $\x\in X_0\cap X_1$,
\begin{equation}\label{eq:Milman-upper-endpoint}
\lim_{\eta\uparrow1}
\bigl(\eta(1-\eta)\bigr)^{1/q}
\|\x\|_{\vec X_{\eta,q}}
=
q^{-1/q}\|\x\|_{\widetilde X_1}
\approx
q^{-1/q}\|\x\|_{X_1}.
\end{equation}

\item
If \eqref{eq:normal} holds for $i=0$, then, for every
\(
\x\in
X_0\cap
\bigcup_{0<\eta<1}\vec X_{\eta,q},
\)
we have
\begin{equation}\label{eq:Milman-lower-endpoint}
\lim_{\eta\downarrow0}
\bigl(\eta(1-\eta)\bigr)^{1/q}
\|\x\|_{\vec X_{\eta,q}}
=
q^{-1/q}\|\x\|_{\widetilde X_0}
\approx
q^{-1/q}\|\x\|_{X_0}.
\end{equation}
\end{enumerate}
\end{theorem}
\vspace{0.3cm}

For \(1\leq p<\infty\) and \(k\in\mathbb N\), set
\[
\vec X_{\mathrm{inhom}}
:=
\bigl(
L_p(\R_\theta^d),
W_p^k(\R_\theta^d)
\bigr),
\qquad
\vec X_{\mathrm{hom}}
:=
\bigl(
L_p(\R_\theta^d),
\dot W_{p,0}^k(\R_\theta^d)
\bigr).
\]
Since \(W_p^k(\R_\theta^d)\) embeds continuously into
\(L_p(\R_\theta^d)\), the couple
\(\vec X_{\mathrm{inhom}}\) is compatible, while the compatibility of
\(\vec X_{\mathrm{hom}}\) follows from
Lemma~\ref{lem:compatible-homogeneous-couple}. 
To implement the interpolation scheme above, we first characterize
the \(K\)-functional of \(\vec X_{\mathrm{inhom}}\) and then derive
the corresponding characterization for
\(\vec X_{\mathrm{hom}}\). 
The following $K$-functional estimates are the quantum Euclidean
analogues of the classical Johnen--Scherer theorem \cite{JS76}; see
also \cite[Theorem~5.4.12]{BS88}.

\begin{theorem}
\label{thm:K-functional}
Let $1\leq p<\infty$ and $k\in\mathbb N$. Then, for every
$\x\in L_p(\R_\theta^d)$ and every $\varepsilon>0$,
\begin{equation}\label{eq:K-functional}
K\bigl(
\x,\varepsilon^k;
L_p(\R_\theta^d),
W_p^k(\R_\theta^d)
\bigr)
\approx
\min\{1,\varepsilon^k\}\|\x\|_{L_p(\R_\theta^d)}
+
\omega_p^k(\x,\varepsilon),
\end{equation}
with implicit constants depending only on $d$ and $k$.
\end{theorem}

\begin{proof}
For brevity, write
\[
K(\x,\varepsilon^k)
=
K\bigl(
\x,\varepsilon^k;
L_p(\R_\theta^d),
W_p^k(\R_\theta^d)
\bigr).
\]
Suppose that $\x=\y+\z$, where $\y\in L_p(\R_\theta^d)$ and $\z\in W_p^k(\R_\theta^d)$. Then
\[
\min(1,\varepsilon^k)\|\x\|_p\leq \|\y\|_p+\,\varepsilon^k\|\z\|_p\leq \|\y\|_p+\,\varepsilon^k\|\z\|_{W_p^k}.
\]
Moreover, by Remark \ref{rem:modulus-properties} and Theorem~\ref{th3.7},
\[
\omega_p^k(\x,\varepsilon) \le \omega_p^k(\y,\varepsilon)+\omega_p^k(\z,\varepsilon) \lesssim \|\y\|_p + \varepsilon^k|\z|_{\dot{W}_p^k} \lesssim \|\y\|_p + \varepsilon^k\|\z\|_{W_p^k}.
\]
Taking the infimum over all such decompositions gives
\[
\min(1,\varepsilon^k)\|\x\|_p + \omega_p^k(\x,\varepsilon) \lesssim K(\x,\varepsilon^k).
\]

We turn to the converse estimate. Since $K(\x,\varepsilon^k)\le\|\x\|_p$ for $\varepsilon\ge 1$, it suffices to show
\begin{equation}\label{eq:K-upper}
K(\x,\varepsilon^k) \lesssim \varepsilon^k\|\x\|_p + \omega_p^k(\x,\varepsilon),\quad 0<\varepsilon<1.
\end{equation}
We first assume that $\x\in\mcs(\R_\theta^d)$, so that all
differentiations under the Bochner integral below are justified. 
Let $\mathbb{I}=[0,1)^d$ in $\R^d$, and for
$\mathbf u=(u_1,\ldots,u_k)\in\mathbb I^k$ set
\[
U(\mathbf u)=u_1+\cdots+u_k.
\]
Define
\begin{align*}
\y = (-1)^k\int_{\mathbb{I}}\cdots\int_{\mathbb{I}}\Delta_{\varepsilon U(\mathbf u)}^k\x\,\mathrm{d}u_1\cdots\mathrm{d}u_k,\quad \z = \x-\y.
\end{align*}
Then
\begin{equation}\label{eq:y-bound}
\|\y\|_p
\leq
\int_{\mathbb I^k}
\bigl\|
\Delta_{\varepsilon U(\mathbf u)}^k\x
\bigr\|_p
\,\mathrm d\mathbf u  \leq \omega_p^k(\x,k\sqrt{d}\,\varepsilon) \lesssim \omega_p^k(\x,\varepsilon),
\end{equation}
where the last inequality uses \eqref{eq:moduli-real-scaling}.
To estimate $\z$, we expand $\Delta_{\varepsilon U(\mathbf u)}^k\x$ by using \eqref{eq:binomial-diff}:
\begin{align*}
\y 
&= \x + (-1)^k\sum_{j=1}^k(-1)^{k-j}\binom{k}{j}\int_{\mathbb{I}^k}T_{j\varepsilon U(\mathbf u)}(\x)\,\mathrm{d}u_1\cdots\mathrm{d}u_k,
\end{align*}
so that
\begin{equation}\label{eq:z-representation}
\z
=
(-1)^{k+1}
\sum_{\ell=1}^k
(-1)^{k-\ell}
\binom{k}{\ell}\z_\ell, \quad \text{where} \quad  \z_\ell
=
\int_{\mathbb I^k}
T_{\ell\varepsilon U(\mathbf u)}\x
\,\mathrm d\mathbf u.
\end{equation}
By \eqref{eq:in-hom-norm-equivalent}, we have
\begin{equation}\label{eq:z-Sobolev}
\|\z\|_{W_p^k}\lesssim\|\z\|_p+|\z|_{\dot{W}_p^k}.
\end{equation}
Since $T_t$ is an isometry on $L_p(\R^d_\theta)$, the representation \eqref{eq:z-representation} gives
\begin{equation}\label{eq:z-Lp}
\|\z\|_p
\leq
\sum_{\ell=1}^k
\binom{k}{\ell}\|\x\|_p
\leq
2^k\|\x\|_p.
\end{equation}

It remains to estimate the derivatives of $\z_\ell$ of order $k$. 
Fix
a multi-index
\[
m=(m_1,\ldots,m_d)\in\mathbb N_0^d,
\qquad |m|=m_1+\cdots+m_d=k,
\]
and set
\[
M_0=0,
\qquad
M_j=m_1+\cdots+m_j,
\quad 1\leq j\leq d.
\]
For each $j\in\{1,\ldots,d\}$, define
\[
R_j
:=
\{M_{j-1}+1,M_{j-1}+2,\ldots,M_j\},
\]
with the convention that $R_j=\varnothing$ if $m_j=0$. Then
$|R_j|=m_j$, and
\[
R_1\mathbin{\cup}\cdots\mathbin{\cup}R_d
=
\{1,\ldots,k\}.
\]
Write $u_r=\sum_{j=1}^d u_{r,j}\mathbf e_j$, $r=1,\ldots,k$. For each $j$, we represent the $m_j$ occurrences of
$\partial_j$ in $\partial^m$ as derivatives with respect
to the \(m_j\) distinct variables
\[
\{u_{r,j}:r\in R_j\}.
\]
Since
\(
U(\mathbf u)=u_1+\cdots+u_k,
\)
the chain rule and \eqref{eq:translation-generator-def} give, for every
$r\in\{1,\ldots,k\}$,
\[
\partial_j
T_{\ell\varepsilon U(\mathbf u)}\x
=
\frac{1}{\ell\varepsilon}
\frac{\partial}{\partial u_{r,j}}
T_{\ell\varepsilon U(\mathbf u)}\x.
\]
Since \(R_1,\ldots,R_d\) form a partition of
\(\{1,\ldots,k\}\), it follows that
\[
\partial^m\z_\ell
=
\partial_1^{m_1}\cdots\partial_d^{m_d}
\int_{\mathbb I^k}
T_{\ell\varepsilon U(\mathbf u)}\x\,\mathrm d\mathbf u
=
\big(\ell\varepsilon\big)^{-k}
\int_{\mathbb I^k}
\left(
\prod_{j=1}^d
\prod_{r\in R_j}
\frac{\partial}{\partial u_{r,j}}
\right)
T_{\ell\varepsilon U(\mathbf u)}\x
\,\mathrm d\mathbf u.
\]
 Fix $j\in\{1,\ldots,d\}$ and $r\in R_j$, and set
\[
U_{r,j}^{(0)}(\mathbf u)
:=
U(\mathbf u)-u_{r,j}\mathbf{e}_j.
\]
Since $U_{r,j}^{(0)}(\mathbf u)$ is independent of $u_{r,j}$, the
Banach-valued fundamental theorem of calculus yields
\begin{equation}\label{eq:FTC-produces-difference}
\begin{aligned}
\int_0^1
\frac{\partial}{\partial u_{r,j}}
T_{\ell\varepsilon U(\mathbf u)}\x
\,\mathrm du_{r,j}
&=
T_{\ell\varepsilon(U_{r,j}^{(0)}(\mathbf u)+\mathbf e_j)}\x
-
T_{\ell\varepsilon U_{r,j}^{(0)}(\mathbf u)}\x
\\
&=
T_{\ell\varepsilon U_{r,j}^{(0)}(\mathbf u)}
\bigl(T_{\ell\varepsilon\mathbf e_j}-I\bigr)\x
\\
&=
T_{\ell\varepsilon U_{r,j}^{(0)}(\mathbf u)}
\Delta_{\ell\varepsilon\mathbf e_j}\x.
\end{aligned}
\end{equation}
Successively applying \eqref{eq:FTC-produces-difference} in the
variables \(u_{r,j}\), \(1\leq j\leq d\) and \(r\in R_j\), we obtain 
\begin{equation}\label{eq:derivative-z-ell}
\partial^m\z_\ell
=
(\ell\varepsilon)^{-k}
\int_{[0,1)^{dk-k}}
T_{\ell\varepsilon V_m(\mathbf v)}
\left(
\Delta_{\ell\varepsilon \mathbf{e}_1}^{m_1}
\cdots
\Delta_{\ell\varepsilon \mathbf{e}_d}^{m_d}\x
\right)
\,\mathrm d\mathbf v,
\end{equation}
where
\[
\mathbf v
=
\bigl(
u_{r,j}:
1\leq j\leq d,\ 
r\in\{1,\ldots,k\}\setminus R_j
\bigr)
\in[0,1)^{dk-k}
\]
and
\[
V_m(\mathbf v)
:=
\sum_{j=1}^d
\sum_{r\in\{1,\ldots,k\}\setminus R_j}
u_{r,j}\mathbf e_j.
\]
Taking \(L_p\)-norms in \eqref{eq:derivative-z-ell} and using the
triangle inequality for Bochner integrals together with the
\(L_p\)-isometry of the translations, we obtain
\begin{align*}
\|\partial^m\z_\ell\|_p
&\leq
(\ell\varepsilon)^{-k}
\left\|
\Delta_{\ell\varepsilon e_1}^{m_1}
\cdots
\Delta_{\ell\varepsilon e_d}^{m_d}\x
\right\|_p
\lesssim
(\ell\varepsilon)^{-k}
\omega_p^k(\x,\ell\varepsilon)
\lesssim
\varepsilon^{-k}
\omega_p^k(\x,\varepsilon),
\end{align*}
where the second inequality follows from
Lemma~\ref{lem3.6}, and the last one from
\eqref{eq:moduli-integer-scaling}.
Summing over
$\ell=1,\ldots,k$ and over all $|m|=k$, we obtain
\begin{equation}\label{eq:z-homogeneous-Sobolev}
\varepsilon^k
|\z|_{\dot W_p^k}
\lesssim
\omega_p^k(\x,\varepsilon).
\end{equation}
Consequently, by \eqref{eq:z-Sobolev},
\eqref{eq:z-Lp}, and \eqref{eq:z-homogeneous-Sobolev},
\begin{align}\label{eq:es5.21}
\varepsilon^k\|\z\|_{W_p^k} \lesssim \varepsilon^k\|\x\|_p + \omega_p^k(\x,\varepsilon).
\end{align}
Together with the bound \eqref{eq:y-bound} for $\y$, this yields \eqref{eq:K-upper} for $\x\in\mcs(\R_\theta^d)$.

We now prove that \eqref{eq:K-upper} holds for arbitrary elements of \(L_p(\R_\theta^d)\).
Fix $0<\varepsilon<1$ and define the linear operators \(Y_\varepsilon\) and \(Z_\varepsilon\)
initially on \(\mcs(\R_\theta^d)\) by
\[
Y_\varepsilon\x
:=
(-1)^k
\int_{\mathbb I^k}
\Delta_{\varepsilon U(\mathbf u)}^k\x
\,\mathrm d\mathbf u,
\qquad
Z_\varepsilon\x
:=
\x-Y_\varepsilon\x.
\]
 By \eqref{eq:y-bound}, \eqref{eq:es5.21}, and
\(\omega_p^k(\x,\varepsilon)\leq2^k\|\x\|_p\), both operators extend uniquely to bounded linear operators
\[
Y_\varepsilon:
L_p(\R_\theta^d)\longrightarrow L_p(\R_\theta^d),
\qquad
Z_\varepsilon:
L_p(\R_\theta^d)\longrightarrow W_p^k(\R_\theta^d),
\]
respectively. Let \(\x\in L_p(\R_\theta^d)\), and choose
\(\x_n\in\mcs(\R_\theta^d)\) such that
\(\x_n\to\x\) in \(L_p\). Then 
\[
Y_\varepsilon\x_n\to Y_\varepsilon\x
\quad\text{in }L_p,
\qquad
Z_\varepsilon\x_n\to Z_\varepsilon\x
\quad\text{in }W_p^k,\quad \text{as} \quad n\to \infty.
\]
Consequently, the decomposition
\[
\x=Y_\varepsilon\x+Z_\varepsilon\x,
\]
is an admissible decomposition in the definition of the
\(K\)-functional.
Moreover, the estimate
\[
\left|
\omega_p^k(\x_n,\varepsilon)
-
\omega_p^k(\x,\varepsilon)
\right|
\leq
2^k\|\x_n-\x\|_p,
\]
together with \eqref{eq:y-bound} and
\eqref{eq:es5.21}, yields
\[
\begin{aligned}
K\bigl(
\x,\varepsilon^k;
L_p(\R_\theta^d),
W_p^k(\R_\theta^d)
\bigr)
&\leq
\|Y_\varepsilon\x\|_p
+
\varepsilon^k
\|Z_\varepsilon\x\|_{W_p^k}\\
&\lesssim
\varepsilon^k\|\x\|_p
+
\omega_p^k(\x,\varepsilon).
\end{aligned}
\]
This proves \eqref{eq:K-upper} for every
$\x\in L_p(\R_\theta^d)$.
\end{proof}

The preceding construction also yields the corresponding
\(K\)-functional characterization for the homogeneous Sobolev
couple \(\vec X_{\mathrm{hom}}
=
\bigl(
L_p(\R_\theta^d),
\dot W_{p,0}^k(\R_\theta^d)
\bigr)\). 

\begin{corollary}\label{cor:homogeneous-K-functional}
Let $1\leq p<\infty$ and $k\in\mathbb N$. Then, for every
$\x\in
L_p(\R_\theta^d)
+
\dot W_{p,0}^k(\R_\theta^d)$ and every $\varepsilon>0$,
\begin{align}
K\bigl(
\x,\varepsilon^k;
L_p(\R_\theta^d),
\dot W_{p,0}^k(\R_\theta^d)
\bigr)
&\approx
\omega_p^k(\x,\varepsilon).
\end{align}
The implicit constants depend only on $d$ and $k$.
\end{corollary}
\begin{proof}
By
\eqref{eq:kth-difference-Sobolev-bound}, for every
\(\z\in\mcs(\R_\theta^d)\),
\[
\|\Delta_h^k\z\|_p
\lesssim_{d,k}
|h|^k
|\z|_{\dot W_p^k(\R_\theta^d)}
=
|h|^k
\|\z\|_{\dot W_{p,0}^k(\R_\theta^d)}.
\]
Hence, for every \(h\in\R^d\), \(\Delta_h^k\) extends uniquely to a
bounded operator
\[
\widetilde\Delta_h^k:\dot W_{p,0}^k(\R_\theta^d) \longrightarrow  L_p(\R_\theta^d).
\]
We first verify that this extension agrees with the usual difference
operator on \(L_p(\R_\theta^d)\cap\dot W_{p,0}^k(\R_\theta^d)\). Indeed, if \(
\x\in
L_p(\R_\theta^d)\cap\dot W_{p,0}^k(\R_\theta^d)
\), then
\(\x\in W_p^k(\R_\theta^d)\) by
\eqref{eq:Sobolev-intersection}. Choose
\(\x_n\in\mcs(\R_\theta^d)\) such that
\(\x_n\to\x\) in \(W_p^k(\R_\theta^d)\). The estimates
\[
\|\Delta_h^k(\x_n-\x)\|_p
\leq
2^k\|\x_n-\x\|_p \longrightarrow0 \qquad (n \to \infty)
\]
and
\[
\|\widetilde\Delta_h^k(\x_n-\x)\|_p
\lesssim
|h|^k \|\x_n-\x\|_{\dot W_{p,0}^k}\longrightarrow0 \qquad (n \to \infty)
\]
show that 
\(
\widetilde\Delta_h^k\x=\Delta_h^k\x.
\)
Consequently, \(\Delta_h^k\) is well defined on \(
L_p(\R_\theta^d)+\dot W_{p,0}^k(\R_\theta^d)
\) by
\[\Delta_h^k(\y+\z)
:=
\Delta_h^k\y+\widetilde\Delta_h^k\z,
\qquad \y\in L_p(\R_\theta^d),
\quad
\z\in\dot W_{p,0}^k(\R_\theta^d).
\]
For \(\x\in L_p(\R_\theta^d)+\dot W_{p,0}^k(\R_\theta^d)\), define
\[
\omega_p^k(\x,\varepsilon)
:=
\sup_{0<|h|\leq\varepsilon}
\|\Delta_h^k\x\|_p.
\]
For every decomposition
\[
\x=\y+\z,
\qquad
\y\in L_p(\R_\theta^d),
\quad
\z\in\dot W_{p,0}^k(\R_\theta^d),
\]
we have
\begin{equation}\label{eq:homogeneous-difference-bound}
\omega_p^k(\x,\varepsilon)
\leq
\omega_p^k(\y,\varepsilon)
+
\omega_p^k(\z,\varepsilon)
\lesssim
\|\y\|_p
+
\varepsilon^k\|\z\|_{\dot W_{p,0}^k}.
\end{equation}
Taking the infimum over all such decompositions gives
\begin{equation}\label{eq:homogeneous-K-lower}
\omega_p^k(\x,\varepsilon)
\lesssim
K\bigl(
\x,\varepsilon^k;
\vec X_{\mathrm{hom}}
\bigr).
\end{equation}

For the reverse estimate, first let
\(\x\in\mcs(\R_\theta^d)\). The construction used in the proof of
Theorem~\ref{thm:K-functional} combined with Remark \ref{rem:Sobolev-intersection} provides a decomposition
\[
\x=\y+\z,
\qquad
\y\in L_p(\R_\theta^d),
\quad
\z\in\dot W_{p,0}^k(\R_\theta^d),
\]
such that, by \eqref{eq:canonical-Sobolev-embedding}, \eqref{eq:y-bound} and
\eqref{eq:z-homogeneous-Sobolev},
\[
\|\y\|_p
+
\varepsilon^k
\|\z\|_{\dot W_{p,0}^k}
\lesssim
\omega_p^k(\x,\varepsilon).
\]
Therefore,
\begin{equation}\label{eq:homogeneous-K-upper-Schwartz}
K\bigl(
\x,\varepsilon^k;
\vec X_{\mathrm{hom}}\bigr)
\lesssim
\omega_p^k(\x,\varepsilon),
\qquad
\x\in\mcs(\R_\theta^d).
\end{equation}

The space
$\mcs(\R_\theta^d)$ is dense in
\(
L_p(\R_\theta^d)+\dot W_{p,0}^k(\R_\theta^d)
\)
with respect to the sum norm
\[
\|\cdot\|_{L_p(\R_\theta^d)+\dot W_{p,0}^k(\R_\theta^d)}
:=
K(\cdot,1;\vec X_{\mathrm{hom}}).
\]
Let \(\x\in L_p(\R_\theta^d)+\dot W_{p,0}^k(\R_\theta^d)\) and choose \(\x_n\in\mcs(\R_\theta^d)\) such that
\[
K\bigl(
\x_n-\x,1;
\vec X_{\mathrm{hom}}
\bigr)
\longrightarrow0,\quad \text{as} \quad n\to \infty.
\]
For every fixed \(\varepsilon>0\),
\[
K(\x_n-\x,\varepsilon^k;\vec X_{\mathrm{hom}})
\leq
\max\{1,\varepsilon^k\}K(\x_n-\x,1;\vec X_{\mathrm{hom}})
\longrightarrow0, \qquad n\to \infty.
\]
Moreover, by \eqref{eq:homogeneous-K-lower},
\begin{align*}
\omega_p^k(\x_n-\x,\varepsilon)
\lesssim
K\bigl(
\x_n-\x,\varepsilon^k;
\vec X_{\mathrm{hom}}
\bigr)
\longrightarrow0, \qquad n\to \infty.
\end{align*}
The subadditivity of the \(K\)-functional and
\(\omega_p^k(\,\cdot\,,\varepsilon)\) therefore allows us to pass to
the limit in \eqref{eq:homogeneous-K-upper-Schwartz}, yielding
\[
K\bigl(
\x,\varepsilon^k;
\vec X_{\mathrm{hom}}
\bigr)
\lesssim
\omega_p^k(\x,\varepsilon).
\]
Together with \eqref{eq:homogeneous-K-lower}, this proves the
assertion.
\end{proof}

Inserting the preceding \(K\)-functional estimates into the
definition of the real interpolation norm gives the following Besov
norm equivalences.

\begin{corollary}\label{cor:interpolation-Besov}
Let $1\leq p<\infty$, $1\leq q\leq\infty$, $k\in\mathbb N$, and
$0<\eta<1$. Set
\(
\alpha=\eta k
\)
and
\[
c_{\eta,q}
:=
\begin{cases}
\bigl(q\eta(1-\eta)\bigr)^{-1/q},
& 1\leq q<\infty,\\[2mm]
1,
& q=\infty.
\end{cases}
\] 
Then
\begin{equation}\label{eq:inhomogeneous-interpolation-Besov}
\bigl(
L_p(\R_\theta^d),W_p^k(\R_\theta^d)
\bigr)_{\eta,q}
=
B_{p,q}^{\alpha}(\R_\theta^d)
\end{equation}
with equivalent norms. More precisely,
\begin{equation}\label{eq:inhomogeneous-interpolation-norm}
\begin{aligned}
\|\x\|_{(L_p,W_p^k)_{\eta,q}}
\approx
c_{\eta,q}\|\x\|_p
+
k^{1/q}
|\x|_{B_{p,q}^{\alpha,k}(\R_\theta^d)}
\approx_{\eta,q,k}
\|\x\|_{B_{p,q}^{\alpha,k}(\R_\theta^d)}.
\end{aligned}
\end{equation}
 Moreover, for every $\x\in L_p(\R_\theta^d)$,
\begin{equation}\label{eq:homogeneous-interpolation-norm}
\|\x\|_{
\bigl(
L_p(\R_\theta^d),\dot W_{p,0}^k(\R_\theta^d)
\bigr)_{\eta,q}
}
\approx
k^{1/q}
|\x|_{B_{p,q}^{\alpha,k}(\R_\theta^d)},
\end{equation}
where both sides may be infinite.
\end{corollary}

\begin{proof}
By Theorem~\ref{thm:K-functional}, \[K(\x,\varepsilon;\,L_p,W_p^k)\approx\min(1,\varepsilon)\|\x\|_p+\omega_p^k(\x,\varepsilon^{1/k}),\qquad \varepsilon>0.\]
For \(1\leq q<\infty\), it follows that
\[
\|\x\|_{(L_p,W_p^k)_{\eta,q}}^q = \int_0^\infty\bigl[\varepsilon^{-\eta}K(\x,\varepsilon)\bigr]^q\,\frac{\mathrm{d}\varepsilon}{\varepsilon} \approx I_1 + I_2,
\]
where
\[
I_1 := \|\x\|_p^q\int_0^\infty\bigl[\varepsilon^{-\eta}\min(1,\varepsilon)\bigr]^q\,\frac{\mathrm{d}\varepsilon}{\varepsilon}=  \frac{1}{q\eta(1-\eta)} \|\x\|_p^q,
\]
and, upon setting \(\varepsilon=t^k\),
\[
I_2
:= \int_0^\infty\bigl[\varepsilon^{-\eta}\omega_p^k(\x,\varepsilon^{1/k})\bigr]^q\,\frac{\mathrm{d}\varepsilon}{\varepsilon}=
k\,|\x|_{B_{p,q}^{\alpha,k}(\R_\theta^d)}^q.
\]
Hence
\[
\|\x\|_{(L_p,W_p^k)_{\eta,q}}
\approx
c_{\eta,q}\|\x\|_p
+
k^{1/q}
|\x|_{B_{p,q}^{\alpha,k}(\R_\theta^d)}.
\]
The case \(q=\infty\) follows in the same way, with suprema replacing
the integrals. This proves the first equivalence in
\eqref{eq:inhomogeneous-interpolation-norm}. The second follows
directly from the definition of
\(\|\cdot\|_{B_{p,q}^{\alpha,k}(\R_\theta^d)}\), while
Lemma~\ref{lem:Besov-difference-characterization} yields
\eqref{eq:inhomogeneous-interpolation-Besov}.

Using Corollary~\ref{cor:homogeneous-K-functional} in place of
Theorem~\ref{thm:K-functional}, the same calculation with the
\(I_1\)-term omitted gives
\eqref{eq:homogeneous-interpolation-norm}.
\end{proof}

\subsection{Proofs of Theorems~\ref{th1.2} and~\ref{th1.3}}
\label{subsec:proofs-BBM-MS}

For the BBM-type limit at the upper endpoint
\(\alpha\uparrow k\), we apply
Theorem~\ref{thm:Milman}\textnormal{(i)} to
\(\vec X_{\mathrm{hom}}\).

{\textit{\textbf {Proof of Theorem~\ref{th1.2}.}}}      
We first verify the \(X_1\)-normality condition:
\begin{equation}\label{eq:normal-upper-homogeneous}
\|\z\|_{\widetilde X_1}
\approx
\|\z\|_{\dot W_{p,0}^k(\R_\theta^d)},
\qquad
\z\in\dot W_{p,0}^k(\R_\theta^d).
\end{equation}
Choose
$\z_n\in\mcs(\R_\theta^d)$ such that
\(
\z_n\longrightarrow \z\)
in \(\dot W_{p,0}^k(\R_\theta^d).
\)
By \eqref{eq:homogeneous-difference-bound}, 
\[
\sup_{\varepsilon>0}
\frac{
\bigl|
\omega_p^k(\z_n,\varepsilon)-\omega_p^k(\z,\varepsilon)
\bigr|
}{\varepsilon^k}
\lesssim
\|\z_n-\z\|_{\dot W_{p,0}^k(\R_\theta^d)}\longrightarrow 0 \qquad (n \to \infty).
\]
Applying Theorem~\ref{th3.7} to \(\z_n\) and passing to the limit
gives
\begin{equation}\label{eq:homogeneous-small-scale-endpoint}
\lim_{\varepsilon\downarrow0}
\frac{\omega_p^k(\z,\varepsilon)}{\varepsilon^k}
=
\sup_{\varepsilon>0}
\frac{\omega_p^k(\z,\varepsilon)}{\varepsilon^k}
\approx
\|\z\|_{\dot W_{p,0}^k(\R_\theta^d)}.
\end{equation}
Corollary~\ref{cor:homogeneous-K-functional} and
\eqref{eq:homogeneous-small-scale-endpoint} now yield
\[
\begin{aligned}
\|\z\|_{\widetilde X_1}
=
\lim_{s\downarrow0}
\frac{K(\z,s;\vec X_{\mathrm{hom}})}{s}
&=
\lim_{\varepsilon\downarrow0}
\frac{
K(\z,\varepsilon^k;\vec X_{\mathrm{hom}})
}{
\varepsilon^k
}\\
&\approx
\lim_{\varepsilon\downarrow0}
\frac{\omega_p^k(\z,\varepsilon)}{\varepsilon^k}
\approx
\|\z\|_{\dot W_{p,0}^k(\R_\theta^d)}.
\end{aligned}
\]
This proves \eqref{eq:normal-upper-homogeneous}.

Now let $\x\in W_p^k(\R_\theta^d)$. By
Remark~\ref{rem:Sobolev-intersection},
$\x\in L_p(\R_\theta^d)\cap \dot W_{p,0}^k(\R_\theta^d)$. Hence Theorem~\ref{thm:Milman}\textnormal{(i)},
\eqref{eq:normal-upper-homogeneous}  and  \eqref{eq:canonical-Sobolev-embedding} give
\begin{align}\label{eq:Milman-BBM-application}
\lim_{\eta\uparrow1}
\bigl(\eta(1-\eta)\bigr)^{1/q}
\|\x\|_{(\vec X_{\mathrm{hom}})_{\eta,q}}
=
q^{-1/q}\|\x\|_{\widetilde X_1}
\approx
q^{-1/q}
\|\x\|_{\dot W_{p,0}^k(\R_\theta^d)}
=
q^{-1/q}
|\x|_{\dot W_p^k(\R_\theta^d)}.
\end{align}
On the other hand, setting \(\alpha=\eta k\) in
\eqref{eq:homogeneous-interpolation-norm},
we obtain
\begin{equation}\label{eq:hom-uniform-equivalence}
\bigl(\eta(1-\eta)\bigr)^{1/q}
\|\x\|_{\bigl(
L_p(\R_\theta^d),\dot W_{p,0}^k(\R_\theta^d)
\bigr)_{\eta,q}}
\approx
\left(\frac{\alpha}{k}\right)^{1/q}
(k-\alpha)^{1/q}
|\x|_{B_{p,q}^{\alpha,k}(\R_\theta^d)},
\end{equation}
with constants independent of \(\eta\). Moreover,
\[
(k-\alpha)^{1/q}\|\x\|_{L_p(\R_\theta^d)}
\longrightarrow0
\qquad (\alpha\uparrow k),
\]
It follows from \eqref{eq:exact-BBM-modulus-limit} that
\[
\lim_{\alpha\uparrow k}
(k-\alpha)^{1/q}
|\x|_{B_{p,q}^{\alpha,k}(\R_\theta^d)}
\]
exists. Since \(\alpha/k\to1\), we may pass to the limit in
\eqref{eq:hom-uniform-equivalence} and use
\eqref{eq:Milman-BBM-application} to obtain
\[
\lim_{\alpha\uparrow k}
(k-\alpha)^{1/q}
\|\x\|_{B_{p,q}^{\alpha,k}(\R_\theta^d)}=\lim_{\alpha\uparrow k}
(k-\alpha)^{1/q}
|\x|_{B_{p,q}^{\alpha,k}(\R_\theta^d)}
\approx
q^{-1/q}
|\x|_{\dot W_p^k(\R_\theta^d)}.
\]
\vspace{0.3cm}

We next establish the MS-type limit at the lower endpoint
\(\alpha\downarrow0\). To apply
Theorem~\ref{thm:Milman}\textnormal{(ii)} to
\(\vec X_{\mathrm{hom}}\), we need to verify the
\(X_0\)-normality condition in \eqref{eq:normal}. By
Corollary~\ref{cor:homogeneous-K-functional}, this reduces to
a
large-scale estimate for the modulus of smoothness.
Its proof uses the following generalized Riemann-Lebesgue lemma.

\begin{lemma}\label{lem:matrix-RL}
Let $d_1, d_2 \in \mathbb{N}$, $G \in L^1(\mathbb{R}^{d_1})$ and let  $L : \mathbb{R}^{d_1} \to \mathbb{R}^{d_2}$ be a surjective linear map.
Then
 \[
\lim_{|h| \to \infty} \int_{\mathbb{R}^{d_1}} e^{2\pi i \langle h, \, L(\xi) \rangle_{\mathbb{R}^{d_2}}} G(\xi) \, \mathrm{d}\xi = 0, \quad h\in\mathbb R^{d_2}.
\]
\end{lemma}

\begin{proof}
Let $\mathbf{A} \in \mathbb{R}^{d_2 \times d_1}$ be the matrix representing
\(L\).
Since 
\[
\langle h, \, L(\xi) \rangle_{\mathbb{R}^{d_2}} = \langle h, \, \mathbf{A}\xi \rangle_{\mathbb{R}^{d_2}} = \langle \mathbf{A}^T h, \, \xi \rangle_{\mathbb{R}^{d_1}},
\]
we have
\[
 \int_{\mathbb{R}^{d_1}} e^{2\pi i \langle \mathbf{A}^T h, \, \xi \rangle_{\mathbb{R}^{d_1}}} G(\xi) \, \mathrm{d}\xi = \widehat{G}\big(-\mathbf{A}^T h\big).
\]
Since \(L\) is surjective, \(\mathbf A\) has full row rank,   and hence \(\mathbf A\mathbf A^T\) is positive definite.  Let $\lambda_{\min} > 0$ denote its smallest eigenvalue. Then, for every $h \in \mathbb{R}^{d_2}$, we obtain 
\[
|\mathbf{A}^T h|^2 = \langle \mathbf{A}^T h, \, \mathbf{A}^T h \rangle_{\mathbb{R}^{d_1}} = \langle \mathbf{A}\mathbf{A}^T h, \, h \rangle_{\mathbb{R}^{d_2}} \geq \lambda_{\min} |h|^2.
\]
Thus \(|\mathbf A^Th|\to\infty\) whenever \(|h|\to\infty\).
The classical Riemann--Lebesgue lemma therefore yields
\[
\widehat G(-\mathbf A^Th)\longrightarrow0
\qquad (|h|\to\infty).
\]
\end{proof}

We now determine the large-scale behavior of the modulus of
smoothness required for the \(X_0\)-normality condition. Marchaud's
inequality reduces the problem to the first-order modulus, whose
analysis rests on the sharp asymptotic behavior of
\(\|T_h\x-\x\|_p\) as \(|h|\to\infty\). The preceding lemma will be
used to show that the mixed tracial terms arising in this limit
vanish.

\begin{lemma}
\label{lem:MS-normal-condition}
Let $1\leq p<\infty$ and $k\in\N$. Then, for every
$\x\in L_p(\R_\theta^d)$, 
\begin{equation}\label{eq:large-scale-modulus}
\lim_{\varepsilon\to\infty}
\omega_p^k(\x,\varepsilon)
\approx
\|\x\|_{L_p(\R_\theta^d)}.
\end{equation}
\end{lemma}

\begin{proof}
Set\[
\omega_p^k(\x,\infty)
:=
\lim_{\varepsilon\to\infty}
\omega_p^k(\x,\varepsilon).
\]
Since
\(\varepsilon\mapsto\omega_p^k(\x,\varepsilon)\) is nondecreasing
and bounded above by \(2^k\|\x\|_p\), the limit
\(\omega_p^k(\x,\infty)\) is well defined and finite.

\textit{Step 1.} \textit{Reduction to the first-order modulus of smoothness.}
Suppose that $k\geq2$. Marchaud's inequality
(Proposition~\ref{prop:Marchaud}) with \(n=1\) gives
\begin{equation}\label{eq:Marchaud-large-scale}
2^{1-k}\omega_p^k(\x,\varepsilon)
\leq
\omega_p^1(\x,\varepsilon)
\lesssim
\varepsilon
\int_\varepsilon^\infty
\frac{\omega_p^k(\x,\delta)}{\delta^2}
\,\mathrm d\delta,
\qquad \varepsilon>0.
\end{equation}
We claim that
\begin{equation}\label{eq:Marchaud-average-limit}
\lim_{\varepsilon\to\infty}
\varepsilon
\int_\varepsilon^\infty
\frac{\omega_p^k(\x,\delta)}{\delta^2}
\,\mathrm d\delta
=
\omega_p^k(\x,\infty).
\end{equation}
Indeed, 
\[
\begin{aligned}
\left|
\varepsilon
\int_\varepsilon^\infty
\frac{\omega_p^k(\x,\delta)}{\delta^2}
\,\mathrm d\delta
-
\omega_p^k(\x,\infty)
\right|
\leq
\varepsilon
\int_\varepsilon^\infty
\frac{
\bigl|
\omega_p^k(\x,\delta)-\omega_p^k(\x,\infty)
\bigr|
}{\delta^2}
\,\mathrm d\delta.
\end{aligned}
\]
Given \(\eta>0\), choose $R>0$ such that
\[
\bigl|
\omega_p^k(\x,\delta)-\omega_p^k(\x,\infty)
\bigr|
<\eta,
\qquad
\delta\geq R.
\]
Then, for $\varepsilon\geq R$,
\[
\begin{aligned}
\left|
\varepsilon
\int_\varepsilon^\infty
\frac{\omega_p^k(\x,\delta)}{\delta^2}
\,\mathrm d\delta
-
\omega_p^k(\x,\infty)
\right|
\leq
\eta\;\varepsilon
\int_\varepsilon^\infty
\frac{\mathrm d\delta}{\delta^2}
=
\eta.
\end{aligned}
\]
This proves \eqref{eq:Marchaud-average-limit}.

Letting $\varepsilon\to\infty$ in
\eqref{eq:Marchaud-large-scale} and using
\eqref{eq:Marchaud-average-limit}, we obtain
\[
2^{1-k}\omega_p^k(\x,\infty)
\leq
\omega_p^1(\x,\infty)
\lesssim
\omega_p^k(\x,\infty),
\]
Consequently,
\[
\omega_p^k(\x,\infty)
\approx
\omega_p^1(\x,\infty),
\qquad k\in\N,
\]
and it therefore remains to prove that
\[
\omega_p^1(\x,\infty)
\approx
\|\x\|_{L_p(\R_\theta^d)}.
\]
The upper bound follows immediately from
\(
\omega_p^1(\x,\varepsilon)
\leq
2\|\x\|_p.
\)
It therefore remains to establish the reverse estimate.

\medskip
\textit{Step 2. The lower bound for even exponents.}
Let $m\in\N$ and
\[
\x=U_\theta(f)\in\mcs(\R_\theta^d),
\qquad
f\in\mcs(\R^d).
\]
We shall prove 
\begin{equation}\label{eq:even-exponent-translation-limit}
\lim_{|h|\to\infty}
\|\Delta_h\x\|_{2m}^{2m}
=
2\|\x\|_{2m}^{2m}.
\end{equation}

Set
\(
a=T_h\x\) and \(b=\x\), we have
\[
\|\Delta_h\x\|_{2m}^{2m}
=
\tau_\theta\Bigl(\big((a-b)^*(a-b)\big)^m\Bigr),
\quad
(a-b)^*(a-b)
=
a^*a+b^*b-a^*b-b^*a.
\]
Expanding the \(m\)-fold product distributively yields \(4^m\) terms
of the form
\[
(-1)^j\tau_\theta(W_1\cdots W_m),
\qquad
W_i\in
\{a^*a,b^*b,a^*b,b^*a\},
\]
where $j$ counts the factors belonging to
$\{a^*b,b^*a\}$.

We classify these $4^m$ terms into three types:
\begin{itemize}
\item \textbf{Pure diagonal:} all $W_i=a^*a$, or all $W_i=b^*b$.
\item \textbf{Mixed diagonal:} every $W_i\in\{a^*a,b^*b\}$, but both appear.
\item \textbf{Mixed:} at least one $W_i\in\{a^*b,b^*a\}$.
\end{itemize}
We show that only the two pure diagonal terms survive as $|h|\to\infty$.

Recall that
\begin{gather*}
U_\theta(g_1)\cdot U_\theta(g_2)=U_\theta(g_1*_\theta g_2),\quad
\tau_\theta(U_\theta(g))=g(0),\\
U_\theta(g)^*=U_\theta(g^\sharp)\;\text{ where }\;
g^\sharp(s):=\overline{g(-s)}.
\end{gather*}
For operators $V_1,\ldots,V_n$ with $V_j=U_\theta(g_j)$, we have $\tau_\theta(V_1\cdots V_n)=(g_1*_\theta\cdots*_\theta g_n)(0)$. 
To evaluate this, we expand the iterated $\theta$-convolution using the definition \eqref{eq:theta-conv} recursively. 
Denote by $G_1=g_1$, $G_k=G_{k-1}*_\theta g_k$ for $k=2,3,\ldots,n$. Introducing integration variables $u_1,\ldots,u_{n-1}\in\R^d$ at each step:
\begin{align*}
G_n(0)&=\int_{\R^d}G_{n-1}(-u_{n-1})\,g_n(u_{n-1})\,\mathrm{d}u_{n-1},\\
G_{n-1}(-u_{n-1})&=\int_{\R^d}e^{\frac{i}{2}\langle u_{n-1},\,\theta u_{n-2}\rangle}
  \,G_{n-2}(-u_{n-1}-u_{n-2})\\
  &\qquad\qquad\cdot g_{n-1}(u_{n-2})\,\mathrm{d}u_{n-2},\\
G_{n-2}(-u_{n-1}-u_{n-2})&=\int_{\R^d}
  e^{\frac{i}{2}\langle u_{n-1}+u_{n-2},\,\theta u_{n-3}\rangle}\\
  &\qquad\qquad\cdot G_{n-3}(-u_{n-1}-u_{n-2}-u_{n-3})\,g_{n-2}(u_{n-3})\,\mathrm{d}u_{n-3},
\end{align*}
and in general, for $k=n-1,n-2,\ldots,2$:
\[
G_k\Bigl(-\sum_{i=k}^{n-1}u_i\Bigr)
=\int_{\R^d}e^{\frac{i}{2}\bigl\langle\sum_{i=k}^{n-1}u_i,\,\theta u_{k-1}\bigr\rangle}
\,G_{k-1}\Bigl(-\sum_{i=k-1}^{n-1}u_i\Bigr)\,g_k(u_{k-1})\,\mathrm{d}u_{k-1}.
\]
Iterating these identities yield the $(n-1)d$-dimensional integral:
\begin{equation}\label{eq:iterated_explicit}
(g_1*_\theta\cdots*_\theta g_n)(0)
=\int_{\R^{(n-1)d}}\Theta_\theta(\mathbf{u})\,
g_1\Bigl(-\sum_{k=1}^{n-1}u_k\Bigr)\,
\prod_{j=2}^{n}g_j(u_{j-1})\,\mathrm{d}\mathbf{u},
\end{equation}
where $\mathbf{u}=(u_1,\ldots,u_{n-1})\in\R^{(n-1)d}$ and 
\[
\Theta_\theta(\mathbf{u})
=\prod_{k=1}^{n-2}
e^{\frac{i}{2}\bigl\langle u_{k+1}+\cdots+u_{n-1},\,\theta u_k\bigr\rangle}.
\]
For later use, define the linear maps
$\ell_j:\R^{(n-1)d}\to\R^d$, \(j=1,\ldots,n\), by
\begin{equation}\label{eq:ell}
\ell_1(\mathbf{u})=-(u_1+\cdots+u_{n-1}),\qquad
\ell_j(\mathbf{u})=u_{j-1}\;\text{ for }j=2,\ldots,n.
\end{equation}
Then \eqref{eq:iterated_explicit} can be written as
\begin{equation}\label{eq:iterated_ell}
(g_1*_\theta\cdots*_\theta g_n)(0)
=\int_{\R^{(n-1)d}}\Theta_\theta(\mathbf{u})\,
\prod_{j=1}^{n}g_j(\ell_j(\mathbf{u}))\,\mathrm{d}\mathbf{u}.
\end{equation}
The phase $\Theta_\theta(\mathbf{u})$ is bounded ($|\Theta_\theta|=1$) and
independent of $h$.

\medskip
\textit{Mixed terms vanish.}
For arbitrary $m\in \N$ and any choice of $W_1,\ldots,W_m$ containing at least one
factor from $\{a^*b,b^*a\}$, we show
$\tau_\theta(W_1\cdots W_m)\to 0$ as $|h|\to\infty$.

Write $W_1\cdots W_m=X_1 X_2\cdots X_{2m}$ with $X_j\in\{a,a^*,b,b^*\}$.
Each elementary factor can be written as
\begin{align}\label{eq:respent}
a=U_\theta(f_h),\quad a^*=U_\theta(f_h^\sharp),\quad
b=U_\theta(f),\quad b^*=U_\theta(f^\sharp),
\end{align}
where $f_h(s)=e^{2\pi i\langle h,s\rangle}f(s)$ and
$f_h^\sharp(s)=e^{2\pi i\langle h,s\rangle}f^\sharp(s)$. Applying \eqref{eq:iterated_ell} with \(n=2m\) to these factors,
we obtain
\[
\tau_\theta(W_1\cdots W_{m})
=\int_{\R^{(2m-1)d}}
\Theta_\theta(\mathbf{u})\,\prod_{j=1}^{2m} g_j(\ell_j(\mathbf u))\,\mathrm{d}\mathbf{u},
\]
where
\[
g_j(s)
=
e^{2\pi i\chi_{\mathcal A}(j)\langle h,s\rangle}
g_j^{(0)}(s),
\qquad
g_j^{(0)}\in\{f,f^\sharp\}.
\]
Then
\[
\prod_{j=1}^{2m} g_j(\ell_j(\mathbf u))
=
e^{2\pi i\langle h,L(\mathbf u)\rangle}
\prod_{j=1}^{2m} g_j^{(0)}(\ell_j(\mathbf u)),
\]
where $L(\mathbf{u})=\sum_{j\in\mathcal{A}}\ell_j(\mathbf{u})$ with
$\mathcal{A}=\{j:X_j\in\{a,a^*\}\}$.
Moreover,
\[
\begin{aligned}
\int_{\R^{(2m-1)d}}
\prod_{j=1}^{2m}
\bigl|
g_j^{(0)}(\ell_j(\mathbf u))
\bigr|
\,\mathrm d\mathbf u
\leq
\|g_1^{(0)}\|_\infty
\prod_{j=2}^{2m}
\|g_j^{(0)}\|_1<\infty.
\end{aligned}
\]
Since $|\Theta_\theta|=1$, the function
\(
F(\mathbf u):=\Theta_\theta(\mathbf u)
\prod_{j=1}^{2m}
g_j^{(0)}(\ell_j(\mathbf u))\in L_1(\R^{(2m-1)d})\,
\)
is independent of $h$ and 
\[
\tau_\theta(W_1\cdots W_{m})
=\int_{\R^{(2m-1)d}}
e^{2\pi i\langle h,L(\mathbf u)\rangle}\,F(\mathbf u)\,\mathrm{d}\mathbf{u},
\]

Taking \(n=2m\) in \eqref{eq:ell}, we obtain
\[
L(\mathbf u)=\sum_{r=1}^{2m-1}c_r u_r,
\qquad
c_r=-\chi_{\mathcal A}(1)+\chi_{\mathcal A}(r+1),
\quad r=1,\ldots,2m-1,
\]
where $\chi_{\mathcal A}$ denotes the indicator function of $\mathcal A\subset\{1,\ldots,2m\}$. 
Regarding
\(\mathbf u=(u_1,\ldots,u_{2m-1})\) as a block column vector in
\(\R^{(2m-1)d}\), write
\[
L(\mathbf u)=\mathbf A_{\mathcal A}\mathbf u,
\qquad
\mathbf A_{\mathcal A}
=
\begin{pmatrix}
c_1I_d & c_2I_d & \cdots & c_{2m-1}I_d
\end{pmatrix}
\in\R^{d\times(2m-1)d},
\]
where \(I_d\) denotes the \(d\times d\) identity matrix.
Since the term under consideration contains at least one factor $a^*b$ or $b^*a$, the set $\mathcal A$ is neither empty nor
$\{1,\ldots,2m\}$. Hence \(c_r\neq0\) for some \(r\), and
\[
\mathbf A_{\mathcal A}\mathbf A_{\mathcal A}^T
=
\left(\sum_{r=1}^{2m-1}c_r^2\right)I_d
\]
is positive definite. Consequently,
\(\operatorname{rank}(\mathbf A_{\mathcal A})=d\), and \(L\) is
surjective.  Applying Lemma~\ref{lem:matrix-RL}, we obtain
\[
\tau_\theta(W_1\cdots W_m)\to 0
\qquad\text{as } \quad |h|\to\infty .
\]

\medskip
\textit{Mixed diagonal terms vanish.}
A mixed diagonal term has
\(W_i\in\{a^*a,b^*b\}\) for every \(i\) and contains at least one
factor of each type.  Let $W_i=U_{\theta}(g_{W_i})$. By \eqref{eq:respent} and \eqref{eq:theta-conv},
\[
g_{a^*a}(s)=e^{2\pi i\langle h,s\rangle}\varphi(s)
\quad\text{and}\quad
 g_{b^*b}(s)=\varphi(s),
 \]
where $\varphi=f^\sharp*_\theta f\in\mcs(\R^d)$. Applying \eqref{eq:iterated_ell} with \(n=m\) and
\(g_j=g_{W_j}\), we obtain 
\[
\tau_\theta(W_1\cdots W_m)
=\int_{\R^{(m-1)d}}e^{2\pi i\langle h,\,L'(\mathbf{u})\rangle}\,
\Theta_\theta(\mathbf{u})\,
\prod_{j=1}^{m}\varphi(\ell_j(\mathbf{u}))\,\mathrm{d}\mathbf{u},
\]
where $L'(\mathbf{u})=\sum_{j\in\mathcal{A}'}\ell_j(\mathbf{u})$ with
$\mathcal{A}'=\{j:W_j=a^*a\}$. 
Since
\[
\int_{\R^{(m-1)d}}
\prod_{j=1}^m
\bigl|\varphi(\ell_j(\mathbf u))\bigr|
\,\mathrm d\mathbf u
\leq
\|\varphi\|_\infty
\|\varphi\|_1^{m-1}
<\infty
\]
and \(|\Theta_\theta|=1\), the function
\[
\mathbf u\longmapsto
\Theta_\theta(\mathbf u)
\prod_{j=1}^{m}\varphi(\ell_j(\mathbf u))
\]
belongs to \(L_1(\R^{(m-1)d})\). Moreover, since the term is mixed
diagonal,
\[
\varnothing
\subsetneq
\mathcal A'
\subsetneq
\{1,\ldots,m\}.
\]
The coefficient argument used above therefore shows that \(L'\) is
surjective. Lemma~\ref{lem:matrix-RL} now yields
\[
\tau_\theta(W_1\cdots W_m)\longrightarrow0
\qquad (|h|\to\infty).
\]

\medskip
\textit{Pure diagonal terms.}
For the two pure diagonal terms, since $T_h$ is a trace-preserving
$*$-automorphism,
\[
\begin{aligned}
\tau_\theta((a^*a)^m)
&=
\tau_\theta\!\left(
T_h\bigl((\x^*\x)^m\bigr)
\right)
=
\tau_\theta((\x^*\x)^m)
=
\|\x\|_{2m}^{2m},\\
\tau_\theta((b^*b)^m)
&=
\tau_\theta((\x^*\x)^m)
=
\|\x\|_{2m}^{2m}.
\end{aligned}
\]
Therefore
\[
\lim_{|h|\to\infty}\|\Delta_h\x\|_{2m}^{2m}
=\tau_\theta((a^*a)^m)+\tau_\theta((b^*b)^m)=2\,\|\x\|_{2m}^{2m}.
\]

\medskip
\textit{Step 3.} \textit{Extension to  $2\le p<\infty$.}
Since $\x\in\mcs(\R^d_\theta)\subset L_\infty(\R^d_\theta)$ and \(
\|\Delta_h\x\|_\infty
\leq
2\|\x\|_\infty
\), we have
\[
0\le|\Delta_h\x|^2\le R\cdot\mathbf{1}, \qquad  0\le|\x|^2\le R\cdot\mathbf{1},\qquad R:=4\|\x\|_\infty^2.
\]
Consequently, both operators have spectra contained in the common
compact interval $[0,R]$.
Let \(E_h\) and \(E\) denote the projection-valued spectral measures
of \(|\Delta_h\x|^2\) and \(|\x|^2\), respectively, and write
\(\mathcal B([0,R])\) for the Borel \(\sigma\)-algebra of \([0,R]\). Since
\(\tau_\theta\) is normal,
\begin{align}\label{eq:spectral-measure}
\mu_h(B):=\tau_\theta(E_h(B)),
\qquad
\mu(B):=\tau_\theta(E(B)),
\qquad
B\in\mathcal B([0,R])
\end{align}
define positive Borel measures on \([0,R]\), not necessarily finite.
By the spectral integral formula
\cite[Remark~1.1.6]{PX03}, for every bounded nonnegative Borel
function \(\varphi\) on \([0,R]\),
\[
\tau_\theta(\varphi(|\Delta_h\x|^2))
=
\int_0^R\varphi(\lambda)\,\mathrm d\mu_h(\lambda),
\qquad
\tau_\theta(\varphi(|\x|^2))
=
\int_0^R\varphi(\lambda)\,\mathrm d\mu(\lambda).
\]

Since
\[
E_h([0,R])=E([0,R])=\mathbf1
\qquad\text{and}\qquad
\tau_\theta(\mathbf1)=\infty,
\]
both \(\mu_h\) and \(\mu\) have infinite total mass. We therefore
consider the weighted measures
\[
\nu_h(B)
:=
\int_B\lambda\,\mathrm d\mu_h(\lambda),
\qquad
\nu(B)
:=
\int_B\lambda\,\mathrm d\mu(\lambda),
\qquad
B\in\mathcal B([0,R]).
\]
The spectral functional calculus and the spectral integration formula
give
\[
\nu_h(B)
=
\tau_\theta\bigl(|\Delta_h\x|^2 E_h(B)\bigr),
\qquad
\nu(B)
=
\tau_\theta\bigl(|\x|^2E(B)\bigr).
\]
In particular, \(\nu_h\) and \(\nu\) are finite, and
\[\nu_h([0,R])
=\tau_\theta(|\Delta_h\x|^2)=\|\Delta_h\x\|_2^2\le 4\|\x\|_2^2, \qquad
\nu([0,R])
=
\tau_\theta(|\x|^2)
=
\|\x\|_2^2.\]

Step~2 gives, for every integer $n\ge 1$,
\begin{equation}\label{eq:moments}
\int_0^R\lambda^{n-1}\,\mathrm{d}\nu_h(\lambda)
=\|\Delta_h\x\|_{2n}^{2n}
\longrightarrow 2\,\|\x\|_{2n}^{2n}
=2\int_0^R\lambda^{n-1}\,\mathrm{d}\nu(\lambda),\qquad |h|\to\infty.
\end{equation}
By linearity, \eqref{eq:moments} implies that
\[
\int_0^R P(\lambda)\,\mathrm d\nu_h(\lambda)
\longrightarrow
2\int_0^R P(\lambda)\,\mathrm d\nu(\lambda)
\qquad (|h|\to\infty)
\]
for every polynomial \(P\). Since
\[
\sup_h\nu_h([0,R])<\infty,
\]
the Weierstrass approximation theorem therefore extends the
convergence from polynomials to all functions in \(C([0,R])\).
Thus, for every \(\psi\in C([0,R])\),
\begin{equation}\label{eq:weak}
\int_0^R\psi(\lambda)\,\mathrm d\nu_h(\lambda)
\longrightarrow
2\int_0^R\psi(\lambda)\,\mathrm d\nu(\lambda)
\qquad (|h|\to\infty).
\end{equation}
Applying \eqref{eq:weak} to $\psi(\lambda)=\lambda^{p/2-1}$ with \(\psi\equiv1\) when \(p=2\),
we obtain
\[
\|\Delta_h\x\|_p^p
=\int_0^R\lambda^{p/2}\,\mathrm{d}\mu_h(\lambda)
=\int_0^R\lambda^{p/2-1}\,\mathrm{d}\nu_h(\lambda)
\longrightarrow 2\int_0^R\lambda^{p/2-1}\,\mathrm{d}\nu(\lambda)
=2\,\|\x\|_p^p.
\]
Therefore, for every $2\le p <\infty$ and $\x\in\mcs(\R^d_\theta)$,
\begin{equation}\label{eq:lower_all}
\lim_{|h|\to\infty}\|\Delta_h\x\|_p^p=2\,\|\x\|_p^p.
\end{equation}

\medskip
\textit{Step 4.} \textit{Extension to $1\leq p<2$.}
The arguments of Steps~2--3 also apply to
\[
S_h\x:=T_h\x+\x.
\]
Indeed, the expansion of
\(((a+b)^*(a+b))^m\), where \(a=T_h\x\) and \(b=\x\), has the same
classification as before, with only the signs of the non-pure terms
changed. Lemma~\ref{lem:matrix-RL} again shows that all mixed and
mixed diagonal terms vanish as \(|h|\to\infty\), while the two pure
diagonal terms contribute \(2\|\x\|_{2m}^{2m}\). Repeating the
spectral-measure argument of Step~3 therefore gives
\[
\lim_{|h|\to\infty}
\|S_h\x\|_p^p
=
2\|\x\|_p^p,
\qquad
2\leq p<\infty.
\]

Fix \(1\leq p<2\). We first establish the required lower bound.
With \(\mu_h\) and \(\mu\) as in
\eqref{eq:spectral-measure}, define positive Borel measures on
\([0,R]\) by
\[
\begin{aligned}
\rho_h(B)
&:=
\int_B\lambda^{p/2}\,\mathrm d\mu_h(\lambda)
=
\tau_\theta\bigl(|\Delta_h\x|^pE_h(B)\bigr),\\
\rho(B)
&:=
\int_B\lambda^{p/2}\,\mathrm d\mu(\lambda)
=
\tau_\theta\bigl(|\x|^pE(B)\bigr),
\qquad
B\in\mathcal B([0,R]).
\end{aligned}
\]
These measures are finite, since
\[
\rho_h([0,R])
=
\|\Delta_h\x\|_p^p
\leq
2^p\|\x\|_p^p,
\qquad
\rho([0,R])
=
\|\x\|_p^p.
\]
Moreover, the weight \(\lambda^{p/2}\) vanishes at zero, and hence
\[
\rho_h(\{0\})=\rho(\{0\})=0.
\]
For every integer \(n\geq1\), we have \(2n+p>2\) and therefore,
\eqref{eq:lower_all} gives
\begin{equation}\label{eq:moment-converge-2}
\int_0^R\lambda^n\,\mathrm{d}\rho_h(\lambda)
=\|\Delta_h\x\|_{2n+p}^{2n+p}\longrightarrow 2\,\|\x\|_{2n+p}^{2n+p}=
2\int_0^R\lambda^n\,\mathrm d\rho(\lambda),
\end{equation}
Choose a sequence \((h_j)\) with \(|h_j|\to\infty\) such that
\[
\rho_{h_j}([0,R])
\longrightarrow
\liminf_{|h|\to\infty}\rho_h([0,R]).
\]
Since
\[
\sup_j\rho_{h_j}([0,R])
\leq
2^p\|\x\|_p^p<
\infty,
\]
the sequence \((\rho_{h_j})\) is bounded in
\(C([0,R])^*\). By the Banach--Alaoglu theorem and the separability
of \(C([0,R])\), after passing to a subsequence we may assume that
\[
\rho_{h_j}\rightharpoonup\gamma
\]
for some finite positive Borel measure \(\gamma\) on \([0,R]\).  Combining \eqref{eq:moment-converge-2} with the weak convergence of
\(\rho_{h_j}\), we obtain
\[
\int_0^R P(\lambda)\,\mathrm d\gamma(\lambda)
=
2\int_0^R P(\lambda)\,\mathrm d\rho(\lambda)
\]
for every polynomial \(P\) satisfying \(P(0)=0\). Such polynomials
are dense in
\(
\bigl\{
f\in C([0,R]):f(0)=0
\bigr\},
\)
and hence
\[
\int_0^R f(\lambda)\,\mathrm d\gamma(\lambda)
=
2\int_0^R f(\lambda)\,\mathrm d\rho(\lambda)
\]
for every \(f\in C([0,R])\) with \(f(0)=0\).

For \(N\in\N\), let
\(
f_N(\lambda):=\min\{N\lambda,1\}.
\)
Since \(f_N\uparrow \chi_{(0,R]}\), the monotone convergence
theorem gives
\[
\gamma((0,R])
=
2\rho((0,R])
=
2\|\x\|_p^p.
\]
Since the set \((0,R]\) is open relative to \([0,R]\) and $\rho_{h_j}(\{0\})=0$,  
the Portmanteau
theorem (see \cite[Corollary~8.2.4]{Bog07}) yields
\[
\liminf_{j\to\infty}\rho_{h_j}([0,R])
=\liminf_{j\to\infty}\rho_{h_j}((0,R])
\ge\gamma((0,R])=2\,\|\x\|_p^p.
\]
By the choice of \((h_j)\), we conclude that
\begin{equation}\label{eq:liminf_p}
\liminf_{|h|\to\infty}\|\Delta_h\x\|_p^p\ge 2\,\|\x\|_p^p.
\end{equation}
Repeating the argument with \(S_h\x\) in place of
\(\Delta_h\x\) gives
\begin{equation}\label{eq:liminf_S}
\liminf_{|h|\to\infty}\|S_h\x\|_p^p\ge 2\,\|\x\|_p^p.
\end{equation}

It remains to establish the matching upper bounds. Suppose first that \(1<p<2\), and let
\(p'=p/(p-1)\). The noncommutative Clarkson inequality \cite[Theorem~5.1\,(i)]{PX03} states that
\[
\|u+v\|_p^{p'}+\|u-v\|_p^{p'}
\leq
2\bigl(\|u\|_p^p+\|v\|_p^p\bigr)^{p'/p},
\quad u,v\in L_p(\mathcal{M}).
\]
Applying this inequality with \(u=T_h\x\) and \(v=\x\), we obtain
\[
\|S_h\x\|_p^{p'}+\|\Delta_h\x\|_p^{p'}
\le 2(2\|\x\|_p^p)^{p'/p}=2^{1+p'/p}\,\|\x\|_p^{p'}.
\]
Since \(t\mapsto t^{p'/p}\) is increasing on \([0,\infty)\),
\eqref{eq:liminf_p} and \eqref{eq:liminf_S} imply
\[
\begin{aligned}
\liminf_{|h|\to\infty}
\|\Delta_h\x\|_p^{p'}
&\geq
2^{p'/p}\|\x\|_p^{p'},\qquad
\liminf_{|h|\to\infty}
\|S_h\x\|_p^{p'}
&\geq
2^{p'/p}\|\x\|_p^{p'}.
\end{aligned}
\]
Consequently,
\[
\limsup_{|h|\to\infty}\|\Delta_h\x\|_p^{p'}
\le 2^{1+p'/p}\,\|\x\|_p^{p'}-\liminf_{|h|\to\infty}\|S_h\x\|_p^{p'}
\le 2^{p'/p}\,\|\x\|_p^{p'},\quad 1<p<2.
\]
For \(p=1\), the triangle inequality gives
\[
\limsup_{|h|\to\infty}
\|\Delta_h\x\|_1
\leq
2\|\x\|_1.
\]
Combining these upper bounds with \eqref{eq:liminf_p}, we obtain
\[
\lim_{|h|\to\infty}
\|\Delta_h\x\|_p^p
=
2\|\x\|_p^p,
\qquad
1\leq p<2.
\]
Together with Step~3, this proves \eqref{eq:lower_all} for every
\(1\leq p<\infty\).

\medskip
\textit{Step 5.} \textit{Completion of the proof.}
Steps~2--4 show that
\[
\lim_{|h|\to\infty}
\|\Delta_h\x\|_p
=
2^{1/p}\|\x\|_p,
\qquad
\x\in\mcs(\R_\theta^d),\quad 1\leq p<\infty.
\]
We now extend this identity to arbitrary
\(\x\in L_p(\R_\theta^d)\). Choose
\(\x_n\in\mcs(\R_\theta^d)\) such that
\(\x_n\to\x\) in \(L_p(\R_\theta^d)\). Since \(T_h\) is an
\(L_p\)-isometry,
\[
\bigl|
\|\Delta_h\x\|_p-\|\Delta_h\x_n\|_p
\bigr|
\leq
\|\Delta_h(\x-\x_n)\|_p
\leq
2\|\x-\x_n\|_p
\]
uniformly in \(h\).
Applying the preceding identity to \(\x_n\) gives
\[
\begin{aligned}
\limsup_{|h|\to\infty}
\left|
\|\Delta_h\x\|_p
-
2^{1/p}\|\x\|_p
\right|
&\leq
2\|\x-\x_n\|_p
+
2^{1/p}
\bigl|
\|\x_n\|_p-\|\x\|_p
\bigr|\\
&\leq
\bigl(2+2^{1/p}\bigr)
\|\x-\x_n\|_p.
\end{aligned}
\]
Letting \(n\to\infty\), we obtain
\begin{equation}\label{eq:sharp-translation-asymptotic}
\lim_{|h|\to\infty}
\|\Delta_h\x\|_{L_p(\R_\theta^d)}
=
2^{1/p}\|\x\|_{L_p(\R_\theta^d)},
\qquad
\x\in L_p(\R_\theta^d),\quad 1\leq p<\infty.
\end{equation}

By the definition of the modulus of smoothness and
\eqref{eq:sharp-translation-asymptotic},
\[
\omega_p^1(\x,\infty)
\geq
\lim_{|h|\to\infty}\|\Delta_h\x\|_p
=
2^{1/p}\|\x\|_p.
\]
Combining this with the upper bound established at the beginning of
the proof yields
\[
\omega_p^1(\x,\infty)
\approx
\|\x\|_p.
\]
The reduction in Step~1 now completes the proof.
\end{proof}

\begin{remark}\label{rem:sharp-asymptotic}
The exact large-translation limit
\eqref{eq:sharp-translation-asymptotic} has a classical counterpart
when \(\theta=0\). It follows for compactly supported functions from
the disjointness of sufficiently distant translates and then extends
to \(L_p(\R^d)\) by density. This support-separation argument has no direct analogue on
\(\R_\theta^d\), since general noncommutative \(L_p\)-elements are
operators rather than functions with spatial supports, and large
translations do not in general produce orthogonal support
projections.
In its place, the proof above uses the
Riemann--Lebesgue decay of the mixed tracial terms, followed by
weighted spectral-measure and density arguments.

No analogous limit holds in general on the quantum torus
\(\mathbb T_\theta^d\). Indeed,
\[
T_{h+n}=T_h,
\qquad
h\in\R^d,\quad n\in\Z^d,
\]
and, for every nonzero Fourier unitary \(U^m\),
\(m\in\Z^d\setminus\{0\}\),
\[
\|\Delta_hU^m\|_p
=
\left|
e^{2\pi i\langle h,m\rangle}-1
\right|,
\]
which has no limit as \(|h|\to\infty\). The mean-zero Poincar\'e
inequality permits the endpoint argument on \(\mathbb T_\theta^d\),
but it does not yield a large-translation limit.
\end{remark}

We now use Lemma~\ref{lem:MS-normal-condition} to verify the
\(X_0\)-normality condition and apply Milman's theorem to
\(\vec X_{\mathrm{hom}}\) to complete the proof of the MS-type limit.

{\textit{\textbf {Proof of Theorem~\ref{th1.3}.}}}   
For the couple \(\vec X_{\mathrm{hom}}\),
Corollary~\ref{cor:homogeneous-K-functional} and
Lemma~\ref{lem:MS-normal-condition} give, for every
\(\x\in X_0=L_p(\R_\theta^d)\),
\[
\begin{aligned}
\|\x\|_{\widetilde X_0}
&=
\lim_{\varepsilon\to\infty}
K(\x,\varepsilon^k;\vec X_{\mathrm{hom}})
&\approx
\lim_{\varepsilon\to\infty}
\omega_p^k(\x,\varepsilon)
\approx
\|\x\|_{L_p(\R_\theta^d)}.
\end{aligned}
\]
Thus the \(X_0\)-normality condition in \eqref{eq:normal} is
satisfied.

Now let
\(\x\in B_{p,q}^{\alpha_0}(\R_\theta^d)\) for some
\(\alpha_0\in(0,k)\), and set \(\eta_0=\alpha_0/k\). By
\eqref{eq:homogeneous-interpolation-norm},
\[
\|\x\|_{(\vec X_{\mathrm{hom}})_{\eta_0,q}}
\approx
k^{1/q}
|\x|_{B_{p,q}^{\alpha_0,k}(\R_\theta^d)}
<
\infty.
\]
Hence
\[
\x\in
X_0\cap
\bigcup_{0<\eta<1}
(\vec X_{\mathrm{hom}})_{\eta,q},
\]
and Theorem~\ref{thm:Milman}\textnormal{(ii)} yields
\begin{equation}\label{eq:Milman-MS-application}
\lim_{\eta\downarrow0}
\bigl(\eta(1-\eta)\bigr)^{1/q}
\|\x\|_{(\vec X_{\mathrm{hom}})_{\eta,q}}
=
q^{-1/q}\|\x\|_{\widetilde X_0}
\approx
q^{-1/q}\|\x\|_{L_p(\R_\theta^d)}.
\end{equation}
On the other hand, setting \(\alpha=k\eta\) in
\eqref{eq:homogeneous-interpolation-norm}, we obtain
\begin{equation}\label{eq:inhom-uniform-equivalence}
\bigl(\eta(1-\eta)\bigr)^{1/q}
\|\x\|_{(\vec X_{\mathrm{hom}})_{\eta,q}}
\approx
\left(
\alpha\left(1-\frac{\alpha}{k}\right)
\right)^{1/q}
|\x|_{B_{p,q}^{\alpha,k}(\R_\theta^d)},
\end{equation}
with constants independent of \(\eta\). 
Moreover,
\[
\alpha^{1/q}
\|\x\|_{L_p(\R_\theta^d)}
\longrightarrow0
\qquad (\alpha\downarrow0).
\]
It follows from \eqref{eq:exact-MS-modulus-limit} and \eqref{eq:difference-Besov-norm} that
\[
\lim_{\alpha\downarrow0}
\alpha^{1/q}
|\x|_{B_{p,q}^{\alpha,k}(\R_\theta^d)}
\]
exists.
Since \(1-\alpha/k\to1\), we may pass to the limit in \eqref{eq:inhom-uniform-equivalence} and combine it with
\eqref{eq:Milman-MS-application} to obtain
\[
\lim_{\alpha\downarrow0}
\alpha^{1/q}
\|\x\|_{B_{p,q}^{\alpha,k}(\R_\theta^d)}
=
\lim_{\alpha\downarrow0}
\alpha^{1/q}
|\x|_{B_{p,q}^{\alpha,k}(\R_\theta^d)}
\approx
q^{-1/q}
\|\x\|_{L_p(\R_\theta^d)}.
\]
This proves \eqref{eq:MS-quantum-main}.
\vspace{0.3cm}

We conclude with the corresponding endpoint limits for
\(\vec X_{\mathrm{inhom}}\). This couple is also normal in the sense
of \eqref{eq:normal}.  Indeed, Theorem~\ref{thm:K-functional},
Theorem~\ref{th3.7}, and
Corollary~\ref{cor:top-order-Sobolev-characterization} give the
\(X_1\)-normality condition. On the other hand, the contractive
embedding
\[
W_p^k(\R_\theta^d)\hookrightarrow L_p(\R_\theta^d)
\]
implies that
\[
K(\x,t;\vec X_{\mathrm{inhom}})
=
\|\x\|_{L_p(\R_\theta^d)},
\qquad
\x\in L_p(\R_\theta^d),\quad t\geq1,
\]
and hence gives the \(X_0\)-normality condition. Thus
Theorem~\ref{thm:Milman} applies to
\(\vec X_{\mathrm{inhom}}\).

To express the resulting interpolation scale in terms of difference
norms, define, for \(0<\alpha<k\),
\begin{equation}\label{eq:inhomogeneous-normalized-Besov-norm}
\|\x\|_{\mathbf B_{p,q}^{\alpha,k}(\R_\theta^d)}
:=
\left(
\frac{k}{q\alpha(k-\alpha)}
\right)^{1/q}
\|\x\|_{L_p(\R_\theta^d)}
+
|\x|_{B_{p,q}^{\alpha,k}(\R_\theta^d)}.
\end{equation}
With \(\eta=\alpha/k\),
\eqref{eq:inhomogeneous-interpolation-norm} gives
\[
\|\x\|_{(\vec X_{\mathrm{inhom}})_{\eta,q}}
\approx
k^{1/q}
\|\x\|_{\mathbf B_{p,q}^{\alpha,k}(\R_\theta^d)},
\]
where the equivalence constants are independent of \(\alpha\).
The endpoint limits for this difference-norm realization are recorded
in the following corollary.

\begin{corollary}
\label{cor:inhomogeneous-endpoint-limits}
Let \(1\leq p,q<\infty\) and \(k\in\N\).

\begin{enumerate}[label=\textnormal{(\roman*)}]
\item
For every \(\x\in W_p^k(\R_\theta^d)\),
\[
\begin{aligned}
\lim_{\alpha\uparrow k}
(k-\alpha)^{1/q}
\|\x\|_{\mathbf B_{p,q}^{\alpha,k}(\R_\theta^d)}
&=
q^{-1/q}
\left(
\|\x\|_{L_p(\R_\theta^d)}
+
\lim_{\varepsilon\downarrow0}
\frac{\omega_p^k(\x,\varepsilon)}{\varepsilon^k}
\right)
\nonumber\\
&\approx
q^{-1/q}
\|\x\|_{W_p^k(\R_\theta^d)}.
\end{aligned}
\]

\item
If
\(\x\in B_{p,q}^{\alpha_0}(\R_\theta^d)\) for some
\(\alpha_0\in(0,k)\), then
\[
\begin{aligned}
\lim_{\alpha\downarrow0}
\alpha^{1/q}
\|\x\|_{\mathbf B_{p,q}^{\alpha,k}(\R_\theta^d)}
&=
q^{-1/q}
\left(
\|\x\|_{L_p(\R_\theta^d)}
+
\lim_{\varepsilon\to\infty}
\omega_p^k(\x,\varepsilon)
\right)
\nonumber\\
&\approx
q^{-1/q}
\|\x\|_{L_p(\R_\theta^d)}.
\label{eq:inhomogeneous-MS-main}
\end{aligned}
\]
\end{enumerate}
\end{corollary}

\begin{proof}
By definition,
\[
\begin{aligned}
\lim_{\alpha\uparrow k}
&(k-\alpha)^{1/q}
\left(
\frac{k}{q\alpha(k-\alpha)}
\right)^{1/q}
=
q^{-1/q}\\
&\lim_{\alpha\downarrow0}
\alpha^{1/q}
\left(
\frac{k}{q\alpha(k-\alpha)}
\right)^{1/q}
=
q^{-1/q}.
\end{aligned}
\]
The asserted identities therefore follow from
\eqref{eq:exact-BBM-modulus-limit} and
\eqref{eq:exact-MS-modulus-limit}. The equivalence in
\textnormal{(i)} follows from Theorem~\ref{th3.7} and
\eqref{eq:in-hom-norm-equivalent}, while that in
\textnormal{(ii)} follows from
Lemma~\ref{lem:MS-normal-condition}.
\end{proof}

\section{Appendix: Exact endpoint asymptotics for difference Besov norms}\label{app:exact-endpoint-asymptotics}
The exact identities
\eqref{eq:exact-BBM-modulus-limit} and
\eqref{eq:exact-MS-modulus-limit} are also of independent interest.
Although Milman's theorem gives the endpoint continuity of the
interpolation scale, the uniform equivalence between the interpolation
and difference norms (see \eqref{eq:hom-uniform-equivalence} and \eqref{eq:inhom-uniform-equivalence}) does not imply either the existence or the exact
values of the corresponding normalized difference-norm limits.
We therefore prove the two identities directly using the following elementary limit
lemma.

\begin{lemma}
\label{lem:endpoint-concentration}
Let $1\leq q<\infty$.
\begin{enumerate}[label=\textnormal{(\roman*)}]
\item
Let $F:(0,1]\to[0,\infty)$ be bounded and measurable. If
\;\(
\lim_{t\downarrow0}F(t)=F_0,
\)
then
\begin{equation}\label{eq:concentration-at-zero}
\lim_{\delta\downarrow0}
\delta
\int_0^1
t^{\delta q}F(t)^q
\,\frac{\mathrm dt}{t}
=
\frac{(F_0)^q}{q}.
\end{equation}

\item
Let $F:[1,\infty)\to[0,\infty)$ be bounded and measurable. If
\;\(
\lim_{t\to\infty}F(t)=F_\infty,
\)
then
\begin{equation}\label{eq:concentration-at-infinity}
\lim_{\delta\downarrow0}
\delta
\int_1^\infty
t^{-\delta q}F(t)^q
\,\frac{\mathrm dt}{t}
=
\frac{(F_\infty)^q}{q}.
\end{equation}
\end{enumerate}
\end{lemma}

\begin{proof}
For part~\textnormal{(i)}, the substitutions
\(t=e^{-u}\) and \(v=\delta qu\) give
\[
\delta
\int_0^1
t^{\delta q}F(t)^q
\,\frac{\mathrm dt}{t}
=
\frac1q
\int_0^\infty
e^{-v}
F\bigl(e^{-v/(\delta q)}\bigr)^q
\,\mathrm dv.
\]
For every \(v>0\), the integrand converges to
\(e^{-v}F_0^q\) as \(\delta\downarrow0\), and it is dominated by an
integrable multiple of \(e^{-v}\). Hence
\eqref{eq:concentration-at-zero} follows from the dominated
convergence theorem.

For part~\textnormal{(ii)}, the substitutions
\(t=e^u\) and \(v=\delta qu\) yield
\[
\delta
\int_1^\infty
t^{-\delta q}F(t)^q
\,\frac{\mathrm dt}{t}
=
\frac1q
\int_0^\infty
e^{-v}
F\bigl(e^{v/(\delta q)}\bigr)^q
\,\mathrm dv,
\]
and the same argument proves
\eqref{eq:concentration-at-infinity}.
\end{proof}

{\textit{\textbf {Proof of \eqref{eq:exact-BBM-modulus-limit} and
\eqref{eq:exact-MS-modulus-limit}.}}}
We first consider the upper endpoint. Let
\(\x\in W_p^k(\R_\theta^d)\) and set
\[
F(\varepsilon)
:=
\varepsilon^{-k}\omega_p^k(\x,\varepsilon),
\qquad
F_0
:=
\lim_{\varepsilon\downarrow0}F(\varepsilon).
\]
By Theorem~\ref{th3.7}, \(F\) is bounded on \((0,1]\) and \(F_0\)
is finite. Moreover,
\[
\begin{aligned}
(k-\alpha)
|\x|_{B_{p,q}^{\alpha,k}(\R_\theta^d)}^q
&=
(k-\alpha)
\int_0^1
\varepsilon^{(k-\alpha)q}
F(\varepsilon)^q
\,\frac{\mathrm d\varepsilon}{\varepsilon}
\\
&\quad+
(k-\alpha)
\int_1^\infty
\varepsilon^{-\alpha q}
\omega_p^k(\x,\varepsilon)^q
\,\frac{\mathrm d\varepsilon}{\varepsilon}.
\end{aligned}
\]
By Lemma~\ref{lem:endpoint-concentration}\textnormal{(i)}, applied
with \(\delta=k-\alpha\), the first term converges to \(F_0^q/q\)
as \(\alpha\uparrow k\), while the second tends to zero by
\eqref{eq:BBM-large-scale-bound}. Therefore,
\[
\lim_{\alpha\uparrow k}
(k-\alpha)^{1/q}
|\x|_{B_{p,q}^{\alpha,k}(\R_\theta^d)}
=
q^{-1/q}F_0.
\]
Together with
\[
(k-\alpha)^{1/q}
\|\x\|_{L_p(\R_\theta^d)}
\longrightarrow0
\qquad (\alpha\uparrow k),
\]
this gives \eqref{eq:exact-BBM-modulus-limit}.

For the lower endpoint, suppose that
\(\x\in B_{p,q}^{\alpha_0}(\R_\theta^d)\) for some
\(\alpha_0\in(0,k)\). For \(0<\alpha<\alpha_0\), we have
\[
\begin{aligned}
\alpha
|\x|_{B_{p,q}^{\alpha,k}(\R_\theta^d)}^q
&=
\alpha
\int_0^1
\varepsilon^{-\alpha q}
\omega_p^k(\x,\varepsilon)^q
\,\frac{\mathrm d\varepsilon}{\varepsilon}
\\
&\quad+
\alpha
\int_1^\infty
\varepsilon^{-\alpha q}
\omega_p^k(\x,\varepsilon)^q
\,\frac{\mathrm d\varepsilon}{\varepsilon}.
\end{aligned}
\]
As \(\alpha\downarrow0\), the first term tends to zero by
\eqref{eq:ms-small-scale}, while
Lemma~\ref{lem:endpoint-concentration}\textnormal{(ii)}, applied
with
\[
F(\varepsilon)=\omega_p^k(\x,\varepsilon),
\qquad
\delta=\alpha,
\]
shows that the second converges to
\[
\frac{\omega_p^k(\x,\infty)^q}{q}.
\]
Consequently,
\[
\lim_{\alpha\downarrow0}
\alpha^{1/q}
|\x|_{B_{p,q}^{\alpha,k}(\R_\theta^d)}
=
q^{-1/q}\omega_p^k(\x,\infty).
\]
Together with
\(
\alpha^{1/q}\|\x\|_{L_p(\R_\theta^d)}\to0
\)
as \(\alpha\downarrow0\), this gives
\eqref{eq:exact-MS-modulus-limit}.

\begin{remark}
The identities
\eqref{eq:exact-BBM-modulus-limit} and
\eqref{eq:exact-MS-modulus-limit} also provide direct proofs of
Theorems~\ref{th1.2} and~\ref{th1.3}, without recourse to Milman's
theorem. Indeed, Theorem~\ref{th3.7} identifies
\[
\lim_{\varepsilon\downarrow0}
\frac{\omega_p^k(\x,\varepsilon)}{\varepsilon^k}
\approx
|\x|_{\dot W_p^k(\R_\theta^d)},
\]
whereas Lemma~\ref{lem:MS-normal-condition} gives
\[
\omega_p^k(\x,\infty)
\approx
\|\x\|_{L_p(\R_\theta^d)}.
\]
Substitution into
\eqref{eq:exact-BBM-modulus-limit} and
\eqref{eq:exact-MS-modulus-limit} yields the two asserted endpoint
formulae.

The two approaches emphasize different aspects of the result. The
direct argument determines the exact asymptotic behavior of the
specific difference norms used in the statements, while the
interpolation argument places both limits within a common framework
through the endpoint continuity of the real interpolation scale.
\end{remark}

{\bf Acknowledgements:} Qingquan Deng is supported by NSFC (No. 12471092 and No. 12531005).
Guixiang Hong is supported by the National Natural Science
Foundation of China (grant nos.~12325105,
W2441002, and W2611005). Ping Li is supported by NSFC (Nos  12371136 and 12571145).


\begin{thebibliography}{99}

\bibitem{ACPS2020}
A.~Alberico, A.~Cianchi, L.~Pick, and L.~Slav\'{i}kov\'{a}, \textit{On the limit as $s\to 1^{-}$ of possibly non-separable fractional Orlicz--Sobolev spaces}, Atti Accad. Naz. Lincei Rend. Lincei Mat. Appl., \textbf{31} (2020), no.~4, 879--899.

\bibitem{ACPS2020b}
A.~Alberico, A.~Cianchi, L.~Pick, and L.~Slav\'{i}kov\'{a}, \textit{On the limit as $s\to 0^{+}$ of fractional Orlicz--Sobolev spaces}, J. Fourier Anal. Appl., \textbf{26} (2020), Paper No.~80, 19 pp.

\bibitem{BS88}
C.~Bennett and R.~Sharpley, \textit{Interpolation of Operators}, Pure and Applied Mathematics, vol.~129, Academic Press, Boston, 1988.

\bibitem{BJ76}
J.~Bergh and J.~L\"{o}fstr\"{o}m, \textit{Interpolation Spaces: An Introduction}, Grundlehren der mathematischen Wissenschaften, vol.~223, Springer-Verlag, Berlin, 1976.

\bibitem{Bog07}
V.~I.~Bogachev, \textit{Measure Theory}, Vol.~II, Springer-Verlag, Berlin, 2007.



\bibitem{BBM2000}
J.~Bourgain, H.~Br\'{e}zis, and P.~Mironescu, \textit{Lifting in Sobolev spaces}, J. Anal. Math., \textbf{80} (2000), 37--86.

\bibitem{BBM01}
J.~Bourgain, H.~Br\'{e}zis, and P.~Mironescu, \textit{Another look at Sobolev spaces}, in: Optimal Control and Partial Differential Equations, IOS, Amsterdam, 2001, pp.~439--455.

\bibitem{BBM2002}
J.~Bourgain, H.~Br\'{e}zis, and P.~Mironescu, \textit{Limiting embedding theorems for $W^{s,p}$ when $s\uparrow 1$ and applications}, J. Anal. Math., \textbf{87} (2002), 77--101.


\bibitem{BSY2023}
D.~Brazke, A.~Schikorra, and P.-L.~Yung, \textit{Bourgain--Brezis--Mironescu convergence via Triebel--Lizorkin spaces}, Calc. Var. Partial Differential Equations, \textbf{62} (2023), no.~2, Paper No.~41.

\bibitem{HB02}
H.~Br\'{e}zis, \textit{How to recognize constant functions. Connections with Sobolev spaces}, Uspekhi Mat. Nauk, \textbf{57} (2002), 59--74; English transl. Russian Math. Surveys, \textbf{57} (2002), 693--708.

\bibitem{HB16}
H.~Br\'{e}zis and H.-M.~Nguyen, \textit{The BBM formula revisited}, Atti Accad. Naz. Lincei Rend. Lincei Mat. Appl., \textbf{27} (2016), 515--533.



\bibitem{CGRS2014}
A.~L.~Carey, V.~Gayral, A.~Rennie, and F.~Sukochev, \textit{Index theory for locally compact noncommutative geometries}, Mem. Amer. Math. Soc., \textbf{231} (2014), no.~1085, vi+130 pp.

\bibitem{CH2026}
D.~Chen and G.~Hong, \textit{The nonlinear estimates on quantum Besov space}, preprint (2026), arXiv:2601.11934.

\bibitem{CHWW2025}
D.~Chen, G.~Hong, L.~Wang, and W.~Wang, \textit{Navier--Stokes equations on quantum Euclidean spaces}, preprint (2025), arXiv:2511.04318.

\bibitem{CXY2013}
Z. Chen, Q. Xu and Z. Yin, \textit{Harmonic analysis on quantum tori}, Comm. Math. Phys., \textbf{322} (2013), 755--805. 

\bibitem{CDKANP2023}
V.~Crismale, L.~De Luca, A.~Kubin, A.~Ninno, and M.~Ponsiglione, \textit{The variational approach to $s$-fractional heat flows and the limit cases $s\to 0^{+}$ and $s\to 1^{-}$}, J. Funct. Anal., \textbf{284} (2023), Paper No.~109851, 38 pp.

\bibitem{DGPYYZ2022}
F.~Dai, L.~Grafakos, Z.~Pan, D.~Yang, W.~Yuan, and Y.~Zhang, \textit{The Bourgain--Br\'{e}zis--Mironescu formula on ball Banach function spaces}, Math. Ann., \textbf{388} (2024), no.~2, 1691--1768.

\bibitem{DLY2024}
\'{O}.~Dom\'{i}nguez, Y.~Li, S.~Tikhonov, D.~Yang, and W.~Yuan, \textit{A unified approach to self-improving property via $K$-functionals}, Calc. Var. Partial Differential Equations, \textbf{63} (2024), no.~9, Paper No.~231.

\bibitem{DM2023}
\'{O}.~Dom\'{i}nguez and M.~Milman, \textit{Bourgain--Brezis--Mironescu--Maz'ya--Shaposhnikova limit formulae for fractional Sobolev spaces via interpolation and extrapolation}, Calc. Var. Partial Differential Equations, \textbf{62} (2023), no.~2, Paper No.~43, 37 pp.


\bibitem{FK96}
T.~Fack and H.~Kosaki, \textit{Generalized $s$-numbers of $\tau$-measurable operators}, Pacific J. Math., \textbf{123} (1986), 269--300.


\bibitem{FHW2025}
 Z. Fan, G. Hong and L. Wang, \textit{Sharp endpoint $L_p$ estimates of quantum Schr\"{o}dinger groups}, Comm. Math. Phys., \textbf{406} (2025), 55 pp.

\bibitem{FFG2023}
G.~F.~Foghem Gounoue, \textit{A remake of Bourgain--Br\'{e}zis--Mironescu characterization of Sobolev spaces}, Partial Differ. Equ. Appl., \textbf{4} (2023), no.~2, Paper No.~16, 36 pp.

\bibitem{Gag58}
E.~Gagliardo, \textit{Propriet\`a di alcune classi di funzioni in pi\`u variabili}, Ric. Mat., \textbf{7} (1958), 102--137.


\bibitem{GaJM2022}
L.~Gao, M.~Junge, and E.~McDonald, \textit{Quantum Euclidean spaces with noncommutative derivatives}, J. Noncommut. Geom., \textbf{16} (2022), 153--213.

\bibitem{GGBI04}
V.~Gayral, J.~M.~Gracia-Bond\'{i}a, B.~Iochum, T.~Sch\"{u}cker, and J.~C.~V\'{a}rilly, \textit{Moyal planes are spectral triples}, Comm. Math. Phys., \textbf{246} (2004), 569--623.

\bibitem{GJP21}
A.~M.~Gonz\'{a}lez-P\'{e}rez, M.~Junge, and J.~Parcet, \textit{Singular integrals in quantum Euclidean spaces}, Mem. Amer. Math. Soc., \textbf{272} (2021), no.~1334, xiii+90 pp.



\bibitem{Groe1946}
H.~J.~Groenewold, \textit{On the principles of elementary quantum mechanics}, Physica, \textbf{12} (1946), 405--460.

\bibitem{GY2021}
Q.~Gu and P.-L.~Yung, \textit{A new formula for the $L^p$ norm}, J. Funct. Anal., \textbf{281} (2021), no.~4, Paper No.~109075, 19 pp.


\bibitem{HPXZ2024}
B.-X.~Han, A.~Pinamonti, Z.~Xu, and K.~Zambanini, \textit{Maz'ya--Shaposhnikova meet Bishop--Gromov}, Potential Anal., \textbf{63} (2025), no.~2, 513--529.

\bibitem{HLW23}
G.~Hong, X.~Lai, and L.~Wang, \textit{Fourier restriction estimates on quantum Euclidean spaces}, Adv. Math., \textbf{430} (2023), Paper No.~109232.


\bibitem{JS76}
H.~Johnen and K.~Scherer, \textit{On the equivalence of the $K$-functional and moduli of continuity and some applications}, in: Constructive Theory of Functions of Several Variables, Lecture Notes in Math., vol.~571, Springer, Berlin, 1976, pp.~119--140.

\bibitem{Junge2002}
M.~Junge, \textit{Doob's inequality for non-commutative martingales}, J. Reine Angew. Math., \textbf{549} (2002), 149--190.

\bibitem{JMPX2021}
M.~Junge, T.~Mei, J.~Parcet, and R.~Xia, \textit{Algebraic Calder\'{o}n--Zygmund theory}, Adv. Math., \textbf{376} (2021), Paper No.~107443, 72 pp.

\bibitem{JX2007}
M.~Junge and Q.~Xu, \textit{Non-commutative maximal ergodic theorems}, J. Amer. Math. Soc., \textbf{20} (2007), 385--439.

\bibitem{KMX05}
G.~E.~Karadzhov, M.~Milman, and J.~Xiao, \textit{Limits of higher-order Besov spaces and sharp reiteration theorems}, J. Funct. Anal., \textbf{221} (2005), 323--339.

\bibitem{KIJ2025}
M.~Kijaczko, \textit{Asymptotics of weighted Gagliardo seminorms}, Ann. Mat. Pura Appl., \textbf{204} (2025), no.~2, 1715--1728.

\bibitem{MK2023}
M.~Kim, \textit{Bourgain, Br\'{e}zis and Mironescu theorem for fractional Sobolev spaces with variable exponents}, Ann. Mat. Pura Appl. (4), \textbf{202} (2023), no.~6, 2653--2664.

\bibitem{KL05}
V.~I.~Kolyada and A.~K.~Lerner, \textit{On limiting embeddings of Besov spaces}, Studia Math., \textbf{171} (2005), 1--13.


\bibitem{LSZ20}
G.~Levitina, F.~Sukochev, and D.~Zanin, \textit{Cwikel estimates revisited}, Proc. Lond. Math. Soc. (3), \textbf{120} (2020), 265--304.


\bibitem{LMSZ2023}
S.~Lord, E.~McDonald, F.~Sukochev, and D.~Zanin, \textit{Singular Traces. Vol.~2: Trace Formulas}, 2nd ed., De Gruyter Studies in Mathematics, vol.~46/2, De Gruyter, Berlin, 2023.

\bibitem{Lu2014}
M.~Ludwig, \textit{Anisotropic fractional Sobolev norms}, Adv. Math., \textbf{252} (2014), 150--157.

\bibitem{MSV2022}
A.~Maione, A.~Salort, and E.~Vecchi, \textit{Maz'ya--Shaposhnikova formula in magnetic fractional Orlicz--Sobolev spaces}, Asymptot. Anal., \textbf{126} (2022), 201--214.

\bibitem{MS02}
V.~Maz'ya and T.~Shaposhnikova, \textit{On the Bourgain, Br\'{e}zis, and Mironescu theorem concerning limiting embeddings of fractional Sobolev spaces}, J. Funct. Anal., \textbf{195} (2002), 230--238.

\bibitem{ME24}
E.~McDonald, \textit{Nonlinear partial differential equations on noncommutative Euclidean spaces}, J. Evol. Equ., \textbf{24} (2024), no.~1, Paper No.~3, 58 pp.

\bibitem{MSX2019}
E.~McDonald, F.~Sukochev, and X.~Xiong, \textit{Quantum differentiability on quantum tori}, Comm. Math. Phys., \textbf{371} (2019), no.~3, 1231--1260.

\bibitem{MSX20}
E.~McDonald, F.~Sukochev, and X.~Xiong, \textit{Quantum differentiability on noncommutative Euclidean spaces}, Comm. Math. Phys., \textbf{379} (2020), 491--542.


\bibitem{Mei2007}
T.~Mei, \textit{Operator valued Hardy spaces}, Mem. Amer. Math. Soc., \textbf{188} (2007), no.~881, vi+64 pp.

\bibitem{MM05}
M.~Milman, \textit{Notes on limits of Sobolev spaces and the continuity of interpolation scales}, Trans. Amer. Math. Soc., \textbf{357} (2005), 3425--3442.


\bibitem{Moyal1949}
J.~E.~Moyal, \textit{Quantum mechanics as a statistical theory}, Proc. Cambridge Philos. Soc., \textbf{45} (1949), 99--124.

\bibitem{NTYYZ2025}
E.~Nakai, M.~Tang, D.~Yang, W.~Yuan, and C.~Zhu, \textit{Maz'ya--Shaposhnikova representation of quasi-norms of ball quasi-Banach function spaces on spaces of homogeneous type with weak reverse doubling property}, preprint (2025), arXiv:2511.22960.


\bibitem{PYYZ2024}
Z.~Pan, D.~Yang, W.~Yuan, and Y.~Zhang, \textit{Gagliardo representation of norms of ball quasi-Banach function spaces}, J. Funct. Anal., \textbf{286} (2024), no.~2, Paper No.~110205.

\bibitem{PSV2017}
A.~Pinamonti, M.~Squassina, and E.~Vecchi, \textit{The Maz'ya--Shaposhnikova limit in the magnetic setting}, J. Math. Anal. Appl., \textbf{449} (2017), no.~2, 1152--1159.

\bibitem{PSV2019}
A.~Pinamonti, M.~Squassina, and E.~Vecchi, \textit{Magnetic BV-functions and the Bourgain--Br\'{e}zis--Mironescu formula}, Adv. Calc. Var., \textbf{12} (2019), no.~3, 225--252.

\bibitem{PX1997}
G.~Pisier and Q.~Xu, \textit{Non-commutative martingale inequalities}, Comm. Math. Phys., \textbf{189} (1997), 667--698.

\bibitem{PX03}
G.~Pisier and Q.~Xu, \textit{Non-commutative $L^p$-spaces}, in: Handbook of the Geometry of Banach Spaces, Vol.~2, North-Holland, Amsterdam, 2003, pp.~1459--1517.


\bibitem{Rie1993}
M.~A.~Rieffel, \textit{Deformation quantization for actions of $\R^d$}, Mem. Amer. Math. Soc., \textbf{106} (1993), no.~506, x+93 pp.

\bibitem{RST2026}
M.~Ruzhansky, S.~Shaimardan, and K.~Tulenov, \textit{$L^p$--$L^q$ boundedness of Fourier multipliers on quantum Euclidean spaces}, Rev. R. Acad. Cienc. Exactas F\'{\i}s. Nat. Ser. A Mat. RACSAM, \textbf{120} (2026), no.~1, Paper No.~8, 28 pp.

\bibitem{Slo58}
L.~N.~Slobodeckij, \textit{Generalized Sobolev spaces and their applications to boundary value problems of partial differential equations}, Leningrad. Gos. Ped. Inst. U\v{c}ep. Zap., \textbf{197} (1958), 54--112.





\bibitem{XXY18}
X.~Xiong, Q.~Xu, and Z.~Yin, \textit{Sobolev, Besov and Triebel--Lizorkin spaces on quantum tori}, Mem. Amer. Math. Soc., \textbf{252} (2018), no.~1203, vi+118 pp.

\bibitem{ZYY2024}
C.~Zhu, D.~Yang, and W.~Yuan, \textit{Bourgain--Br\'ezis--Mironescu-type characterization of inhomogeneous ball Banach Sobolev spaces on extension domains}, J. Geom. Anal., \textbf{34} (2024), Paper No.~295.

\end{thebibliography}
\end{document}